\documentclass[10pt]{amsart}

\usepackage{amssymb,amsmath,geometry,latexsym,graphics,tabularx,shapepar,enumerate,
hyperref,stmaryrd}

\usepackage[dvipsnames]{xcolor}
\usepackage{tikz}
\usepackage{tikz-cd}
\usetikzlibrary{matrix, calc, arrows}

\newcommand{\GL}{\operatorname{GL}}

\newcommand{\ZZ}{\mathbb{Z}}
\newcommand{\FF}{\mathbb{F}}
\newcommand{\CC}{\mathbb{C}}
\newcommand{\QQ}{\mathbb{Q}}
\newcommand{\RR}{\mathbb{R}}
\newcommand{\EE}{\mathbb{E}}
\newcommand{\GG}{\mathbb{G}}
\newcommand{\KK}{\mathbb{K}}

\newcommand{\NN}{\mathbb{N}}
\newcommand{\PP}{\mathbb{P}}
\newcommand{\TT}{\mathbb{T}}

\newtheorem{Theorem}{Theorem}[section]
\newtheorem{Lemma}[Theorem]{Lemma}
\newtheorem{Proposition}[Theorem]{Proposition}
\newtheorem{Problem}[Theorem]{Problem}

\newtheorem{Corollary}[Theorem]{Corollary}
\newtheorem{Definition}[Theorem]{Definition}
\newtheorem{Remark}[Theorem]{Remark}

\title[Exponential and modular forms]{From the Carlitz exponential to Drinfeld modular forms}
\author{F. Pellarin}

\address{Federico Pellarin\\
Institut Camille Jordan, UMR 5208\\
Site de Saint-Etienne \\
23 rue du Dr. P. Michelon \\
42023 Saint-Etienne, France}
\date{\today}
\begin{document}

\begin{abstract}
This paper contains the written notes of a course the author gave at the VIASM of Hanoi in the Summer 2018. It provides an elementary introduction to the analytic naive theory of Drinfeld modular forms for the simplest 'Drinfeld modular group' $\GL_2(\FF_q[\theta])$ also providing some perspectives of development, notably in the direction of the theory of vector modular forms with values in certain ultrametric Banach algebras. \end{abstract}

\maketitle

\setcounter{tocdepth}{1}
\tableofcontents

\section{Introduction}

The present paper contains the written notes of a course the author gave at the VIASM of Hanoi in the Summer 2018. It provides an elementary introduction to the analytic naive theory of Drinfeld modular forms essentially for the simplest 'Drinfeld modular group' $\GL_2(\FF_q[\theta])$ also providing some perspectives of development, notably in the direction of the theory of vector modular forms with values in certain ultrametric Banach algebras initiated in \cite{PEL1}. 

The course was also the occasion to introduce the very first basic elements of the arithmetic theory of Drinfeld modules in a way suitable to sensitise the attendance also to more familiar processes of the classical theory of modular forms and elliptic curves. Most parts of this work are not new and are therefore essentially covered by many other texts and treatises such as the seminal works of Goss \cite{GOS1,GOS2,GOS3} and Gekeler \cite{GEK}. The present text also has interaction and potential developments along with the contributions to this volume by Poineau-Turchetti and Tavares Ribeiro \cite{POI&TUR,TAV}. It also contains suggestions for further developments,
see Problems \ref{problem1},\ref{problem2},\ref{problem3},\ref{problem4} and \ref{problem5}.

This paper will not cover several advanced recent works such that the higher rank theory, including the delicate compactification questions in the path of Basson, Breuer, Pink \cite{BBP1,BBP2,BBP3}, Gekeler \cite{GEK-DMF1,GEK-DMF2,GEK-DMF3,GEK-DMF4} and it does not even go in the direction of the important arithmetic explorations notably involving the cohomological theory of crystals by B\"ockle \cite{BOE1,BOE2} or toward several other crucial recent works by several other authors we do not mention here, at once inviting the reader to realise a personal bibliographical research to determine the most recent active areas.

Perhaps, one of the original points of our contribution is instead to consider {\em exponential functions} from various viewpoints, all along the text, stressing how they interlace with modular forms. The paper describes, for example, a product expansion of the exponential function associated to the lattice $A:=\FF_q[\theta]$ in the Ore algebra of non-commutative formal series in the Frobenius automorphism which is implicit in Carlitz's work \cite{CAR}. It will be used to give a rather precise description of the analytic structure of the cusp of $\Gamma=\GL_2(A)$ acting on the Drinfeld upper-half plane by homographies. We will also use it in connection with local class field theory for the local field $K_\infty=\FF_q((\frac{1}{\theta}))$. Another new feature is that, in the last two sections, we explore structures which at the moment have no known analogue in the classical complex setting. Namely, Drinfeld modular forms with values in modules over Tate algebras, following the ideas of \cite{PEL1}.

Here is, more specifically, the plan of the paper.
In the very elementary \S \ref{rings-and-fields}
the reader familiarises with the rings and the fields which carry the values of the special functions we are going to study in this paper. Instead of the field of complex numbers $\CC$, our 'target' field is a complete, algebraically closed field of characteristic $p>0$. There is an interesting parallel with the classical complex theory where 
we have the quadratic extension $\CC/\RR$ and the quotient group $\RR/\ZZ$ is compact, but there are also 
interesting differences to take into account as the analogue $\CC_\infty/K_\infty$ of the extension $\CC/\RR$ is infinite dimensional, $\CC_\infty$ is not locally compact, although the analogue $A:=\FF_q[\theta]$ of $\ZZ$ is discrete and co-compact in the analogue $K_\infty=\FF_q((\frac{1}{\theta}))$ of $\RR$.

We dedicate the whole \S \ref{Anderson-modules-and-uniformisation} to exponential functions. More precisely, we give a proof of the correspondence by Drinfeld between {\em $A$-lattices} of $\CC_\infty$ and {\em Drinfeld $A$-modules}. To show that to any Drinfeld module we can naturally associate a lattice we pass by the more general Anderson modules. We introduce Anderson's modules in an intuitive way, privileging one of the most important and useful properties, namely that they are equipped with an exponential function at a very general level. Just like abelian varieties, Anderson modules can be of any dimension. When the dimension is one, one speaks about Drinfeld modules.

In \S \ref{Carlitz-module-and-exponential} we focus on a particular case of Drinfeld module: the Carlitz module. This is the analogue of the multiplicative group in this theory. We give a detailed  account of the main properties of its exponential function denoted by $\exp_C$. We point out that its (multiplicative, rescaled) inverse $u$ is used as uniformiser at infinity to define the analogue of the classical complex '$q$-expansions' for our modular forms. In this section we prove, for example, that any generator of the lattice of periods of $\exp_C$
can be expressed by means of a certain convergent product expansion (known to Anderson). To do this, we use the so-called {\em omega function} of Anderson and Thakur. 

In \S \ref{Drinfeld-upper-half-plane} we first study the Drinfeld 'half-plane' $\Omega=\CC_\infty\setminus K_\infty$ topologically. We use, to do this, a fundamental notion of distance from the analogue of the real line $K_\infty$. The group $\GL_2(A)$ acts on $\Omega$ by homographies and we construct a fundamental domain for this action. After a short invitation to the basic notions of rigid analytic geometry, we describe the Bruhat-Tits tree of $\Omega$, the natural action of $\GL_2(K_\infty)$ on it, and, after a glimpse on Schottky groups (see \cite{POI&TUR} for a more in-depth development),
we construct a reasonable analogue of a fundamental domain for the homographic action of $\GL_2(\FF_q[\theta])$ on $\Omega$.

In \S \ref{quotient-spaces} we discuss the following question: {\em find an analogue for the Carlitz module of the following statement: Every holomorphic function which is invariant for the translation by one has a Fourier series.} The answer is: every $\FF_q[\theta]$-translation invariant 
function has a '$u$-expansion'. We show why in this section. To do this we introduce the problem of rigid analytic structures on quotient spaces. We mainly focus on the example of the quotient of the rigid affine line $\mathbb{A}^{1,an}_{\CC_\infty}$ by the group of translations by the elements of $A$. The reader will notice how hard things can become without the use of the tool of the analytification functor, also discussed in this section.

In \S \ref{Drinfeld-modular-forms} we give a quick account of (scalar) Drinfeld modular forms for the group $\GL_2(A)$ (characterised by the $u$-expansion in $\CC_\infty[[u]]$). This appears already in many other references: the main feature is that $\CC_\infty$-vector spaces of Drinfeld modular forms are finitely dimensional spaces. Also, non-zero Eisenstein series can be constructed; this was first observed by D. Goss in \cite{GOS2}. The coefficients of the $u$-expansions of Eisenstein series are, after normalisation, in $A=\FF_q[\theta]$. 

The paper also has advanced, non-foundational parts. In \S \ref{local-class} we apply the developed knowledge of the Carlitz exponential function to give an explicit description of local class field theory for the field $K_\infty$; this subsection is also independent from the rest of the paper.
In \S \ref{Modular-forms-Tate-algebras} and \ref{Modular-forms-revisited} we revisit Drinfeld modular forms. We introduce vector Drinfeld modular forms with values in other fields and algebras, following \cite{PEL1}; the case we are interested in is that of functions which take values in finite dimensional $\KK$-vector spaces where $\KK$ is the completion for the Gauss norm of the field of rational functions in a finite set of variables with coefficients in $\CC_\infty$. With the use of certain Jacobi-like functions, we deduce an identity relating a matrix-valued Eisenstein series of weight one with certain weak modular forms of weight $-1$ from which one easily deduces \cite[Theorem 8]{PEL1} in a different, more straightforward way. 


\subsection{Acknowledgements} The author is thankful to the VIASM of Hanoi for the very nice conditions that surrounded the development of the course and the stimulating environment in which he was continuously immersed all along his visit in June 2018. Part of this text was written during a stay at the MPIM of Bonn in April 2019 and the author wishes to express gratitude for the very good conditions of work there. The author is thankful to A. Thuillier for the proof of Proposition \ref{thuillier} and to L. Gurney for fruitful discussions. He expresses his gratitude to the reviewers that, by means of a careful reading and interesting suggestions, allowed to improve it.
This work was supported by the ERC ANT.

\medskip

This work is dedicated to the memory of Velia Stassano. She was the mother of the author.

\section{Rings and fields}\label{rings-and-fields}

Before entering the essence of the topic, we first propose the reader to familiarise with certain rings and fields, notably local fields and non-archimedean valued fields. For more about these topics read, for example, the books \cite{CAS,SER}. The reader must notice that the basic notations of 
the two other chapters of this volume \cite{POI&TUR,TAV} slightly differ from ours. 

Let $R$ be a ring.

\begin{Definition}
{\em A {\em real valuation} $|\cdot|$ (or simply a {\em valuation}) over $R$ is a map $R\xrightarrow{|\cdot|}\RR_{\geq 0}$ with the following properties. 
\begin{enumerate}
\item For $x\in R$, $|x|=0$ if and only if $x=0$.
\item For $x,y\in R$, $|xy|=|x||y|$.
\item For $x,y\in R$ we have $|x+y|\leq\max\{|x|,|y|\}$ and if $|x|\neq|y|$, then
$|x+y|=\max\{|x|,|y|\}$.
\end{enumerate}
The inequality $|x+y|\leq\max\{|x|,|y|\}$ is usually called the {\em ultrametric inequality} (the term `ultrametric' indicates a reinforced triangular inequality). 
A ring with valuation is called a {\em valued ring}. A valuation is {\em non-trivial} if its image is infinite. If the image of a valuation is finite, then it is equal to the set of two elements $\{0,1\}\subset\RR_{\geq0}$ and all the non-zero elements of $R$ are sent
to $1$ while $0$ is sent to $0$. This is the {\em trivial valuation} of $R$.
A map as above satisfying (2), (3) but not (1) is called a {\em semi-valuation.}}\end{Definition}

A valuation over a ring $R$ induces a metric in an obvious way and one easily sees that $R$, together with this metric, is totally disconnected (the only connected subsets are $\emptyset$ and the points). To any valued ring $(R,|\cdot|)$ we can associate the subset $\mathcal{O}_R=\{x\in R:|x|\leq 1\}$ which is a subring of $R$, called the {\em valuation ring} of $|\cdot|$. This ring has the prime ideal $\mathcal{M}_R=\{x\in R:|x|<1\}$. The quotient ring $k_R:=\mathcal{O}_R/\mathcal{M}_R$ is called the {\em residue ring}. The ring homomorphism $f\in\mathcal{O}_R\mapsto \overline{f}+\mathcal{M}_R\in \mathcal{O}_R/\mathcal{M}_R$ is called the {\em reduction map}. 
With $R$ a ring, we denote by $R^\times$ the multiplicative group of invertible elements.
The image $|R^\times|=\{|x|:x\in R^\times\}$ is a subgroup of $\RR^\times$ called the {\em valuation group}. 

If $R$ is a field, $\mathcal{M}_R$ is a maximal ideal. Two valuations $|\cdot|$ and $|\cdot|'$ over a ring $R$ are {\em equivalent} if for all $x\in R$, $c_1|x|\leq |x|'\leq c_2|x|$
for some $c_1,c_2>0$. Two equivalent valuations induce the same topology. If $(R,|\cdot|)$ is a valued ring,
we denote by $\widehat{R}$ (or $\widehat{R}_{|\cdot|}$) the topological space completion of $R$ for $|\cdot|$.
It is a ring and if additionally $R$ is a field, $\widehat{R}$ is also a field.

While working over complete valued fields, many properties which are usually quite delicate to check for real numbers, have simple analogues in this context. For instance, the reader can check that in a valued field $(L,|\cdot|)$, a sequence $(x_n)_{n\geq 0}$ is Cauchy if and only if $(x_{n+1}-x_n)_{n\geq 0}$ tends to zero. A series $\sum_{n\geq 0}x_n$ converges if and only if $x_n\rightarrow0$ and an infinite product $\prod_{n\geq 0}(1+x_n)$ converges if and only if $x_n\rightarrow0$. Another immediate property is that if $(x_n)_{n\geq 0}$ is convergent, then $(|x_n|)_{n\geq 0}$ is ultimately constant.

\subsection{Local compactness, local fields}

Let $(L,|\cdot|)$ be a valued field. Choose $r\in |L^\times|$ and $x\in L$. We set
$$D_L(x,r)=\{y\in L:|x-y|\leq r\}.$$ This is the {\em disk of center $x$ and radius $r$}. 
Some authors like to call $r$ the {\em diameter} to stress the fact that the metric induced by the valuation makes every point of $D_L(x,r)$ into a center so that it does not really distinguishes between `radius' and `diameter'.

Observe that 
$\mathcal{O}_L=D_L(0,1)$. Also, $$\mathcal{M}_L=\bigcup_{\begin{smallmatrix}r\in|L^\times|\\ r<1\end{smallmatrix}}D_L(0,r)=:D_L^\circ(0,1).$$ More generally we write $D_L^\circ(0,r)=\{x\in L:|x|<r\}$.
We use the simpler notation $D(x,r)$ or $D^\circ(0,r)$ when $L$ is understood from the context.
Note that 
$D(x,r)=x+D(0,r)$ and $D(0,r)$ is an additive group. If $|x|\leq r$ (that is, $x\in D(0,r)$), then
$D(x,r)=D(0,r)$. If $|x|>r$ (that is, $x\not\in D(0,r)$), then $D(x,r)\cap D(0,r)=\emptyset$. In other words, if two disks with same radii have a common point, then they are equal. If the radii are not equal, non-empty intersection implies that one is contained in the other.

Now pick $r\in|L^\times|$ and $x_0\in L^\times$ with $|x_0|=r$. Then, $D(0,r)=x_0D(0,1)=x_0\mathcal{O}_{L}$. This means that all disks are homeomorphic to $\mathcal{O}_L=D(0,1)$. This is due to the fact that we are 
choosing $r\in|L^\times|$. 

A complete valued field $L$ is {\em locally compact} if every disk is compact. We have the following:

\begin{Lemma}\label{valued-field}
A valued field which is complete is locally compact if and only if the valuation group is discrete and the residue field is finite.
\end{Lemma}

\begin{proof}
Let $L$ be a field with valuation $|\cdot|$, complete. 
We first show that $\mathcal{O}_L=D(0,1)$ is compact if the valuation group is discrete (in this case there exists $r\in]0,1[\cap|L^\times|$ such that $\mathcal{M}_L=D(0,r)$) and the residue field is finite. Let $B$ be any infinite subset of $\mathcal{O}_L$. We choose a complete set of representatives $\mathcal{R}$ of $\mathcal{O}_L$ modulo $\mathcal{M}_L$. Note the disjoint union
$$\mathcal{O}_L=\bigsqcup_{\nu\in\mathcal{R}}(\nu+\mathcal{M}_L).$$ Multiplying all elements of $B$ by an element of $L^\times$ (rescaling), we can suppose that
there exists $b_1\in B$ with $|b_1|=1$. Then, the above decomposition induces a partition of $B$ and by the fact that $k_L$ is finite and the box principle there is an infinite subset $B_1\subset B\cap(b_1+\mathcal{M}_L^{n_1})$ for some integer $n_1>0$. We continue in this way and we are led to a sequence 
$b_1,b_2,\ldots$ in $B$ with $b_{i+1}\in\mathcal{M}_L^{n_i}\setminus\mathcal{M}_L^{n_i-1}$
with the sequence of the integers $n_i$ which is strictly increasing (set $n_0=0$). Hence, $b_{m+1}-b_m\in\mathcal{M}_L^{n_m}$ is a Cauchy sequence, thus converging in $L$ because it is complete.

Let us suppose that $k_L$ is infinite. Then any set of representatives $\mathcal{R}$ of $\mathcal{O}_L$ modulo $\mathcal{M}_L$ is infinite. For all $b,b'\in\mathcal{R}$ distinct, we have $|b-b'|=1$ and $\mathcal{R}$ has no converging infinite sub-sequence. Let us suppose that the valuation group $G=|L^\times|$ is dense in $\RR_{>0}$.
There is a strictly decreasing sequence $(r_i)_i\subset G$
with $r_i\rightarrow1$. This means that for all $i$, there exists $a_i\in\mathcal{O}_L$ such that
$|a_i|=r_i$ and for all $i\neq j$ we have that $|a_i-a_j|=\max\{r_i,r_j\}$ so that we cannot extract from $(a_i)_i$ a convergent sequence and $\mathcal{O}_L$ is not compact.
\end{proof}

\begin{Definition}{\em A valued field which is locally compact is called a {\em local field}.}\end{Definition}

Note that $\RR$ and $\CC$, with their euclidean topology, are locally compact, but not valued. Some authors define local fields as locally compact topological field for a non-discrete topology. Then, they distinguish between the
non-Archimedean (or ultrametric) local fields, which are the valued ones, and the Archimedean local fields: $\RR$ and $\CC$.

An important property is the following. Any valued local field $L$ of characteristic $0$ is isomorphic to a finite extension of the field of $p$-adic numbers $\QQ_p$ for some $p$, while any local field $L$ of characteristic $p>0$ is isomorphic to a local field $\FF_q((\pi))$, and with $q=p^e$ for some integer $e>0$. We say that $\pi$ is an {\em uniformiser}. Note that $|L^\times|=|\pi|^\ZZ$ and $|\pi|<1$. The proof of this result is a not too difficult deduction from the following well known fact: a locally compact topological vector space over a non-trivial locally compact field has finite dimension.

\subsection{Valued rings and fields for modular forms}\label{our-main-settings}

Let $\mathcal{C}$ be a smooth, projective, geometrically irreducible curve over $\FF_q$, together with a closed point $\infty\in\mathcal{C}$. We set $$R=A:=H^0(\mathcal{C}\setminus\{\infty\},\mathcal{O}_\mathcal{C}).$$ This is the $\FF_q$-algebra of the rational functions over $\mathcal{C}$ which are regular everywhere except, perhaps, at $\infty$. The choice of $\infty$ determines an equivalence class of valuations $|\cdot|_\infty$ on $A$ in the following way. Let $d_\infty$ be the degree of $\infty$, that is, the degree of the extension $\FF$ of $\FF_q$ generated by $\infty$ (which is also equal to the least integer $d>0$ such that $\tau^d(\infty)=\infty$,
where $\tau$ is the geometric Frobenius endomorphism). 
Then, for any $a\in A$, the degree $$\deg(a):=\dim_{\FF_q}(A/aA)$$ is a multiple $-v_\infty(a)d_\infty$ of $d_\infty$ and we set $|a|_\infty=c^{-v_\infty(a)}$ for $c>1$, which is easily seen to be a valuation. 
It is well known that 
$A$ is an arithmetic Dedekind domain with $A^\times=\FF^\times$. In addition $v_\infty(a)\leq 0$ for all $a\in A\setminus\{0\}$ and $v_\infty(a)=0$ if and only if $a\in \FF^\times$ (as a consequence of the proof of the subsequent Lemma \ref{A-coco}).
A good choice to normalise $|\cdot|_\infty$ is $c=q$. We can thus consider the field $K_\infty:=\widehat{K}_{|\cdot|_\infty}$ completion of $K$ for $|\cdot|_\infty$ which can be written as the Laurent series field $\FF((\pi))$ where $\pi$ is a {\em uniformiser element} of $K_\infty$ (such that $v_\infty(\pi)=1$). $K_\infty$ is a local field with valuation ring $\mathcal{O}_{K_\infty}=\FF[[\pi]]$, maximal ideal $\mathcal{M}_{K_\infty}=\pi\FF[[\pi]]$, residue field $\FF$ and valuation group $|\pi|_\infty^\ZZ$. Note that we have the direct sum of $\FF_q$-vector spaces:
$$K_\infty=\FF[\pi^{-1}]\oplus\mathcal{M}_{K_\infty}.$$ The case of $\mathcal{C}=\PP_{\FF_q}^1$ with its point at infinity $\infty$ (defined over $\FF_q$) is the simplest one. Let $\theta$ be any rational function having a simple pole at infinity, regular away from it. Then, $A=\FF_q[\theta]$, $K=\FF_q(\theta)$ and we can take $\pi=\theta^{-1}$ so that $K_\infty=\FF_q((\frac{1}{\theta}))$ the completion of $K$ for the valuation $|\cdot|_\infty=q^{\deg_\theta(\cdot)}$. Note that for all $\pi=\lambda\theta^{-1}+
\sum_{i>1}\lambda_i\theta^{-i}\in K_\infty$ with $\lambda\in\FF_q^\times$ and $\lambda_i\in\FF_q$,
we have $K_\infty=\FF_q((\pi))$. 
The field $K_\infty$ has an advantage over the field $\RR$: it has uniformisers. But there also is a disadvantage: there is no canonical choice in the uncountable subset of uniformisers.

We come back to the case of $A$ general. 
Let $U$ be a subset of $K_\infty$. We say that $U$ is {\em strongly discrete} if any disk $$D_{K_\infty}(x,r)=\{y\in K_\infty:|x-y|_\infty\leq r\}\subset K_\infty$$ only contains finitely many $y\in U$ for every $r\geq 0$. Note that, transposing the definition to the case of $\RR$, the ring $\ZZ$ is discrete and co-compact in $\RR$ (this is well known). 

Analogously:

\begin{Lemma}\label{A-coco}
The $\FF_q$-algebra $A$ is strongly discrete and co-compact in $K_\infty$.
\end{Lemma}

\begin{proof}[Proof of the first part of Lemma \ref{A-coco}]
That $A$ is strongly discrete in $K_\infty$ can be seen by using the {\em Liouville inequality}, asserting that for any $x\in A\setminus\{0\}$,
$|x|_\infty\geq 1$. The fraction field $K$ of $A$ is an extension of $\FF$ of transcendence degree one, and $\FF$ is algebraically closed in $K$. The closed points $P$ of $\mathcal{C}$ correspond to the classes of equivalence of multiplicative valuations over $K$
which have discrete image in $\RR_{>0}$ (discrete valuations), and which are trivial over $\FF$.
There is a set of valuations $|\cdot|_P$ (associated to the closed points of $\mathcal{C}$ different from $\infty$) such that for all $a\in A\setminus\{0\}$, $|a|_P\leq 1$ and $|a|_P=1$ for all but finitely many $P$, and such that
$$|a|_\infty\prod_P|a|_P=1,$$ see the axiomatic theory of Artin and Whaples and \cite[Theorem 2]{ART&WAP}.
This is the {\em product formula} for $A$. Let us consider $x\in A\setminus\{0\}$. We cannot have 
$|x|_\infty<1$ because this would violate the product formula. Therefore, $|x|_\infty\geq 1$ and this suffices to show strong discreteness. 
\end{proof}
We deduce that $A\cap\mathcal{M}_{K_\infty}=\{0\}$.
The next Lemma tells us that, as a `valued vector space over $\FF$', $A$ is not too different from $\FF[\pi^{-1}]$. This can be used to show co-compactness. 
\begin{Lemma}\label{the-space-V}
There exists a finite dimensional vector space $V\subset\FF[\pi^{-1}]$ over $\FF$ such that, isometrically,
$K_\infty\cong A\oplus V\oplus\mathcal{M}_{K_\infty}$.
\end{Lemma}

\begin{proof}
We can invoke Weierstrass' gap Theorem (it can be seen as one of the consequences of the Theorem of Riemann-Roch and it is nicely presented in Stichtenoth's book, \cite[Theorem 1.6.8]{STI}). Let $H(\infty)$ be the subset of $\NN$ whose elements are the nonnegative integers $k$ such that
there exists an element $f$ in $A$ with polar divisor $k[\infty]$. The Weierstrass gap Theorem asserts that the set $\NN\setminus H(\infty)$ contains exactly $g$ elements, where $g$ is the genus of $\mathcal{C}$. Additionally, $\NN\setminus H(\infty)=\{n_1,\ldots,n_g\}$, with $1=n_1<\cdots<n_g\leq 2g-1$ (so that if the genus $g$ of $\mathcal{C}$ is zero, this set is empty). We set $V:=\oplus_{i=1}^g\FF\pi^{-n_i}$ and if $g=0$ we set $V=\{0\}$. Note that $A\cap V=V\cap\mathcal{M}_{K_\infty}=A\cap\mathcal{M}_{K_\infty}=\{0\}$. Then, every element $f$ of $K_\infty$ can be decomposed in a
unique way as $f=a\oplus v\oplus m$ with $a\in A,v\in V$ and $m\in\mathcal{M}_{K_\infty}$.
\end{proof}

\begin{proof}[Proof of the second part of Lemma \ref{A-coco}]
Co-compactness is equivalent to the property that, for the metric induced on the quotient $K_\infty/A$, every sequence contains a convergent sequence. We have an isometric isomorphism
$$\frac{K_\infty}{A}\cong V\oplus\mathcal{M}_{K_\infty}$$ where $V$ is a vector space as in Lemma 
\ref{the-space-V} and we deduce that $K_\infty/A$, with the induced metric, is compact.
\end{proof} 
Up to a certain extent, the tower of rings
$A\subset K\subset K_\infty$ associated to the datum $(\mathcal{C},\infty)$ can be viewed in analogy with the tower of rings
$\ZZ\subset\QQ\subset\RR$.

Here is a fact which encourages to 'think ultrametrically'. We cannot cover a disk of radius $q$ (e.g. $D_L(0,q)$) of a non locally compact field $L$, with finitely many disks of radius $1$. Of course, this is possible, by local compactness, for the disk $D_{K_\infty}(0,q)$ in $K_\infty$. Explicitly, in the case $\mathcal{C}=\PP_{\FF_q}^1$:
$$D_{K_\infty}(0,q)=D_{K_\infty}(0,1)\oplus\FF_q\theta=\sqcup_{\lambda\in\FF_q^\times}D_{K_\infty}(\lambda\theta,1)\sqcup D_{K_\infty}(0,1).$$

\subsection{Algebraic extensions}\label{algebraic-extensions}

We start with an example in the local field $L=\FF_q((\pi))$ (with $|\pi|<1$). Let $M$ be an element of $L$
such that $|M|<1$. We want to solve the equation
\begin{equation}\label{artin-schreier}
X^q-X=M.
\end{equation}
Assuming that there exists a solution $x\in L$ we have $x=x^q-M$ so that inductively for all $n$:
$$x=x^{q^{n+1}}-\sum_{i=0}^nM^{q^i}.$$
The series $\sum_{i=0}^nM^{q^i}$ converges to $H$ in $\mathcal{M}_{L}$ by the hypothesis on $M$ and $|H|=|M|$. But $H^q-H=M$ and $x=H$ is a solution of (\ref{artin-schreier}) and the polynomial $X^q-X-M$ totally splits in $L[X]$
as all the roots are in $\{H+\lambda:\lambda\in\FF_q\}$. If $|M|=1$ we could think of writing $M=M_0+M'$
with $M_0\in\FF_q^\times$ and $|M'|<1$ but the equation (\ref{artin-schreier}) with $M=M_0$ has no roots 
in $\FF_q$. One easily sees that the equation (\ref{artin-schreier}) has no roots in $L$ if $|M|\geq 1$.
What makes the above algorithm of approximating a solution in the case $|M|<1$ is that the equation $X^q-X$ has solutions in $\FF_q$. These arguments can be generalised and formalised in what is called 
{\em Hensel's lemma}. It can be used to show the following property, which is basic and will be used everywhere. Let $L$ be a valued field with valuation $|\cdot|=c^{-v(\cdot)}$ (with a map $v:L\rightarrow\RR\cup\{\infty\}$), complete, and let us consider $F/L$ a finite extension (necessarily complete). Then, setting
$$N_{F/L}(x)=\left(\prod_{\sigma\in S}\sigma(x)\right)^{[F:L]_i},\quad x\in F,$$
where $S$ is the set of embeddings of $F$ in an algebraic closure of $L$ and $[F:L]_i$ is the inseparable degree of the extension $F/L$, 
the map 
$w:F\rightarrow \RR\cup\{\infty\}$ determined by $w(0)=\infty$ and
$$w(x)=\frac{v(N_{F/L}(x))}{[F:L]},\quad x\in F^\times$$
defines a valuation $|\cdot|_w:=c^{-w(\cdot)}$ extending $|\cdot|$ over $F$ in the only possible way.
Coming back to the local field $L=\FF_q((\pi))$, denoting by 
$L^{\operatorname{ac}}$ an algebraic closure of $L$, there is a unique valuation over $L^{\operatorname{ac}}$ extending the one of $L$; we will denote it by $|\cdot|$ by abuse of notation. The valuation group is $|\pi|^\QQ=\{|\pi|^\rho:\rho\in\QQ\}$ therefore dense in $\RR_{>0}$ and the residue field
is the algebraic closure $\FF_q^{\operatorname{ac}}$ of $\FF_q$. It is easy to see that $L^{\operatorname{ac}}$ is not complete, although each intermediate finite extension is so. 

\begin{Lemma}\label{from-goss-book}
The completion $\widehat{L^{\operatorname{ac}}}$ of $L^{\operatorname{ac}}$ is algebraically closed.
\end{Lemma}

\begin{proof} We follow \cite[Proposition 2.1]{GOS}. Let $F/\widehat{L^{\operatorname{ac}}}$ be a finite extension. Then, as seen previously, $F$ carries a unique extension of the valuation $|\cdot|$ of $\widehat{L^{\operatorname{ac}}}$.
Let $x$ be an element of $F$. We want to show that $x\in\widehat{L^{\operatorname{ac}}}$. For a polynomial $P=\sum_iP_iX^i\in \widehat{L^{\operatorname{ac}}}[X]$ we set 
$\|P\|:=\sup\{|P_i|\}$. It is easy to see that $\|\cdot\|$ is a valuation over $\widehat{L^{\operatorname{ac}}}[X]$, called the {\em Gauss valuation} (to see the multiplicativity it suffices to study the image of a polynomial in $\mathcal{O}_{\widehat{L^{\operatorname{ac}}}}[X]$ by the residue map $$\mathcal{O}_{\widehat{L^{\operatorname{ac}}}}[X]\rightarrow k_{\widehat{L^{\operatorname{ac}}}}[X]$$ which is a ring homomorphism).
Let $P\in\widehat{L^{\operatorname{ac}}}[X]$ be the minimal polynomial of $x$ over $\widehat{L^{\operatorname{ac}}}$. 
For $\|\cdot\|$, $P$ is a limit of polynomials of the same degree, which split completely. It is easy to show that for all $\epsilon>0$, there exists $N\geq 0$ with the property that for all $i\geq N$, a root $x_i\in K_\infty^{\operatorname{ac}}$ of $P_i$ satisfies $|x-x_i|_\infty<\epsilon$. This shows that $x$ is a limit of a sequence of $\widehat{L^{\operatorname{ac}}}$
and therefore, $x\in \widehat{L^{\operatorname{ac}}}$.
\end{proof}

\subsection{Analytic functions on disks}\label{Analytic-functions}

To introduce the next discussions we recall here some basic facts about ultrametric analytic functions in disks, following \cite[Chapter 3]{GOS}. 
In this subsection, $L$ denotes a valued field which is algebraic closed and complete for a valuation $|\cdot|$ (e.g. $\CC_\infty$). We consider a 
map $v:L^\times\rightarrow\RR$ such that $|\cdot|=c^{-v(\cdot)}$ for some $c>1$.
We consider a formal power series
\begin{equation}\label{asinhere}
f=\sum_{i\geq 0}f_iX^i\in L[[X]].
\end{equation}
The {\em Newton polygon} $\mathcal{N}$ of $f$ is the lower convex hull in $\RR^2$ of the set $\mathcal{S}=\{(i,v(f_i)):i\geq 0\}$. It is equal to 
$\bigcap_{\mathcal{H}}\mathcal{H}$ where $\mathcal{H}$ runs over all the closed half-planes of $\RR^2$ which contain at once $\mathcal{S}$ and a half-line $\{(x,y):y\gg0\}$ for some $x\in\RR$, where $y\gg0$ (`large enough') means that $y\geq y_0$ for some $y_0\in\RR$. 

Here is a practical method of constructing the Newton polygon $\mathcal{N}$ of a formal series $f\in L[[X]]$, if you have on-hand a wooden board, a pencil, nails, a hammer, string and a compass. Draw the axes coordinates $i$ and $v$ on the board, with the positive direction of the latter pointed toward the north, as indicated by the compass. Mark the coordinate points $(i,v (i))$ with the pencil, then hammer nails into the points. Place yourself in front of the wooden board pointing north. Take the string and pull it tautly between your hands, then begin winding it from south to north (being careful to not choose $f=0$, meaning you must have hammered in at least one nail!). A polygon figure will appear, which, transferred on the board, represents the Newton polygon of $f$.

Note that if $f\neq0$, there is always a vertical side on the left of $\mathcal{N}$. If $f$ is a non-zero polynomial, there is also a vertical side on the right. If $x\in L$ and $|f_ix^i|\rightarrow0$
then the series $\sum_if_ix^i$ converges in $L$ to an element that we denote by $f(x)$. There exists $r\in|L|$
such that $f(x)$ is defined for all $x\in D(0,r):=D_L(0,r)$ and we have thus defined a function $$D(0,r)\xrightarrow{f}L$$
that we call {\em analytic function} on the disk $D(0,r)$ (note the abuse of langage).
\begin{Proposition}\label{prop-goss}
The following properties hold.
\begin{enumerate} 
\item The sequence of slopes of $\mathcal{N}$ is strictly increasing and
its limit is $-\rho(f) = \lim\sup_{i\rightarrow\infty}v(f_i)$. The real number $\rho(f)$ is unique with the 
property that the series $f(x)$ converges for $x\in L$ such that $v(x) > \rho(f)$, and $f(x)$ diverges if $v(x) < \rho(f)$.
\item If there is a side of the Newton polygon of $f$ which has slope $-m$ and such that it does not contain any point of the Newton polygon in its interior, then $f$ has exactly $r(m)$ zeroes $x$ counted with multiplicity, with $v(x) = m$, where $r(m)$ is the length of the projection of this side of slope $-m$ onto the horizontal line. There are no other zeroes of $f$ with this property.
\item 
If $\rho(f) = -\infty$, assuming that $f$ is not identically zero, we can expand, in a unique way (Weierstrass product
expansion):
$$f(X)=cX^n\prod_i\left(1-\frac{X}{\alpha_i}\right)^{\beta_i}$$
with $c\in L^\times$, where $\alpha_i\rightarrow\infty$ is the sequence of zeroes such that $v(\alpha_i)>v(\alpha_{i+1})$ (with multiplicities $\beta_i\in\NN^*$).
\end{enumerate}
\end{Proposition}
By (2) of the proposition, if we set $r=c^{-\rho(f)}\in\RR_{\geq 0}$, $f$ is analytic on $D(0,r')$ for all $r'\in|L|$
such that $r'<r$ and $r$ is maximal with this property. 
If $\rho(f) = -\infty$ then we say that $f$ is {\em entire}. 
We can show easily that if $f$ is entire and non-constant, then it is surjective, and furthermore, an entire function without zeroes is constant. Also, if $f$ as above is non-entire and non-constant, in general it is not surjective,
but we have a reasonable description of the image of disks by it, given by the next corollary, the proof of which is left to the reader.

\begin{Corollary}\label{imageofdisks}
Let $f$ be as in (\ref{asinhere}) with $f_0=0$ and let us suppose that it converges on $D_L(0,r)$ with $r\in|L^\times|$. Then, $f\big(D_L(0,r)\big)=D_L(0,s)$ for some $s\in|L|$.
\end{Corollary}
To be brief: an analytic function sends disks to disks.

\subsection{Further properties of the field $\CC_\infty$}

We consider as in \S \ref{our-main-settings} the local field $K_\infty$. Then, $K_\infty=\FF((\pi))$ for some
uniformiser $\pi$ and by Lemma \ref{from-goss-book}, the field $$\CC_\infty:=\widehat{K_\infty^{\operatorname{ac}}}$$
is algebraically closed and complete.
It will be used in the sequel as an alternative to $\CC$ 'for silicon-based mathematicians' (\footnote{Opposed to 'carbon-based mathematicians', following David Goss.}), but there are many important differences.
For instance, note that $\CC/\RR$ has degree $2$, while $\CC_\infty/K_\infty$ is infinite dimensional, as the reader can easily see by observing that 
$\FF$-linear elements of $\FF_q^{\operatorname{ac}}$ are also $K_\infty$-linearly independent (in fact, this $K_\infty$-vector space is uncountably-dimensional and the group of automorphisms is an infinite, profinite group).
 
Complex analysis makes heavy use of local compactness so that we can cover a compact analytic space with finitely many disks. For example, we can cover an annulus with finitely many disks so that the union does not contain the center, which is very useful in path integration of analytic functions over $\CC\setminus\{0\}$. The ultrametric counterparts of this and other familiar and intuitive statements are false in $\CC_\infty$ as well as in other non-locally compact fields. We cannot use 'partially overlapping disks' to 'move' in $\CC_\infty$, or, more generally, in a non-Archimedean  space. The intuitive idea of `moving' itself is different even thought it is not too different, as two annuli, or a disk and an annulus, may overlap somewhere without being one included in the other. 

On another hand, the field $\CC_\infty$ also has 'nice' properties. Let us review some of them; we denote by $L^{\operatorname{sep}}$ the separable closure of a field $L$. 

\begin{Lemma}\label{lemma-sep}
We have $\CC_\infty=\widehat{K_\infty^{\operatorname{sep}}}$.
\end{Lemma}

\begin{proof}
This is consequence of simple metric properties of Artin-Schreier extensions. We follow \cite{AX}. 
First look at the equation $$X^{q'}-X=M$$ with $M\in K_\infty$ and where $q'=p^{e'}$ for some $e'>0$. Then, if $|M|_\infty>1$, all the solutions $\gamma\in\CC_\infty$ of the equation are such that
$|\gamma|_\infty^{q'}=|M|_\infty$ and $|\gamma^{q'}-M|_\infty<|M|_\infty$. This also is a very simple consequence of Proposition \ref{prop-goss}: the reader can study the Newton polygon 
of $f(X)=X^{q'}-X-M$ inspecting the three different cases $|M|_\infty<1,|M|_\infty=1$ and $|M|_\infty>1$. Here, with $|M_\infty|>1$, the extension $K_\infty(\gamma)/K_\infty$ is clearly separable but ramified as by Proposition \ref{prop-goss}, the polynomial $X^{q'}-X-M$ has $q'$ distinct roots $x$ in $K_\infty^{\operatorname{ac}}$ with valuation $|x|_\infty=|M|_\infty^{\frac{1}{q'}}$. It is in fact a {\em wildly ramified} extension:
this means that the characteristic $p$ of $\FF_q$ divides the index of ramification. 

We now consider $\alpha\in K_\infty^{\operatorname{ac}}$. We want to show that $\alpha$ is a limit of $K_\infty^{\operatorname{sep}}$. There exists $q'=p^{e'}$ with $a:=\alpha^{q'}\in K_\infty^{\operatorname{sep}}$. For instance, we can take $q'=[K_\infty(\alpha):K_\infty]_i$ (inseparable degree). Consider $b\in K_\infty^\times$ and a root $\beta\in K_\infty^{\operatorname{ac}}$ of the polynomial equation $X^{q'}-bX-a=0$.
Clearly, $\beta\in K_\infty^{\operatorname{sep}}$. 
Let $\lambda\in K_\infty^{\operatorname{sep}}$ be such that $\lambda^{q'-1}=b$. Then, setting $\gamma=\frac{\beta}{\lambda}$,
we have $\gamma^{q'}=\frac{\beta^{q'}}{\lambda^{q'}}=\frac{\beta^{q'}}{b\lambda}$ so that 
$$\gamma^{q'}-\gamma=\frac{a}{b\lambda}=:M.$$
We can choose $b\in K_\infty^\times$ such that $|b|_\infty$ is small enough so that $|M|_\infty>1$. If this is the case, then $|\gamma|_\infty^{q'}=|\frac{a}{b\lambda}|_\infty$ so that
$$|\beta|_\infty^{q'}=|a|_\infty.$$
Since $(\beta-\alpha)^{q'}=\beta^{q'}-a=b\beta$,
$$v_\infty(\beta-\alpha)=\frac{1}{q'}v_\infty(\beta^{q'}-a)=\frac{1}{q'}\left(v_\infty(b)+v_\infty(\beta)\right)=\frac{1}{q'}\left(v_\infty(b)+\frac{1}{q'}v_\infty(a)\right).$$
We choose a sequence $(b_i)_i\subset K_\infty^\times$ with $b_i\rightarrow0$. For all $i$, let $\beta_i\in K_\infty^{\operatorname{sep}}$ be such that
$\beta^{q'}=\beta_ib_i+a$ and $\beta_i\rightarrow\alpha$. Then, 
$v_\infty(\beta_i-\alpha)\rightarrow\infty$ as $v_\infty(b_i)\rightarrow\infty$ so that
$\beta_i\rightarrow\alpha$. 
\end{proof}

We deduce that $|\CC_\infty^\times|_\infty=|\pi|_\infty^\QQ$ with $\pi$ a uniformiser of $K_\infty$, and the residue field of $\CC_\infty$ is $\FF_q^{ac}$ the algebraic closure of $\FF_q$ in $\CC_\infty$. 

The next results are not used in the rest of the text but mentioning them is helpful in understanding important subtleties lying in the bases of the theory of Drinfeld modular forms.

\begin{Lemma}\label{mapping-c-infty}
The group $\CC_\infty^\times$ contains a subgroup $\pi^\QQ\cong(\QQ,+)$ which is totally ordered for $|\cdot|_\infty$. There are group epimorphisms 
$$\CC_\infty^\times\xrightarrow{\varpi}\pi^\QQ,\quad \CC_\infty^\times\xrightarrow{\operatorname{sgn}}(\FF_q^{ac})^\times$$
such that $\varpi$ induces the identity on $\pi^\ZZ$, $\operatorname{sgn}$ induces the identity on $(\FF_q^{ac})^\times$, and 
for all $x\in\CC_\infty^\times$, $$|x-\varpi(x)\operatorname{sgn}(x)|_\infty<|x|_\infty.$$
\end{Lemma}
One can see that a choice of $\pi^\QQ$, $\varpi$ etc. corresponds to an embedding of $\CC_\infty$ in a {\em maximal immediate extension} of it (that is to say, a field extension which is maximal with same valuation group and same residue field) or, equivalently, in a certain type of field of {\em Hahn generalised series}, {\em spherically complete}. Read Poineau and Turchetti's contribution \cite[Definition I.2.17, Theorem I.2.18, Example I.2.20]{POI&TUR}. Read also Kedlaya's \cite{KED}. 

The group $G:=\operatorname{Gal}(K_\infty^{\operatorname{sep}}/K_\infty)$ acts on $\CC_\infty$ by continuous 
$K_\infty$-linear automorphisms. Then the following important result holds, where the completion on the right is that of the perfect closure of $K_\infty$ in $\CC_\infty$ (see for example \cite{AX}):
\begin{Theorem}[Ax-Sen-Tate]\label{ax-sen-tate}
$\CC_\infty^G:=\{x\in\CC_\infty:g(x)=x,\forall g\in G\}=\widehat{K_\infty^{\operatorname{perf}}}$.
\end{Theorem}

\section{Drinfeld modules and uniformisation}\label{Anderson-modules-and-uniformisation}

Drinfeld modules are also at the hearth of Tavares Ribeiro contribution to this volume, read 
\cite[\S 1.4]{TAV}.
Let $R$ be an $\FF_q$-algebra and $\tau:R\rightarrow R$ be an $\FF_q$-linear endomorphism.
We denote by $R[\tau]$ the left $R$-module of the finite sums $\sum_if_i\tau^i$ ($f_i\in R$) equipped with the $R$-algebra structure given by $\tau b=\tau(b)\tau$ for $b\in R$ (\footnote{It would be more appropriate, to define this $R$-algebra, to choose an indeterminate $X$ and consider as the underlying $R$-module the polynomial ring $R[X]$ setting the product to be $Xb=\tau(b)X$. This is an Ore algebra and the standard notation for it is $R[X;\tau]$. For the purposes we have in mind, the abuse of notation $R[\tau]$ is harmless.}).

Let $f=\sum_{i=0}^nf_i\tau^i$ be in $R[\tau]$. For any $b\in R$ we can evaluate $f$ in $b$ by setting
$$f(b)=\sum_{i=0}^nf_i\tau^i(b)\in R.$$
This gives rise to an $\FF_q$-linear map $R\rightarrow R$. Note that the element $f=\sum_if_i\tau^i$
and the associated evaluation map $f:R\rightarrow R$ are two completely different objects. However, in this text, we will denote them with the same symbols.

We choose $R$ by returning to the notations of \S \ref{our-main-settings}. In particular considering the $\FF_q$-algebra $A=H^0(\mathcal{C}\setminus\{\infty\},\mathcal{O}_{\mathcal{C}})$ we construct the tower of rings
$$A\subset K\subset K_\infty\subset \CC_\infty$$ arising from \S \ref{algebraic-extensions} which is analogous of $\ZZ\subset\QQ\subset\RR\subset\CC$. 

\subsection{Drinfeld $A$-modules and $A$-lattices}\label{Drinfeld-modules}

We show here the crucial correspondence between Drinfeld $A$-modules and $A$-lattices, due to Drinfeld \cite{DRI}. 
The definition of Drinfeld module that we give here is not the most general one but it will nevertheless be enough for our purposes. Remember that, in the construction of the tower of rings
$A\subset K\subset K_\infty\subset K_\infty^{\operatorname{ac}}\subset \CC_\infty$ we have in fact chosen an embedding $A\subset\CC_\infty$.

\begin{Definition}\label{def-Drinfeld-modules}{\em An injective $\FF_q$-algebra morphism $\phi:A\rightarrow\operatorname{End}_{\FF_q}(\GG_a(\CC_\infty))\cong\CC_\infty[\tau]$ is a {\em Drinfeld $A$-module of rank $r>0$} if for all $a\in A$ 
$$\phi_a:=\phi(a)=a+(a)_1\tau+\cdots+(a)_{rd_\infty\deg(a)}\tau^{rd_\infty\deg(a)}\in\CC_\infty[\tau],$$
where $d_\infty$ is the degree over $\FF_q$ of the residue field of $A$ and the coefficients $(a)_i$ are in $\CC_\infty$ and depend on $a$, and where $\deg(a)=\dim_{\FF_q}(A/(a))$. If $R$ is an $\FF_q$-subalgebra of $\CC_\infty$ containing $A$ and the coefficients $(a)_i$ with $1\leq i\leq rd_\infty\deg(a)$ and $a\in A$, we say that the Drinfeld $A$-module $\phi$ is {\em defined over $R$ and we write $\phi/R$.}}
\end{Definition}

Note that geometrically, a Drinfeld module defined over $\CC_\infty$ is just $\GG_a$ over $\CC_\infty$. What makes the theory interesting is the fact that there are many embeddings of $A$ in $\operatorname{End}_{\FF_q}(\GG_a(\CC_\infty))$. The case of the {\em Carlitz module}, which can be viewed as the 'simplest' Drinfeld module of rank one, is analysed in \S \ref{Carlitz-module-and-exponential}.

The set of Drinfeld $A$-modules of rank $r$ is equipped with a natural structure of small category.
 If $\varphi$ and $\phi$ are two Drinfeld $A$-modules, we say that they are {\em isogenous} if there exists $\nu\in\CC_\infty[\tau]$ such that $\varphi_a\nu=\nu\psi_a$ for all $a\in A$.  If $\nu$, seen as a non-commutative polynomial in $\tau$, is constant, then we say that $\varphi$ and $\psi$ are {\em isomorphic}. Being isogenous induces an equivalence relation on Drinfeld $A$-modules and isogenies are the morphisms connecting Drinfeld $A$-modules of same rank in our category.
 
We prove that the category of Drinfeld $A$-modules of rank $ r$  is equivalent to another category, that of {\em $A$-lattices}.
 
\begin{Definition}{\em An {\em $A$-lattice} in $\CC_\infty$ is a finitely generated strongly discrete $A$-submodule
$\Lambda\subset\CC_\infty$ and two $A$-lattices $\Lambda$ and $\Lambda'$ are isogenous if there exists $c\in\CC_\infty^\times$ such that $c\Lambda\subset\Lambda'$ with $c\Lambda$ of finite index in $\Lambda'$.}\end{Definition}
Isogenies are the morphisms connecting lattices. Clearly, this also defines an equivalence relation. If two $A$-lattices $\Lambda$ and $\Lambda'$ are such that there exists
$c\in\CC_\infty$ with $c\Lambda=\Lambda'$, then we say that $\Lambda$ and $\Lambda'$ are isomorphic.

Since $A$ is a Dedekind ring, any $A$-lattice $\Lambda$ is projective and has a rank $r=\operatorname{rank}_A(\Lambda)$.
We have the following lemma, the proof of which is left to the reader.
\begin{Lemma}
Let $\Lambda$ be a projective $A$-module of rank $r$. Then $\Lambda$ is an $A$-lattice if and only if $K_\infty$-vector space generated by $\Lambda$ has dimension $r$. 
\end{Lemma}
Observe that, in contrast with the complex case, for all $r>1$ there exist infinitely many non-isomorphic $A$-lattices (this can be deduced from the fact that $\CC_\infty$ is not locally compact). 
We choose an $A$-lattice $\Lambda$ of rank $r$ as above.

By Proposition \ref{prop-goss} the following product (where the dash $(\cdot)'$ indicates that the factor corresponding to $\lambda=0$ is omitted)
$$\exp_\Lambda(Z):=Z\sideset{}{'}\prod_{\lambda\in\Lambda}\left(1-\frac{Z}{\lambda}\right)$$
converges to an entire function $\CC_\infty\rightarrow\CC_\infty$ (hence surjective) called the {\em exponential function} associated to $\Lambda$. Note that this is an $\FF_q$-linear entire function with kernel $\Lambda$, and we can write
$$\exp_\Lambda(Z)=\sum_{i\geq 0}\alpha_i\tau^i(Z),\quad \alpha_i\in\CC_\infty,\quad \alpha_0=1,\quad \forall Z\in\CC_\infty.$$
In particular, $\frac{d}{dZ}\exp_\Lambda(Z)=1$, and the 'logarithmic derivative' (defined in the formal way) of $\exp_\Lambda$ coincides with its multiplicative inverse and is equal to the series
$$\sum_{\lambda\in\Lambda}\frac{1}{Z-\lambda},\quad Z\in\CC_\infty\setminus\Lambda.$$
We refer to \cite[\S 2]{GEK} for an account on the properties of this fundamental class of analytic functions. 

It is not always an easy task to construct explicitly Drinfeld $A$-modules for a given $A=H^0(\mathcal{C}\setminus\{\infty\},\mathcal{O}_{\mathcal{C}})$, if $\mathcal{C}\neq\PP_{\FF_q}^1$.
The following result is due to Drinfeld \cite{DRI} and shows the depth of the problem.
\begin{Theorem}\label{drinfeld-theorem}
There is an equivalence of small categories $$\{A-\text{lattices of rank $r$}\}\rightarrow\{\text{Drinfeld $A$-modules of rank $r$ defined over $\CC_\infty$}\}.$$
\end{Theorem}

\begin{proof} The proof that we propose is essentially self-contained except for the use of Theorem \ref{theoexponentialexists} which is the crucial tool, showing how to associate to any Drinfeld $A$-module 
an exponential function. We postpone this result and its proof to \S \ref{modules-to-exponential}.

Let $\Lambda$ be a lattice of rank $r$ (so that it is a projective $A$-module). The $\FF_q$-linear entire map
$\exp_\Lambda$ gives rise to the exact sequence of $\FF_q$-vector spaces
$$0\rightarrow \Lambda\rightarrow \CC_\infty\xrightarrow{\exp_\Lambda}\CC_\infty\rightarrow0.$$
For any $a\in A$ there is a unique $\FF_q$-linear map $\CC_\infty\xrightarrow{\phi_a}\CC_\infty$ 
such that $$\exp_\Lambda(aZ)=\phi_a(\exp_\Lambda(Z))$$ for all $Z\in\CC_\infty$ and we want to show
that the family $(\phi_a)_{a\in A}$ gives rise to a Drinfeld $A$-module of rank $r$. By the snake lemma
we get $\operatorname{ker}(\phi_a)\cong\Lambda/a\Lambda\cong(A/(a))^r$. Note also that
$\operatorname{ker}(\phi_a)=\exp_\Lambda(a^{-1}\Lambda)$. We set 
$$P_a(Z):=aZ\sideset{}{'}\prod_{\alpha\in\operatorname{ker}(\phi_a)}\left(1-\frac{Z}{\alpha}\right)=aZ+(a)_1Z^q+\cdots+(a)_{r\deg(a)}Z^{q^{r\deg(a)}}.$$
Note that the functions $P_a(\exp_\Lambda(Z))$ and $\exp_\Lambda(aZ)$ are both entire with divisor
$a^{-1}\Lambda$ and the coefficient of $Z$ in their entire series expansions are equal. Hence these functions are equal and we can write
$$\phi_a(Z)=aZ+(a)_1Z^q+\cdots+(a)_{r\deg(a)}Z^{q^{r\deg(a)}},\quad \forall a\in A,\quad Z\in\CC_\infty.$$
This defines a Drinfeld $A$-module $\phi$ of rank $r$ such that $\exp_\Lambda(aZ)=\phi_a(\exp_\Lambda(Z))$
for all $a\in A$ so we have defined a map associating to $\Lambda$ an $A$-lattice of rank $r$ a Drinfeld module $\phi_\Lambda$
of rank $r$. 

The next step is to show that the map $\Lambda\mapsto\phi_\Lambda$ that we have just constructed, from the 
set of $A$-lattices of rank $r$ to the set of Drinfeld $A$-modules of rank $r$, is surjective. From the proof it will be possible to derive that it is also injective. Let $\phi$ be a Drinfeld $A$-module of rank $r$. We want to construct $\Lambda$ an $A$-lattice of rank $r$ such that $\phi=\phi_\Lambda$. By the subsequent Theorem \ref{theoexponentialexists}, there exists a unique entire $\FF_q$-linear function $\exp_\phi:\CC_\infty\rightarrow\CC_\infty$ such that for all $a\in A$, 
$\exp_\phi(aZ)=\phi_a(\exp_\phi(Z))$, and this, for all $Z\in\CC_\infty$. We set $\Lambda=\operatorname{Ker}(\exp_\phi)$. Then $\Lambda$ is a strongly discrete $A$-module in $\CC_\infty$. The snake lemma implies that 
$\Lambda/a\Lambda\cong\operatorname{Ker}(\phi_a)$, which is an $\FF_q$-vector space of dimension 
$r\deg(a)$. Let $\epsilon>0$ be a real number and let $V_\epsilon$ be the $K_\infty$-subvector space of $\CC_\infty$ generated by $\Lambda\cap D(0,\epsilon)$. We also set $\Lambda_\epsilon:=V_\epsilon\cap\Lambda$. Observe that 
$\Lambda_\epsilon$ is an $A$-lattice (it is a finitely generated $A$-module because of the finiteness of the 
dimension of $V_\epsilon$) which is saturated by construction. Hence
$\Lambda_\epsilon/a\Lambda_\epsilon$ injects in $\Lambda/a\Lambda$ and this for all $\epsilon>0$ which 
means $\operatorname{rank}_A(\Lambda_\epsilon)=\dim_{K_\infty}(\Lambda_\epsilon)\leq r$ for all $\epsilon>0$. Setting $V=\cup_\epsilon V_\epsilon$ we see that $\dim_{K_\infty}(V)\leq r$. From this we easily deduce that $\Lambda$ is finitely generated and since $\Lambda/a\Lambda\cong(A/(a))^{r}$ we derive that $\Lambda$ is an $A$-lattice of rank $r$.

Hence the map $\Lambda\mapsto\phi_\Lambda$ is surjective and one sees easily that it is also injective by looking at $\exp_\Lambda$. Finally, 
the map is in fact an equivalence of small categories with the natural notions of morphisms between $A$-lattices and Drinfeld $A$-modules that we have introduced. We leave the details of these verifications to the reader.
\end{proof}

\subsection{From Drinfeld modules to exponential functions}\label{modules-to-exponential}

In order to complete the proof of Theorem \ref{drinfeld-theorem} it remains to show how to associate 
to a Drinfeld $A$-module an exponential function. This is the object of the present subsection and we will take the opportunity to present things in a rather more general setting, by introducing Anderson's $A$-modules.
We recall here the definition of Hartl and Juschka in \cite{HAR&JUS}. 
\begin{Definition}
{\em An {\em Anderson $A$-module} of dimension $d$ (over $\CC_\infty$) is a pair $\underline{E}=(E,\varphi)$ where $E$ is an $\FF_q$-module scheme isomorphic to $\GG_a(\CC_\infty)^d$, together with a ring homomorphism $\varphi:A\rightarrow\operatorname{End}_{\FF_q}(E)$, such that
 for all $a\in A$, $(\operatorname{Lie}(\varphi(a))-a)^d=0$.}
 \end{Definition} If $R$ is a ring,
we denote by $R^{m\times n}$ the set of matrices with $m$ rows and $n$ columns with entries in $R$.
Note that there is an $\FF_q$-isomorphism $\operatorname{End}_{\FF_q}(E)\cong\CC_\infty[\tau]^{d\times d}$. If $d=1$ we are brought to Definition \ref{def-Drinfeld-modules} of Drinfeld $A$-modules.

Anderson modules fit in a category which can be compared to that of commutative algebraic groups; this category is of great importance for the study of global function field arithmetic. A remarkable feature which  
allows to track similarities with commutative algebraic groups is the fact that we can associate, to every such module, an exponential function. In \cite[Proposition 8.7]{BOE&HAR} (see also Anderson in \cite[Theorem 3]{AND}) B\"ockle and Hartl proved that every Anderson's $A$-module $\underline{E}$ possesses a unique exponential function $$\exp_{\underline{E}}:\operatorname{Lie}(\underline{E})\rightarrow E(\CC_\infty)$$ in the following way (compare also with \cite[Proposition 1.11]{TAV}). Identifying 
$\operatorname{Lie}(\underline{E})$ (defined fonctorially) with $\CC_\infty^{d\times 1}$, $\exp_{\underline{E}}$ is an entire function of $d$ variables $z={}^t(z_1,\ldots,z_d)\in\CC_\infty^{d\times 1}$ (${}^t(\cdots)$ denotes the transposition)
$$z \mapsto\exp_{\underline{E}}(z)=\sum_{i\geq 0}E_iz^{q^i}$$ with $E_0=I_d$ and $E_i\in \CC_\infty^{d\times d}$
such that, for all $a\in A$ and $z\in\CC_\infty^d$, 
$$\exp_{\underline{E}}(\operatorname{Lie}(\varphi_a)z)=\varphi_a(\exp_{\underline{E}}(z)).$$

We show how to construct $\exp_{\underline{E}}$ in a slightly more general setting.
Let $B$ be any commutative integral countably dimensional $\FF_q$-algebra. 
We follow \cite{GAZ&MAU} and we define $\|\cdot\|_\infty$ on $A\otimes_{\FF_q}B$
by setting, for $x\in A\otimes_{\FF_q}B$, $\|x\|_\infty$ to be the infimum of the values $\max_{i}|a_i|_\infty$, running over any finite
sum decomposition $$x=\sum_ia_i\otimes b_i$$ with $a_i\in A$ and $b_i\in B\setminus\{0\}$. Then, $\|\cdot\|_\infty$ is a norm 
of $A\otimes_{\FF_q}B$ extending the valuation of $A$ via $a\mapsto a\otimes1$. The $\FF_q$-algebra $A\otimes_{\FF_q}B$ is equipped with the $B$-linear endomorphism 
$\tau$ defined by $a\otimes b\mapsto a^q\otimes b$ (thus extending the $q$-th power map $a\mapsto a^q$
which is an $\FF_q$-linear endomorphism of $A$). Similarly, we can consider the $\CC_\infty$-algebra 
$$\TT=\CC_\infty\widehat{\otimes}_{\FF_q}B,$$ the completion of $\CC_\infty\otimes_{\FF_q}B$ for $\|\cdot\|_\infty$ defined accordingly, and we also have a $B$-linear extension of $\tau$. Let $d>0$ be an integer.
We allow $\tau$ to act on $d\times d$ matrices of $\TT^{d\times d}$ with entries in $\TT$ on each coefficient.
Then, $\TT[\tau]$
acts on $\TT$ by evaluation and $\TT[\tau]^{d\times d}\subset\operatorname{End}_B(\TT^{d\times 1})$.
If $f\in \TT[\tau]^{d\times d}$ we can write $f=\sum_{i=0}^nf_i\tau^i$ with $f_i\in \TT^{d\times d}$ and we
set $\operatorname{Lie}(f):=f_0$ which provides a $\TT$-algebra morphism $$\operatorname{Lie}(f):\TT[\tau]^{d\times d}\rightarrow\TT^{d\times d}.$$

\begin{Definition}\label{def-anderson-module}{\em An {\em Anderson $A\otimes_{\FF_q} B$-module $\varphi$ of dimension $d$} 
is an injective $B$-algebra homomorphism $$A\otimes_{\FF_q} B\xrightarrow{\varphi}\TT[\tau]^{d\times d}$$
such that for all $a\in A$, $(\operatorname{Lie}(\varphi(a))-a)^d=0$.}
\end{Definition}
We prefer to write $\varphi_a$ in place of $\varphi(a)$.

We now revisit the proof of Proposition 8.7 of \cite{BOE&HAR} and the method is flexible enough to adapt to the setting of Definition \ref{def-anderson-module}. Note also that later in this text, we will be interested in the case $B=\FF_q$ only, case in which we essentially recover \cite[Theorem 3]{AND}. In the following, 
the non-commutative ring $\TT[[\tau]]$ is defined in the obvious way with $\TT[\tau]$ as a subring.
In the following, 
we denote by $\|M\|_\infty$ the supremum of $\|x\|_\infty$ where $x$ varies in the entries of a matrix $M\in\TT^{m\times n}$. 
We show:

\begin{Theorem}\label{theoexponentialexists}
Given an Anderson $A\otimes_{\FF_q} B$-module $\varphi$, there exists a unique series $$\exp_\varphi=\sum_{i\geq 0}E_i\tau^i\in\TT[[\tau]]^{d\times d}$$ with the coefficients $E_i\in \TT^{d\times d}$ and with $E_0=I_d$, such that the evaluation series $\exp_\varphi(z)$ is convergent for all $z\in\TT^{d\times 1}$, and such that 
$$\varphi_a(\exp_\varphi(z))=\exp_\varphi(\operatorname{Lie}(\varphi_{a})z),$$ for all $z\in\TT^{d\times 1}$ and 
$a\in A\otimes_{\FF_q} B$. For all $a\in A\setminus\{0\}$ we have that $\exp_\varphi$
 is the limit for $n\rightarrow\infty$ of the sequence of entire functions $\varphi_{a^n}a^{-n}\in \CC_\infty[[\tau]]^{d\times d}$, uniformly convergent on every subset of $\TT[[\tau]]^{d\times 1}$, 
 bounded for the norm $\|\cdot\|_\infty$.
 \end{Theorem}

Before proving this result, we need two lemmas.

\begin{Lemma}\label{convergence-to-zero}
Let us consider $\mathcal{L},\mathcal{M}\in\TT[\tau]^{d\times d}$ with $\mathcal{L}=\alpha+\mathcal{N}$,
with $\alpha\in\operatorname{GL}_d(\TT)$ such that $\|\alpha\|_\infty>1$ and 
$\mathcal{M},\mathcal{N}\in(\TT[\tau]\tau)^{d\times d}$. Then, for all $R\in\|\TT^\times\|_\infty$,
the sequence of functions given by the evaluation of $(\mathcal{L}^N\mathcal{M}\alpha^{-N})_{N\geq 0}$ converges uniformly on $D_\TT(0,R)^{d\times 1}$ to the zero function.
\end{Lemma}

\begin{proof} The multiplication defining $\mathcal{L}^N\mathcal{M}\alpha^{-N}$ is that of $\TT[[\tau]]^{d\times d}$.
Locally near the origin, $\alpha^{-1}\mathcal{L}$ is an isometric isomorphism and there exists $R_0\in\|\TT^\times\|_\infty$
with $0<R_0<1$ such that for all $x\in D_{\TT}(0,R_0)^{d\times 1}$,
$\|\mathcal{L}(x)\|_\infty=\|\alpha x\|_\infty\leq \|\alpha\|_\infty \|x\|_\infty$. 
Hence, for $N\geq 0$, if 
$\|x\|_\infty\leq \|\alpha\|_\infty^{-N}R_0$ ($<R_0$ because of the hypothesis on $\alpha$), we have $\|\mathcal{L}^N(x)\|_\infty\leq \|\alpha\|_\infty^N\|x\|_\infty$. 

We can choose $R_0$ small enough so that 
$\|\mathcal{M}(x)\|_\infty\leq \beta\|x\|_\infty^{q^l}$ for some $\beta\in \|\TT^\times\|_\infty$ and $l>0$. Let $R$ be in $\|\TT^\times\|_\infty$ fixed, and let us suppose that $N$ is large enough so that 
$\|\alpha\|_\infty^{-N}R\leq R_0$. Then, for all $x\in D_\TT(0,R)^d$, $\|\mathcal{M}(\alpha^{-N}x)\|_\infty\leq\beta(\|\alpha\|_\infty^{-N}R)^{q^l}$. If $N$ is large enough, we can also suppose that
$$\beta(\|\alpha\|_\infty^{-N}R)^{q^l}<\|\alpha\|_\infty^{-N}R_0$$ (because $l>0$). Therefore, 
$\|(\mathcal{L}^N\mathcal{M})(\alpha^{-N}x)\|_\infty\leq\|\alpha\|_\infty^N\beta(\|\alpha\|_\infty^{-N}R)^{q^l}\rightarrow0$
as $N\rightarrow\infty$, for all $x\in D_\TT(0,R)^{d\times 1}$.
\end{proof}
We consider an Anderson $A\otimes B$-module $\varphi$ and we recall that $\operatorname{Lie}(\varphi_a)$ is the coefficient in $\TT^{d\times d}$
of $\tau^0I_d$ in the expansion of $\varphi_a\in \TT[\tau]^{d\times d}$ along powers of $I_d\tau$. If $a\in A\otimes B\setminus\FF_q\times B$, $\operatorname{Lie}(\varphi_a)=aI_d+N_a$
with $N_a$ nilpotent. 
Then, $\alpha=\operatorname{Lie}(\varphi_a)\in\operatorname{GL}_d(\TT)$ is such that $\|\alpha\|_\infty>1$. 
Indeed otherwise $N_a-\alpha-aI_d$ would be invertible.

Let us consider $a,b\in A\otimes B$, $\|a\|_\infty>1$.
We construct the sequence of $B$-linear functions $\TT^{d\times 1}\xrightarrow{\mathcal{F}_N^a}\TT^{d\times 1}$ defined by 
$$\mathcal{F}^a_N=\varphi_{a^Nb}\operatorname{Lie}(\varphi_{a^Nb})^{-1},\quad N\geq 0.$$

\begin{Lemma}\label{uniqueness-existence}
The sequence $(\mathcal{F}_N^a)$ converges uniformly on every polydisk $D_\TT(0,R)^{d\times 1}$ and the limit function $\TT^{d\times1}\rightarrow\TT^{d\times1}$ is independent of the choice of $b$.
\end{Lemma}

\begin{proof}
We set $\mathcal{G}_N^a=\mathcal{F}_{N+1}^a-\mathcal{F}_N^a$. Then, 
$$\mathcal{G}_N^a=\underbrace{\varphi_{a^N}}_{=:\mathcal{L}^N}\underbrace{\varphi_b(\varphi_a\operatorname{Lie}(\varphi_a)^{-1}-I_d)\operatorname{Lie}(\varphi_b)^{-1}}_{=:\mathcal{M}}\operatorname{Lie}\underbrace{(\varphi_a)^{-N}}_{=:\alpha^{-N}}$$
and by Lemma \ref{convergence-to-zero}, the sequence converges uniformly to the zero function
on every polydisk $D_\TT(0,R)^{d\times 1}$ which ensures the uniform convergence of the sequence 
$'\mathcal{F}_N^a)$. Observe now that, writing momentarily $\mathcal{F}_N^{a,b}$ to designate the above 
function associated to the choice of $a,b$, 
$$\mathcal{F}_{a^N}^{a,b}-\mathcal{F}_{a^N}^{a,1}=\underbrace{\varphi_{a^N}}_{=:\mathcal{L}^N}\underbrace{(\varphi_b\operatorname{Lie}(\varphi_{b})^{-1}-I_d)}_{=:\mathcal{M}}\underbrace{\operatorname{Lie}(\varphi_{a^N})^{-1}}_{=:\mathcal{\alpha}^{-N}},$$
so that, again by Lemma \ref{convergence-to-zero} this sequence tends to zero uniformly on every polydisk, and the limit $\mathcal{F}^a$ of the sequence $\mathcal{F}^a_N$ is uniquely determined, independent of $b$.
\end{proof}

\begin{proof}[Proof of Theorem \ref{theoexponentialexists}]
Let us denote by $\mathcal{F}^a$ the continuous $B$-linear 
map which, by Lemma \ref{uniqueness-existence} is the common limit of all the sequences $(\mathcal{F}^{a,b}_N)_N$ (that can be identified with a
formal series $x\mapsto\sum_{i\geq 0}E_i\tau^i(x)\in\TT^{d\times d}[[\tau]]$). First of all, note that $E_0=I_d$ so that this map 
is not identically zero. Moreover, observe that, for all $b\in A\otimes B$:
\begin{eqnarray*}
\varphi_b\mathcal{F}^a&=&\varphi_b\lim_{N\rightarrow\infty}\mathcal{F}_N^{a,1}\\
&=&\varphi_b\lim_{N\rightarrow\infty}\varphi_{a^N}\operatorname{Lie}(\varphi_{a^N})^{-1}\\
&=&\lim_{N\rightarrow\infty}\varphi_{ba^N}\operatorname{Lie}(\varphi_{ba^N})^{-1}\operatorname{Lie}(\varphi_{b})\\
&=&\lim_{N\rightarrow\infty}\mathcal{F}_N^{a,b}\operatorname{Lie}(\varphi_{b})\\
&=&\mathcal{F}^a\operatorname{Lie}(\varphi_{b}).
\end{eqnarray*}
Hence we see that for all $a$, $\mathcal{F}^a$ satisfies the property of the theorem. Now, let $\mathcal{F}_1$
and $\mathcal{F}_2$ be two elements of $\TT^{d\times d}[[\tau]]$ such that $\varphi_b(\mathcal{F}_i(z))=
\mathcal{F}_i(bz)$ for all $b\in A\otimes B$ and $i=1,2$, and with the property that 
$\mathcal{F}_3=\mathcal{F}_1-\mathcal{F}_2\in \TT^{d\times d}[[\tau]]\tau$. Suppose by contradiction that $\mathcal{F}_3$ is non-zero. Then we can write 
$\mathcal{F}_3=\sum_{i\geq i_0}F_i\tau^i$ with $F_i\in\TT^{d\times d}$ and $F_{i_0}$ non-zero. Since 
$\mathcal{F}_3$ also satisfies the same functional identities of both $\mathcal{F}_1,\mathcal{F}_2$ (for $b\in A\otimes B$), we get $\operatorname{Lie}(\varphi_b)F_{i_0}=F_{i_0}\tau^{i_0}(\operatorname{Lie}(\varphi_b))$ for all $b$.
Let $w$ be an eigenvector of $F_{i_0}$ with non-zero eigenvalue, defined over some algebraic closure of 
the fraction field of $\TT$. We consider $b\in A\otimes B$ with $\|b\|_\infty>1$. Writing $\operatorname{Lie}(\varphi_b)=b+N_b$ with $N_b$ nilpotent, we see that $\operatorname{Lie}(\varphi_b)w=\tau^{i_0}(\operatorname{Lie}(\varphi_b))w$ which implies $(b-\tau^{i_0}(b))w=(\tau^{i_0}(N_b)-N_b)w=Mw$ and $M$ is nilpotent. Hence, there is a power $c$ of $b-\tau^{i_0}(b)$ such that $cw=0$ which means that $b=\tau^{i_0}(b)$; a contradiction because the valuations do not agree. This means that $\mathcal{F}_1=\mathcal{F}_2$.
In particular, $\mathcal{F}=\mathcal{F}^a$ does not depend on the choice of $a$ and the theorem is proved.
\end{proof}

\section{The Carlitz module and its exponential}\label{Carlitz-module-and-exponential}
 
 In this section we set
 $$A=H^0(\PP_{\FF_q}^1\setminus\{\infty\},\mathcal{O}_{\PP_{\FF_q}^1})=\FF_q[\theta],$$
 $\theta$ being a rational function over $\PP_{\FF_q}^1$ having a simple pole at $\infty$ and no other singularity. 
The simplest example of Anderson's $A$-module is the {\em Carlitz module} which is discussed here; it has rank one and it is perhaps the only one with which we can make very simple computations so it is legitimate to spend some time on it. 
 In order to simplify our notations, we write $$|\cdot|=|\cdot|_\infty=q^{-v_\infty(\cdot)},\quad \|\cdot\|=\|\cdot\|_\infty$$ from now on; this will not lead to confusion.

\begin{Definition}[Cf. Example 1.9 of \cite{TAV}]{\em 
The {\em Carlitz $A$-module} is the Drinfeld $A$-module $A\xrightarrow{C} \CC_\infty[\tau]$ uniquely
defined by $C_\theta=C(\theta)=\theta+\tau$.}
\end{Definition} 
Let $a$ be in $A$. Then,
$C_a\in A[\tau]$ has degree $\deg_\theta(a)$ in $\tau$ and the rank is $1$.
Note also that $C$ is defined over the $\FF_q$-algebra $A$.


 We give an example of computation where we can see how this $A$-module structure over an $A$-algebra $R$ works. We suppose $q=2$. Let $1$ be the unit of $R^\times$. We have $C_\theta(1)=\theta+1$. Hence, $$C_{\theta^2+\theta}(1)=C_{\theta+1}(C_\theta(1))=(\theta+1)^2+\theta^2+1=0.$$ This means that $1$ is a $(\theta^2+\theta)$-torsion point for this $A$-module structure given by the Carlitz module.

By Theorem \ref{theoexponentialexists}, the limit series $$\exp_C:=\lim_{N\rightarrow\infty}C_{\theta^N}\theta^{-N}\in\CC_\infty[[\tau]],$$ not identically zero and which can be identified with an entire $\FF_q$-linear endomorphism of $\CC_\infty$, satisfies 
\begin{equation}\label{functional-exp}
\exp_Ca=C_a\exp_C\end{equation} for all $a\in A$ and has constant term (with respect to the expansion in powers of $\tau$) equal to one. 
By Theorem \ref{drinfeld-theorem}, the Carlitz module $C$ corresponds to a rank one lattice $\nu A\subset\CC_\infty$, with generator $\nu\in\CC_\infty$, and we have
$$\exp_{C}(Z)=\exp_{\nu A}(Z)=Z\sideset{}{'}\prod_{\lambda\in \nu A}\left(1-\frac{Z}{\lambda}\right),\quad Z\in\CC_\infty.$$
Our next purpose is to compute $\nu$ explicitly. To do this, we are going to use properties of the Newton polygon of $\exp_C$. Indeed, staring at (\ref{functional-exp}) it is a simple exercise to show that there is a unique solution $Y\in\CC_\infty[[\tau]]$ of $C_\theta Y=Y\theta$ with the coefficient of $\tau^0$ equal to one, and by uniqueness, we find
$$\exp_C=\sum_{i\geq 0}d_i^{-1}\tau^i,$$
where $$d_i=(\theta^{q^i}-\theta^{q^{i-1}})\cdots(\theta^{q^i}-\theta^{q})(\theta^{q^i}-\theta)=(\theta^{q^i}-\theta)d_{i-1}^q$$ (if $i>0$ and with $d_0=1$). 
From $v_\infty(d_{i})=-iq^i$ we observe again that $\exp_C$ defines an $\FF_q$-linear entire function which is therefore also surjective over $\CC_\infty$ (use Proposition \ref{prop-goss}). We have the normalisation of $|\cdot|$ by $|\theta|=q$.

\begin{Proposition} There exists an element $\nu\in\CC_\infty$ with $v_\infty(\nu) =-\frac{q}{q-1}$, such that the kernel of $\exp_C$ is equal to the $\FF_q$-vector space $\nu A$. The 
element $\nu$ is defined up to multiplication by an element of $\FF_q^\times$.\end{Proposition}

\begin{proof}
We know already from Theorem \ref{drinfeld-theorem} that the kernel of $\exp_C$ has rank one over $A$. The novelty here is that we can compute the valuation of its generators, a property which is not available from the theorem.
The Newton polygon of $\exp_C$ is the lower convex hull in $\RR^2$ of the set whose elements are the points $(q^i, iq^i)$. Since
$$(q^{i+1}, (i+1)q^{i+1})-(q^i, iq^i) = (q^i(q-1), iq^i(q-1) + q^{i+1})$$ for $i\geq 0$, the sequence $(m_i)$ of the slopes of the Newton polygon is
$$\frac{iq^i(q-1)+q^{i+1}}{q^i(q-1)}=i+\frac{q}{q-1}.$$
Projecting this polygon on the horizontal axis we deduce that for all $i\geq0$, $\exp_C$ has exactly $q^i(q-1)$ zeroes $x$ such that $v_\infty(x) = -i-\frac{q}{q-1}$ (counted with multiplicity) and no other zeroes. In particular, we have $q- 1$ distinct zeroes such that $v_\infty(x) = -\frac{q}{q-1}$. The multiplicity of any such zero is one (note that
$\frac{d}{dX}\exp_C(X)=1$) so they are all distinct. Now, since $\exp_C$ is $\FF_q$-linear, we have that all the roots $x$ such that $v_\infty(x) = -1-\frac{1}{q-1}$ are multiple, with a factor in $\FF_q^\times$, of a single element $\nu$ (there are $q-1$ choices).
We denote by $A[d]$ the set of polynomials of $A$ of exact degree $d$. For all $a\in A[d]$,
$0 = C_a(\exp_C(\nu)) = \exp_C(a\nu)$ and $v_\infty(a\nu) = -d-\frac{q}{q-1}$ . This defines an injective
map from $A[d]$ to the set of zeroes of $\exp_C$ of valuation $-d-\frac{q}{q-1}$. But this set
has cardinality $q^d(q-1)$ which also is the cardinality of $A[d]$. This means that $\exp_C(x) = 0$ if and only if $x\in \nu A$.
\end{proof}
\begin{Corollary}\label{exactsequence}
We have $\exp_C(X)=X\prod_{a\in A\setminus\{0\}}\left(1-\frac{X}{a\nu}\right)$ and $\exp_C$ induces an exact sequence of $A$-modules
$$0\rightarrow \nu A\rightarrow\CC_\infty\xrightarrow{\exp_C}C(\CC_\infty)\rightarrow0.$$
\end{Corollary}

\subsection{A formula for $\nu$}

We have seen that if $\Lambda\subset\CC_\infty$ is the kernel of $\exp_C$, then $\Lambda$ is a free $A$-module of rank one generated by $\nu\in\CC_\infty$ with $v_\infty(\nu)=-\frac{q}{q-1}$, defined up to multiplication by an element of $\FF_q^\times$. Let us choose a $(q-1)$-th root $(-\theta)^{\frac{1}{q-1}}$ of $-\theta$; this is also defined up to multiplication by an element of $\FF_q^\times$, and the valuation is $-\frac{1}{q-1}$. We want to prove the following formula:
$$\nu=\theta(-\theta)^{\frac{1}{q-1}}\prod_{i>0}\left(1-\frac{\theta}{\theta^{q^i}}\right)^{-1}.$$
To do this, we will use Theorem \ref{theoexponentialexists}. We recall that this result implies that the sequence
$$f_n(z)=\exp_C(z)-C_{\theta^n}(z\theta^{-n})$$ converges uniformly on every bounded disk of $\CC_\infty$ to the zero function. To continue further, we need to introduce the {\em function $\omega$ of Anderson and Thakur}.
This function is defined by the following product expansion:
$$\omega(t)=(-\theta)^{\frac{1}{q-1}}\prod_{i\geq 0}\left(1-\frac{t}{\theta^{q^i}}\right)^{-1}.$$
The convergence of this product is easily seen to hold for any $t\in\CC_\infty\setminus\{\theta,\theta^q,\theta^{q^2},\ldots\}$. Also, for all $n\neq1$, the function
$$(t-\theta)(t-\theta^q)\cdots(t-\theta^{q^{n-1}})\omega(t)$$
extends to an analytic function over $D_{\CC_\infty}(0,q^{n-1})$ (we can also say that $\omega$ defines a 
meromorphic function over $\CC_\infty$ having simple poles at the singularities defined above). 
To study the arithmetic properties of $\omega$, it is useful to work in {\em Tate algebras}. However, at this level of generality, this is not necessary, strictly speaking. For the purposes we have in mind now, it will suffice to work with formal Newton-Puiseux series.
Let $y,t$ be two variables, choose a $(q-1)$-th root of $y$ and 
define:
$$F(y,t)=(-y)^{\frac{1}{q-1}}\prod_{i\geq 0}\left(1-\frac{t}{y^{q^i}}\right)^{-1}\in\FF_q((y^{-\frac{1}{q-1}}))((t)).$$
Then, $$F(y^q,t)=(t-y)F(y,t).$$ Writing the series expansion $$\omega(t)=\sum_{i\geq 0}\lambda_{i+1}t^i\in\CC_\infty[[t]],$$ we deduce, from the uniqueness of the series expansion of an analytic function in $D_{\CC_\infty}(0,1)$, that the sequence $(\lambda_i)_{i\geq 0}$ can be defined by setting $\lambda_0=0$ and the algebraic relations $$C_\theta(\lambda_{i+1})=\lambda_{i+1}^q+\theta\lambda_{i+1}=\lambda_i$$ which include $\lambda_1=(-\theta)^{\frac{1}{q-1}}$. Now set $\mu_i=\theta^i\lambda_i$, $i\geq 0$.

\begin{Lemma}
For all $i\geq 1$, $|\mu_i|=q^{\frac{q}{q-1}}$ and $(\mu_i)_{i\geq0}$ is a Cauchy sequence.
\end{Lemma}

\begin{proof}
Developing the product defining $\omega$ we see that $|\lambda_i|=q^{\frac{q}{q-1}-i}$. To see that $(\mu_i)$ is a Cauchy sequence, it suffices to show that $\mu_{i+1}-\mu_i\rightarrow0$. But
$$\mu_{i+1}-\mu_i=\theta^{i+1}\lambda_{i+1}-\theta^i\lambda_i=\theta^i(\lambda_i-\lambda_{i+1}^q)-\theta^i\lambda_i=-\theta^i\lambda_{i+1}^q\rightarrow0.$$
\end{proof}

Let $\mu\in\CC_\infty$ be the limit of $(\mu_i)$. 

\begin{Lemma}
We have $\mu=-\lim_{t\rightarrow\theta}(t-\theta)\omega(t)=\theta(-\theta)^{\frac{1}{q-1}}\prod_{i>0}(1-\theta^{1-q^i})^{-1}$.
\end{Lemma}
\begin{proof}
From the functional equation of $F(y,t)$ we see that $\lim_{t\rightarrow\theta}(t-\theta)\omega(t)=(-\theta)^{\frac{q}{q-1}}\prod_{i>0}(1-\theta^{1-q^i})^{-1}=\sum_{i\geq 0}\theta^i\lambda_{i+1}^q$, the latter series being convergent. Using that $C_\theta(\lambda_{i+1})=\lambda_i$ we see that the last sum is:
$$\sum_{i\geq 0}\theta^i(\lambda_i-\theta\lambda_{i+1})=\sum_{i=0}^{N-1}\theta^i(\lambda_i-\theta\lambda_{i+1})+\sum_{i\geq N}\theta^i\lambda_{i+1}^q,\quad \forall N.$$
The first sum telescopes to $-\theta^N\lambda_N$ while the second being a tail series of a convergent series, it converges and the sum depending on $N$ tends to $0$ as $N\rightarrow\infty$. 
\end{proof}
Hence $\mu$ is the residue of $-\omega$ at $t=\theta$. We can write
$$\mu=-\operatorname{Res}_{t=\theta}(\omega).$$
This is the analogue of a well known lemma sometimes called Appell's Lemma: if $(a_n)$ is a converging sequence of complex numbers, then $\lim_na_n=\lim_{x\rightarrow 1^-}(1-x)\sum_na_nx^n$.

We are now ready to prove the following well known and classical result:

\begin{Theorem}\label{explicit-product-pi}
The kernel $\Lambda$ of $\exp_C$ is generated, as an $A$-module, by $$\mu=\nu=\theta(-\theta)^{\frac{1}{q-1}}\prod_{i>0}\Big(1-\theta^{1-q^i}\Big)^{-1}.$$
\end{Theorem}

\begin{proof}
Since $\Lambda=\nu A$ for some $\nu\in\CC_\infty$ such that $|\nu|=q^{\frac{q}{q-1}}$ and since $|\mu|=q^{\frac{q}{q-1}}$, it suffices to show that $\exp_C(\mu)=0$.
Now, we can write $\mu=\mu_n+\epsilon_n$ where $\epsilon_n\rightarrow0$ and $|\epsilon_n|<q^{\frac{q}{q-1}}$. Also, we have $\exp_C(z)=f_n(z)+C_{\theta^n}(\theta^{-n}z)$ and we have that the sequence of entire functions
$(f_n)$ converges uniformly to the zero function on any bounded subset of $\CC_\infty$. We have:
\begin{eqnarray*}
\exp_C(\mu) & =& (C_{\theta^n}\theta^{-n}+f_n)(\mu_n+\epsilon_n)\\
&=&\underbrace{C_{\theta^n}(\lambda_n)}_{=0}+\underbrace{f_n(\mu_n)}_{\rightarrow0}+\underbrace{\exp_C(\epsilon_n)}_{\rightarrow 0}.
\end{eqnarray*}
Hence, $\mu=\nu$.
\end{proof}
\begin{Remark}
{\em The formula of Theorem \ref{explicit-product-pi} can be easily derived from the following 
result of Carlitz in \cite{CAR} that also appears in \cite[Theorem 3.2.8]{GOS}. Let $\eta$ be a $(q-1)$-th root of $\theta-\theta^q$ in the algebraic closure $K^{\operatorname{ac}}$ of $K$ in $\CC_\infty$. We set:
$$\xi=\eta\prod_{j\geq1}\left(1-\frac{\theta^{q^j}-\theta}{\theta^{q^{j+1}}-\theta}\right)\in K_\infty^{\operatorname{ac}}.$$ Then $\mu\in\FF_q^\times\xi$. To see this, observe the identity: 
$$\prod_{j=1}^{d-1}\left(1-\frac{\theta^{q^j}-\theta}{\theta^{q^{j+1}}-\theta}\right)=\prod_{j=0}^{d-1}(1-\theta^{q^j(1-q)})\prod_{i=1}^d(1-\theta^{1-q^i})^{-1},\quad d\geq 1.$$
Both products on $d$ converge in $K_\infty$ for $d\rightarrow\infty$.
If we set $H=\eta\prod_{j\geq0}(1-\theta^{q^j(1-q)})\in K_\infty^{\operatorname{ac}}$ we see that $H$ is algebraic over $K$
by the relations $H^q=(\theta-\theta^q)\eta(1-\theta^{1-q})^{-1}\prod_{j\geq0}(1-\theta^{q^j(1-q)})=-\theta^qH$. Since $-\theta^q=\theta^{q-1}(-\theta)$, we deduce that $H\in\FF_q^\times\theta(-\theta)^{\frac{1}{q-1}}$. The formulation that we adopt in our text is that of Anderson, Brownawell and Papanikolas in \cite[\S 5.1]{ABP}.
In fact, the proof of Theorem \ref{explicit-product-pi} that we gave above is inspired by that of these authors.}
\end{Remark}

One of the most used notations for our $\mu$ is $\widetilde{\pi}$. 
This is suggestive due to the resemblance between the exact sequence of Corollary \ref{exactsequence}
and $0\rightarrow 2\pi i\ZZ\rightarrow\CC\xrightarrow{\exp}\CC\rightarrow1$; there is an analogy between $\widetilde{\pi}\in\CC_\infty$ and $2\pi i\in\CC$.
It can be proved, by the product expansion we just found, that $\widetilde{\pi}$ in transcendental over $K=\FF_q(\theta)$. The first transcendence proof of it is that of Wade in \cite{WAD} but there are several others, very different from each other. See for example \cite[\S 3.1.2]{ABP}. There are proofs which make use of computations of dimensions of `motivic Galois groups' which connect to the topics of Di Vizio's contribution to this volume \cite{VIZ} and which are the roots of a vast program in transcendence and algebraic independence inaugurated by Anderson, Brownawell and Papanikolas in \cite{ABP}, and later by Papanikolas in \cite{PAP0}.

\subsection{A factorization property for the Carlitz exponential}
\label{factorization-property}

In Corollary \ref{exactsequence}, we described the Weierstrass product expansion of the entire function
$\exp_C:\CC_\infty\rightarrow\CC_\infty$. We now look again at $\exp_C$ as a formal series and we provide it with another product expansion, this time in $\CC_\infty[[\tau]]$; see Proposition \ref{prod-facto}. This result is implicit in Carlitz's \cite[(1.03), (1.04) and (5.01)]{CAR}. The function we factorise is not $\exp_C$ but a related one: $$\exp_A(z)=z\prod_{a\in A\setminus\{0\}}\left(1-\frac{z}{a}\right)=\widetilde{\pi}^{-1}\exp_C(\widetilde{\pi}z),$$ so that
$$\exp_A=\sum_{i\geq 0}d_i^{-1}\widetilde{\pi}^{q^i-1}\tau^i\in K_\infty[[\tau]].$$

Before going on we must discuss the {\em Carlitz logarithm}. It is easy to see that in $\CC_\infty[[\tau]]$, there exists a unique formal series $\log_C$ with the following properties: (1) $\log_C=1+\cdots$ (the constant term in the power series in $\tau$ is the identity $1=\tau^0$)
and (2) for all $a\in A$, $a\log_C=\log_CC_a$, a condition which is equivalent to $\theta\log_C=\log_CC_\theta$
by the fact that $A=\FF_q[\theta]$. Writing $\log_C=\sum_{i\geq0}l_i^{-1}\tau^i$ and using this remark one easily shows that $$l_i=(\theta-\theta^q)(\theta-\theta^{q^2})\cdots(\theta-\theta^{q^i}),$$ $i\geq 0$. We note that 
$v_\infty(l_i)=-q\frac{q^i-1}{q-1}$. This means that the series $\log_C$ does not converge to an entire function
but for all $R\in|\CC_\infty^\times|$ such that $R<|\widetilde{\pi}|$, $\log_C$ defines an $\FF_q$-linear 
function on $D_{\CC_\infty}(0,R)$. We also note, reasoning with the Newton polygons of $\exp_C$ and $\log_C$, that
\begin{equation}\label{isometry}
|\exp_C(z)|=|z|=|\log_C(z)|,\quad \forall z\in D_{\CC_\infty}^\circ(0,|\widetilde{\pi}|),
\end{equation}
which implies that the Carlitz's exponential induces an isometric automorphism of $D_{\CC_\infty}^\circ(0,|\widetilde{\pi}|)$. More generally, the exponential function of a Drinfeld module induces, locally, an 
isometric automorphism, see \cite[Corollary 1.12]{TAV}.
We observe that the series $U=\exp_C\log_C$ and $V=\log_C\exp_C$ in $K_\infty[[\tau]]$
 satisfy $Ua=aU$ and $Va=aV$ for all $a\in A$. Since they further satisfy $U=1+\cdots$ and $V=1+\cdots$, we deduce that $\log_C$ is the inverse of $\exp_C$ in $K_\infty[[\tau]]$. In particular, 
$$C_a=\exp_C a\log_C\in K_\infty[\tau],\quad \forall a\in A.$$
We define:
$$C_z=\exp_C z\log_C\in \CC_\infty[[\tau]],\quad z\in\CC_\infty.$$
Then,
\begin{eqnarray*}
C_z&=&\sum_{i\geq 0}d_i^{-1}\tau^iz\sum_{j\geq 0}l_j^{-1}\tau^j\\
&=&\sum_{i\geq 0}d_i^{-1}z^{q^i}\tau^i\sum_{j\geq0}l_j^{-1}\tau^j\\
&=&\sum_{k\geq 0}\left(\underbrace{\sum_{i=0}^kd_i^{-1}l_{k-i}^{-q^i}z^{q^i}}_{=:E_k(z).}\right)\tau^k
\end{eqnarray*}
We can thus expand, for all $z\in\CC_\infty$:
$$C_z=\sum_{k\geq 0}E_k(z)\tau^k\in\CC_\infty[[\tau]]$$ with the coefficients
$$E_k(z)=\sum_{i=0}^kd_i^{-1}l_{k-i}^{-q^i}z^{q^i}=\frac{z}{l_k}+\cdots+\frac{z^{q^k}}{d_k}\in K[z]$$ which are $\FF_q$-linear polynomials of degree $q^k$ in $z$ for $k\geq 0$. They are called the 
{\em Carlitz' polynomials.}
In the next proposition we collect some useful properties of these polynomials.
\begin{Proposition}\label{propthreepropofek}
The following properties hold:
\begin{enumerate}
\item For all $k\geq 0$ we have $$E_k(z)=d_k^{-1}\prod_{\begin{smallmatrix}a\in A \\
|a|<q^k\end{smallmatrix}}(z-a).$$
\item For all $k\geq 0$ and $z\in\CC_\infty$ we have 
$$E_k(z)^q=E_k(z)+(\theta^{q^{k+1}}-\theta)E_{k+1}(z).$$
\item We have $l_kE_k(z)\rightarrow \exp_A(z)$ uniformly on every bounded subset of $\CC_\infty$.
\end{enumerate}
\end{Proposition}
\begin{proof}
(1). Since $C_a\in A[\tau]$ has degree in $\tau$ which is equal to $\deg_\theta(a)$,
$E_k$ vanishes on $A(<k)$ the $\FF_q$-vector space of the polynomials of $A$ which have degree $< k$.
Since the cardinality of this set is equal to the degree of $E_k$, this vector space exhausts the zeroes of $E_k$, and the leading coefficient is clearly $d_k^{-1}$.

(2) This is a simple consequence of the relation $C_aC_z=C_zC_a$.

(3) We note that 
$$\frac{l_k}{d_k}\prod_{|a|<q^k}(z-a)=\frac{l_k}{d_k}z\prod_{\begin{smallmatrix}a\neq 0\\ |a|<q^k\end{smallmatrix}}(-a)\left(1-\frac{z}{a}\right).$$ Now, it is easy to see that 
\begin{equation}\label{a-product}
\prod_{0\neq |a|<q^k }(-a)=\prod_{0\neq |a|<q^k}a=\frac{d_k}{l_k}.\end{equation} (see \cite[\S 3.2]{GOS}). The uniform convergence is clear.
\end{proof}
 
We come back to the series $\exp_A=\sum_{i\geq 0}d_i^{-1}\widetilde{\pi}^{q^i-1}\tau^i\in K_\infty[[\tau]].$ We now show that
\begin{multline}\label{noncommutativeidentity}
\exp_A=\cdots \left(1-\frac{\tau}{l_n^{q-1}}\right)\left(1-\frac{\tau}{l_{n-1}^{q-1}}\right)\cdots\left(1-\frac{\tau}{l_1^{q-1}}\right)\left(1-\tau\right)=\\=
\cdots l_n(1-\tau)\frac{1}{\theta^{q^n}-\theta}(1-\tau)\cdots \frac{1}{\theta^{q^2}-\theta}(1-\tau)\frac{1}{\theta^{q}-\theta}(1-\tau).
\end{multline}
in $K_\infty[[\tau]]$ with its $(\tau)$-topology. We have in fact more:
\begin{Proposition}\label{prod-facto}
On every bounded subset of $\CC_\infty$, the entire function $\exp_A(z)$ is the uniform limit of the sequence
of $\FF_q$-linear polynomials $$\left(z-\frac{z^q}{l_n^{q-1}}\right)\circ\left(z-\frac{z^q}{l_{n-1}^{q-1}}\right)\circ\cdots\circ\left(z-\frac{z^q}{l_1^{q-1}}\right)\circ\left(z-z^q\right),$$ where $\circ$ is the composition.
\end{Proposition}

\begin{proof} We write:
$$\widetilde{\mathcal{E}}_n=\left(1-\frac{\tau}{l_{n-1}^{q-1}}\right)\cdots\left(1-\frac{\tau}{l_1^{q-1}}\right)\left(1-\tau\right)\in K[\tau].$$
We also denote by $\mathcal{E}_n\in K[\tau]$ the unique element such that for all $z\in\CC_\infty$, 
$\mathcal{E}_n(z)=E_n(z)$ (evaluation). Part (3) of Proposition \ref{propthreepropofek} implies that $l_kE_k$ converges uniformly to $\exp_A(z)$ on every bounded subset of $\CC_\infty$. Hence, we are done if we show that the evaluations agree:
$\widetilde{\mathcal{E}}_n=l_n\mathcal{E}_n$ for all $n\geq 0$. This is certainly true if $n=0$. 
We continue by induction. From part (2) of Proposition \ref{propthreepropofek} we see that
$\tau\mathcal{E}_n=\mathcal{E}_n+(\theta^{q^{n+1}}-\theta)\mathcal{E}_{n+1}$ for all $n\geq 0$. Therefore:
\begin{eqnarray*}
\widetilde{\mathcal{E}}_{n+1} & = & \left(1-\frac{\tau}{l_n^{q-1}}\right)\widetilde{\mathcal{E}}_n\\
&=&\left(1-\frac{\tau}{l_n^{q-1}}\right)l_n\mathcal{E}_n\\
&=&l_n\mathcal{E}_n-l_n^ql_n^{-q+1}\tau\mathcal{E}_n\\
&=&l_n\mathcal{E}_n-l_n(\mathcal{E}_n+(\theta^{q^{n+1}}-\theta)\mathcal{E}_{n+1})\\
&=&\underbrace{l_n(\theta-\theta^{q^{n+1}})}_{=l_{n+1}}\mathcal{E}_{n+1},
\end{eqnarray*}
and we are done.
\end{proof}

Proposition \ref{prod-facto} was essentially known by Carlitz; it can be derived easily with elementary manipulations on the left-hand side of \cite[(5.01)]{CAR}.
It is interesting to note the two rationality properties for $\exp_C=\exp_{\widetilde{\pi}A}$ and 
$\exp_A$ which follow from the above result: the terms of the series defining $\exp_C$ are defined over $K$ (the coefficients $d_i^{-1}$) and the factors of the infinite product of $\exp_A$ we just considered are also defined over $K$ (the coefficients are $l_i^{1-q}$).

\begin{Problem}\label{problem1}{\em 
Generalise Lemma \ref{property-of-Fn} and Proposition \ref{prod-facto} to the framework of {\em Drinfeld-Hayes $A$-modules of rank one} considered in \cite{HAY} for a general $\FF_q$-algebra of regular functions $A$ and highlight a connection to the {\em shtuka functions} in the sense of \cite[\S 7.11]{GOS} in this context, see also \cite[\S 4.2]{TAV}.}
\end{Problem}

\begin{Remark} 
{\em This can be viewed as a digression. There is a simple connection with Thakur's multiple zeta values, defined by:
$$\zeta_A(n_1,n_2,\ldots,n_r):=\sum_{\begin{smallmatrix}a_1,\ldots,a_n\in A^+\\ |a_1|>\cdots>|a_r|\end{smallmatrix}}a_1^{-n_1}\cdots a_r^{-n_r}\in K_\infty,\quad n_1,\ldots,n_r\in \NN^*,\quad r\geq 1,$$
where $A^+$ denotes the subset of monic polynomials of $A$.
Indeed, one sees directly that the coefficient of $\tau^r$ in (\ref{noncommutativeidentity}) is equal to 
$$(-1)^r\sum_{i_1>\cdots>i_r\geq 0}l_{i_1}^{1-q}l_{i_2}^{q-q^2}\cdots l_{i_r}^{q^{r-1}-q^r}.$$ One proves easily
$\sum_{\begin{smallmatrix}a\in A^+\\ |a|=q^i\end{smallmatrix}}a^{-l}=l_{i}^{-l}$ for $1\leq l\leq q$ and
we deduce that
$$\exp_A=\sum_{r\geq 0}(-1)^r\zeta_A(q-1,q(q-1),\ldots,q^{r-1}(q-1))\tau^r.$$ 
Therefore, equating the corresponding coefficients of the powers of $\tau$ we reach the formula:
$$\zeta_A(q-1,q(q-1),\ldots,q^{r-1}(q-1))=(-1)^r\frac{\widetilde{\pi}^{q^r-1}}{d_r},\quad r\geq 0,$$
with the convention $\zeta_A(\emptyset)=1$. Note that the identity derived by the specialisation $t=\theta$ in \cite[(22)]{PEL5} rather involves the `reversed' multiple zeta values $\zeta_A^*(q^{r-1}(q-1),\ldots,q(q-1),q-1)$,
the $*$ denoting the variant of multiple zeta value involving sums with non-strict inequalities $|a_1|\geq\cdots\geq|a_r|$.}
\end{Remark}
 
\subsection{The function $\exp_A$ and local class field theory}\label{local-class}

This subsection is not logically related to the other topics of the text.
Just as the Euler exponential function, the Carlitz exponential function has an important role in explicit class field theory for the field $K$
(see Hayes \cite{HAY} for the rational function field $K=\FF_q(\theta)$, \cite{HAY2} and the more recent work of Zywina \cite{ZYW}, for the general case). Note that even more recently, a direct link between the explicit class field theory of $K=\FF_q(\theta)$ and the function $\omega$ of Anderson and Thakur has been found in \cite{ANG&PEL2}. It does not belong to our purposes to describe these results here. In this subsection we are going to achieve a more modest objective which is to apply, in the case $A=\FF_q[\theta]$, the properties of the function $\exp_A$ we have reviewed so far,
in relation with the local class field theory for $K_\infty=\FF_q((\frac{1}{\theta}))$. Interestingly, these properties do not seem to have simple analogues in the theory of Euler's exponential function. 

Let $L\subset\CC_\infty$ be an algebraic extension of $K_\infty$. Then, $\exp_A$ defines 
an $\FF_q$-linear map $L\rightarrow L$. Indeed, for all $x\in L$, $K_\infty(x)/K_\infty$ is a finite extension, hence complete, and $\exp_A(K_\infty(x))\subset K_\infty(x)$.

\begin{Definition}\label{uniformised}{\em 
 We say that $L$ is {\em uniformised by $\exp_A$} if the map $\exp_A:L\rightarrow L$ is surjective.}
\end{Definition}

For example, $L=\CC_\infty$ is uniformised by $\exp_A$, thanks to Proposition \ref{prop-goss}.
Observe that if $L,L'\subset\CC_\infty$ are two algebraic extensions of $K_\infty$ which are uniformised by $\exp_A$, then also $L\cap L'$ is uniformised by $\exp_A$. Indeed, let $x$ be an element of $L\cap L'$ and let $y\in L,y'\in L'$ be such that $\exp_A(y)=\exp_A(y')=x$. Then 
$y-y'\in A=\operatorname{Ker}(\exp_A)\subset K_\infty$ so that $y,y'\in L\cap L'$. Hence,
there is a {\em minimal algebraic extension} $L/K_\infty$ in $\CC_\infty$ that is uniformised by $\exp_A$; this is what we want to study here.

We denote by $K_\infty^{\operatorname{ab}}$ the {\em maximal abelian extension} of $K_\infty$ in 
$K_\infty^{\operatorname{sep}}\subset\CC_\infty$, that is, the maximal extension of $K_\infty$ which is Galois, with abelian Galois group. We also choose $\lambda_\theta$ a $(q-1)$-th root of $-\theta\in K_\infty^{\operatorname{sep}}$ and we note that if $L/K_\infty$ is an algebraic extension, then $L[\lambda_\theta]$ is an algebraic extension of $K_\infty$.
The aim of this subsection is to prove:

\begin{Theorem}\label{local-class-field}
Let $L$ be the minimal algebraic extension of $K_\infty$ in $\CC_\infty$ which is uniformised by $\exp_A$. Then, $L[\lambda_\theta]=K_\infty^{\operatorname{ab}}$.
\end{Theorem}
In the complex setting, and for the Eulerian exponential, we would have the analogue but deceiving result: the minimal algebraic extension of $\RR$ which is uniformised by $z\mapsto e^z$ is $\CC$. Theorem \ref{local-class-field} confirms that in some sense, function field arithmetic is more transparent and allows to see more structure in the watermark.
We need the next:
\begin{Lemma}\label{property-of-Fn} Let $n$ be a non-negative integer. For every $r\in|\CC_\infty^\times|$ with $r<|l_n|$ 
the product
$$\mathcal{F}_n:=\cdots\left(1-\frac{\tau}{l_{n+1}^{q-1}}\right)\left(1-\frac{\tau}{l_{n}^{q-1}}\right)\in K[[\tau]]$$ defines an entire function $\CC_\infty\rightarrow\CC_\infty$ and 
induces an isometric bi-analytic isomorphism of the disk $D_{\CC_\infty}(0,r)$.
\end{Lemma}

\begin{proof}
This is easy to verify by using Proposition \ref{prop-goss} and Corollary \ref{imageofdisks}. Indeed,
if we set $$\psi_m:=1-\frac{\tau}{l_{m}^{q-1}},\quad m\geq 0$$ we see that for all $z\in\CC_\infty$ such that $|z|<|l_n|$, $\psi_m(z)=z+z'$ with $z'\in\CC_\infty$ depending on $m$ and $|z'|<|z|$, for all $m\geq n$.
\end{proof}

\begin{proof}[Proof of Theorem \ref{local-class-field}]
We have a well defined $\FF_q$-linear map $\exp_A:K_\infty^{\operatorname{ab}}\rightarrow K_\infty^{\operatorname{ab}}$. We first show that this map is surjective so that if $L$ is the minimal 
algebraic extension of $K_\infty$ which is uniformised by $\exp_A$, then
$L\subset K_\infty^{\operatorname{ab}}$. To do this, we note that we have, for all $n\geq 0$, a well defined $\FF_q$-linear algebraic map $E_n:\mathbb{A}^1_{K_\infty^{\operatorname{ab}}}\rightarrow\mathbb{A}^1_{K_\infty^{\operatorname{ab}}}$ given by the Carlitz polynomials ($\mathbb{A}^n_L$ denotes the affine space of dimension $n$ over a field $L$). By the proof of Proposition
\ref{prod-facto}, $E_n$ is surjective. Indeed, for all $y'\in K_\infty^{\operatorname{ab}}$, the splitting field of the polynomial $E_n(X)-y'\in K_\infty(y')[X]$ is an abelian extension of $K_\infty(y')$ which can be constructed by iterating Artin-Schreier extensions. Let $x$ be an element of $K_\infty^{\operatorname{ab}}$. There exists $n\geq 0$ such that 
$|x|<|l_n|$. By Lemma \ref{property-of-Fn}, $\mathcal{F}_n^{-1}(x)\in K_\infty^{\operatorname{ab}}$ is well defined.
Let $x'\in K_\infty^{\operatorname{ab}}$ be such that
$$l_nE_n(x')=\mathcal{F}_n^{-1}(x).$$
Then we have, by  Proposition \ref{prod-facto},
$\exp_A(x')=\mathcal{F}_n(l_nE_n(x'))=\mathcal{F}_n(\mathcal{F}_n^{-1}(x))=x$ and we have proved that $K_\infty^{\operatorname{ab}}$ is uniformised by $\exp_A$.
Now let $L\subset \CC_\infty$ be an algebraic extension of $K_\infty$ that is uniformised by 
$\exp_A$. To show that $L[\lambda_\theta]$ contains $K_\infty^{\operatorname{ab}}$ we proceed in two steps. 

In the first step, we show that $K_\infty^{\operatorname{un}}$, the maximal abelian extension of $K_\infty$ which is unramified at the $\infty$-place, is contained in $L$. To do this it suffices to show that the algebraic closure $\FF_q^{\operatorname{ac}}$ of $\FF_q$ in $\CC_\infty$ is contained in $L$. Indeed, it is easy to see that $K_\infty^{\operatorname{un}}=\FF_q^{\operatorname{ac}}((\frac{1}{\theta}))$.

By using Proposition \ref{prop-goss} we see that for every $y\in\CC_\infty$ such that $|y|=1$
there exists a unique $x\in\CC_\infty$ with $|x|=1$, such that $\exp_A(x)=y$, and of course if 
$y\in L$, then $x\in L$ because we have supposed that $L$ is uniformised by $\exp_A$. Since $\FF_q\subset K_\infty\subset L$, if $y\in\FF_q^\times$, there exists $x\in L$, $|x|=1$, such that $\exp_A(x)=y$. Now observe with Proposition \ref{prod-facto} that $\exp_A(x)=(\mathcal{F}_1\circ\mathcal{E}_1)(x)=y$ and applying Lemma \ref{property-of-Fn}
$$x-x^q=\mathcal{E}_1(x)=\mathcal{F}^{-1}(y)=y+y'$$ where $y'\in\frac{1}{\theta}\FF_q[[\frac{1}{\theta}]]$. Setting $x'=\sum_{i\geq 0}(y')^{q^i}\in\frac{1}{\theta}\FF_q[[\frac{1}{\theta}]]$ we deduce that $x-x'\in\FF_{q^2}\setminus\FF_q\subset\CC_\infty$ is an element of $L$, and $\FF_{q^2}\subset L$. This shows that $\FF_{q^2}((\frac{1}{\theta}))\subset L$ because 
$\FF_{q^2}((\frac{1}{\theta}))=\FF_q((\frac{1}{\theta}))[\FF_{q^2}]$. We can of course repeat this argument with $y\in\FF_{q^2}\subset L$ etc. to show that, inductively, 
$\FF_{q^d}\subset L$ for all $d\geq 1$ so that $\FF_{q^d}((\frac{1}{\theta}))=K_\infty[\FF_{q^d}]\subset L$ for all $d\geq 1$ and with a little additional work we conclude that $K_\infty^{\operatorname{un}}\subset L$.

Before passing to the second step we need a little bit of terminology. We say that a sequence $(x_i)_{i\geq 0}$ in $K_\infty^{\operatorname{ab}}$ is a {\em Lubin-Tate sequence} if $\frac{1}{\theta}x_0+x_0^q=0$ and 
$$\frac{1}{\theta}x_i+x_i^q=x_{i-1},\quad i>0.$$
We note that $x_0\lambda_\theta\in\FF_q^\times$. Similarly, we say that a sequence 
$(y_i)_{i\geq 0}$ of $K_\infty^{\operatorname{ab}}$ is an {\em Artin-Schreier sequence} if $y_0=1$
and
$$\mathcal{E}_1(y_i)=y_i-y_i^{q}=\theta y_{i-1},\quad i>0.$$ By a simple application of Proposition \ref{prop-goss} we see that $|y_i|=|\theta|^{\frac{1}{q}+\cdots+\frac{1}{q^i}}$ for all $i>0$. Moreover,
$$\frac{1}{\theta}x_0y_i+(x_0y_i)^q=x_0y_{i-1},\quad i>0$$ so that if $(y_i)_{i\geq 0}$ is an Artin-Schreier sequence and $x_0$ satisfies the previous equation, then $(x_0y_i)_{i\geq 0}$ is a Lubin-Tate sequence and if $(x_i)_{i\geq 0}$
is a Lubin-Tate sequence with $x_0\lambda_\theta=1$, then $(\frac{x_i}{x_0})_{i\geq 0}$ is an Artin-Schreier sequence.

The second step of the proof of our theorem is to show that $L$ contains an Artin-Schreier sequence. First of all, we note that for any Artin-Schreier sequence $(y_i)_{i\geq 0}$,
$\theta y_i\in D_{K_\infty^{\operatorname{ab}}}(0,r)$ for all $r\in|\CC_\infty^\times|$ such that
$r<|\theta|^{\frac{q}{q-1}}$ so that
$|\theta y_i|<|l_1|$ for all $i\geq 0$. We fix $i\geq 0$. Let $a_{i+1}\in K_\infty^{\operatorname{ab}}$
be such that
$$a_{i+1}-a_{i+1}^q=\mathcal{F}_1^{-1}(\theta y_i).$$ We have that 
$$\exp_A(a_{i+1})=\mathcal{F}_1(\mathcal{F}^{-1}(\theta y_i))=\theta y_i.$$ Since by hypothesis,
$L$ is uniformised by $\exp_A$, we have that $a_{i+1}\in L$ if $y_i\in L$. It is easy to see that 
$\mathcal{F}_1^{-1}(\theta y_i)=\theta y_i+y_i'$ where $|y_i'|<1$. In particular, $a_{i+1}'=\sum_{j\geq 0}(y_i')^{q^j}$ converges to an element of $L$ such that $a_{i+1}'-(a_{i+1}')^q=y_i'$. If we set $b_{i+1}=a_{i+1}-a_{i+1}'$ we can conclude, under the hypothesis that
$y_i\in L$, that $b_{i+1}\in L$ is such that 
$$b_{i+1}-b_{i+1}^q=\theta y_i.$$ By induction over $i\geq 0$ we obtain that $L$ contains an 
Artin-Schreier sequence $(y_i)_{i\geq 0}$.

We can now conclude the proof of the theorem. By what written earlier, $L[\lambda_\theta]$
contains a Lubin-Tate sequence $(x_i)_{i\geq 0}$. We set $\widetilde{K}:=K_\infty[x_i:i\geq 0]$. By Lubin-Tate theory (see \cite{LUB&TAT}) $K_\infty^{\operatorname{ab}}$ is the compositum in 
$\CC_\infty$ of $\widetilde{K}$ and $K_\infty^{\operatorname{un}}$ and therefore, 
$L[\lambda_\theta]$ contains $K_\infty^{\operatorname{ab}}$.
\end{proof}

We are not going to deepen the facts outlined below, but
the main theorem of local class field theory asserts, in the special case of our local field $K_\infty$
(it holds for any local field with appropriate modifications) the existence of an isomorphism of profinite groups
$$\widehat{\theta}_{K_\infty}:\widehat{K}^\times_\infty\rightarrow\operatorname{Gal}(K_\infty^{\operatorname{ab}}/K_\infty),$$ the {\em local Artin homomorphism}, where $\widehat{K}^\times_\infty$ is the profinite group completion of $K_\infty^\times\cong \FF_q[[\frac{1}{\theta}]]^\times\times \ZZ$, non-canonically isomorphic to the profinite group $\FF_q[[\frac{1}{\theta}]]^\times\times\widehat{\ZZ}$. The non-canonical isomorphism depends on the choice of a uniformiser $\pi$ of $K_\infty$. If we set
$K_\pi$ to be the subfield of $K_\infty^{\operatorname{ab}}$ which is fixed by $\widehat{\theta}_{K_\infty}(\pi)\in \operatorname{Gal}(K_\infty^{\operatorname{ab}}/K_\infty)$, then 
$K_\infty^{\operatorname{ab}}$ is the compositum $K_\pi K_\infty^{\operatorname{un}}$, 
and we have isomorphisms $\operatorname{Gal}(K_\infty^{\operatorname{un}}/K_\infty)\cong\widehat{\ZZ}$ and 
$\operatorname{Gal}(K_\pi/K_\infty)\cong\FF_q[[\frac{1}{\theta}]]^\times$.
Choosing a Lubin-Tate sequence in $K_\infty^{\operatorname{ab}}/K_\infty$ is therefore equivalent to
the choice of a uniformiser $\pi$ of $K_\infty$. One can see, along these remarks (but we will not give full details), that
the minimal algebraic extension $L\subset K_\infty^{\operatorname{ab}}$ of $K_\infty$ that is uniformised by $\exp_A$ is determined by $\operatorname{Gal}(K_\infty^{\operatorname{ab}}/L)\cong\FF_q^\times$.

\begin{Problem}\label{problem2}
{\em The notion of minimal field extension of $K_\infty$ which is uniformised by the exponential $\exp_A$ can be generalised to e.g. Drinfeld $A$-modules via Theorem \ref{drinfeld-theorem} in a natural way, but it is unclear how this field can be characterised in the light of local class field theory so that the role of a statement like Theorem \ref{local-class-field} must be clarified in this more general setting.}
\end{Problem}

\section{Topology of the Drinfeld upper-half plane}\label{Drinfeld-upper-half-plane}

We go back to the settings and notations of \S \ref{our-main-settings}, considering the $\FF_q$-algebra $A=H^0(\mathcal{C}\setminus\{\infty\},\mathcal{O}_{\mathcal{C}})$ with $\mathcal{C}$ a smooth projective curve over $\FF_q$ and $\infty$ a closed point. We therefore have the tower of inclusions of $\FF_q$-algebras $A\subset K\subset K_\infty \subset \CC_\infty$. 
In this section we give an explicit topological description of what is called the {\em Drinfeld upper-half plane} $\Omega$. It goes back to Drinfeld, in \cite{DRI}. 
D. Goss called it the {\em 'algebraist's upper-half plane'} in \cite{GOS1}. It can be viewed as an analogue of the complex upper-half plane that can be constructed by cutting $\CC$ in two along the real line and taking one piece only. As a set, $\Omega$ is very simple:
$$\Omega=\CC_\infty\setminus K_\infty,$$
but subtracting $K_\infty$ results in a different operation than cutting; this is what we are going to show here.
We begin by presenting some elementary properties following \cite{GER&PUT}. We recall that $\CC_\infty=\widehat{K_\infty^{\operatorname{ac}}}$, where $K_\infty=\FF((\pi))$ for some uniformiser $\pi$.
First of all, there is an action of $\GL_2(K_\infty)$ on $\Omega$ by homographies. If $\gamma=(\begin{smallmatrix} a & b \\ c & d\end{smallmatrix})\in\GL_2(K_\infty)$, then we have the automorphism of $\PP_{\FF_q}^1(\CC_\infty)$ uniquely defined by $$z\mapsto\gamma(z):=\frac{az+b}{cz+d}$$ if 
$z\not\in\{\infty,-\frac{d}{c}\}$. Observe that if $F/L$ is a field extension, then $\GL_2(L)$ acts by homographies on the set $F\setminus L$. For instance, $\GL_2(\RR)$ acts on $\CC\setminus\RR=\mathcal{H}^+\sqcup\mathcal{H}^-$ (disjoint union of the complex upper- and lower-half planes). 

It is well known that the imaginary part $\Im(z)$ of a complex number $z$, the {\em distance} of $z$ from the real axis, is submitted to the following transformation rule under the action by homographies. If $\gamma=(\begin{smallmatrix} a & b \\ c & d\end{smallmatrix})\in\GL_2(\RR)$:
\begin{equation}\label{modularity-abs}
\Im(\gamma(z))=\frac{\Im(z)\det(\gamma)}{|cz+d|^2},\quad z\in\CC\setminus\RR.
\end{equation} There is an analogous notion of
{\em distance from $K_\infty$} in $\CC_\infty$. We set:
$$|z|_\Im:=\inf\{|z-x|:x\in K_\infty\},\quad z\in\CC_\infty.$$
We have the following result.

\begin{Proposition}\label{lemmadisk}
(1) For all $z\in\CC_\infty$, $|z|_\Im$ is a minimum, and $|z|_\Im=0$ if and only if $z\in K_\infty$. (2) Let $z$ be an element of $\Omega$. Then, there exist
$z_0=\pi^{m}(\alpha_0+\cdots+\alpha_n\pi^{-n})\in\FF_q[\pi,\pi^{-1}]$ and $z_1\in\Omega$ with
$|z_1|=|z_1|_\Im<|\pi|^{m}$, uniquely determined, with $n\in\NN\cup\{-\infty\}$ and $\alpha_0\neq 0$ if $n\neq -\infty$, such that
$z=z_0+z_1$.
\end{Proposition}

\begin{proof} (1)
If $z\in K_\infty$, there is nothing to prove. Assume thus that $z\in\Omega\subset\CC_\infty$ is fixed.
Define the map $K_\infty\xrightarrow{f}|\CC_\infty^\times|$, $f(x)=|z-x|$. Then, $f$ is locally constant, hence continuous. But $K_\infty$ is locally compact so there is $x_0\in D_{K_\infty}(0,|z|)$ (not uniquely determined) such that $f(x_0)$ is a minimum (note that if $|x|>|z|$, then $f(x)=|z|$) and $|z|_\Im=|z-x_0|$. 

(2) For all $x\in K_\infty$,
$|x|>|z|$, we have $|z-x|=|x|$. Then, we have two cases.

(a). For all $x\in D_{K_\infty}(0,|z|)$,
$|z-x|=|z|$. In this case, $|z|_\Im=|z|$ and $|z|_\Im$ is a minimum. We thus get $n=-\infty$, $z_0=0$ and $z=z_1$.

(b). There exists $x\in D_{K_\infty}(0,|z|)\setminus\{0\}$
such that $|z|=|x|$ and $|z-x|<|z|$. This implies that 
the image of $z/x$ in the residue field of $\CC_\infty$ is
$1$. We can therefore write $z=\lambda_1\pi^{-n_1}+\eta_1$ with $\lambda_1\in\FF$ and $\eta_1\in\Omega$,
$|\eta_1|<|z|=|\theta|^{n_1}$.

We can iterate by studying now $\eta_1$ at the place of $z$. Either the procedure stops and we get a decomposition
$z=\lambda_1\pi^{-n_1}+\cdots+\lambda_k\pi^{-n_k}+\eta_k$ with $n_1>\cdots>n_k$, $|z|_\Im=|\eta_k|=|\eta_k|_\Im$ and there exists $z_0\in K_\infty$ such that $|z-z_0|=|z|_\Im>0$ as claimed in the statement,
or the procedure does not stop but in this case we have $z\in K_\infty$ which is excluded.
\end{proof}

In particular, either $|z_1|\not\in|K_\infty^\times|$, or
$|z_1|=|\pi^{m}|$ but the image of $z_1\pi^{-m}$ in the residue field of $\CC_\infty$ is not one of the elements of $\FF^\times$. Part (2) of Proposition \ref{lemmadisk}
implies that for all $x=z_0+y$ with $y\in D_{K_\infty}(0,|z_1|)$, $|z-x|=|z|_\Im=|z_1|=|z_1|_\Im$.

We also have the following elementary consequences of the above proposition. First of all, if $c\in K_\infty$, then $|cz|_\Im=|c||z|_\Im$ for all $z\in\Omega$. Moreover, if $v_\infty(z)\not\in\ZZ$, then $|z|_\Im=|z|$. Also, if $|z|=1$, we have $|z|_\Im=1$ if and only if the image of $z$ in the residue field of $\CC_\infty$ is not in $\FF$.

 The next property is the analogous of (\ref{modularity-abs}):

\begin{Lemma}\label{invariance} For all $z\in\Omega$ and $\gamma=(\begin{smallmatrix} * & * \\ c & d\end{smallmatrix})\in\GL_2(K_\infty)$, 
$$|\gamma(z)|_\Im=\frac{|\det(\gamma)||z|_\Im}{|cz+d|^2}.$$
\end{Lemma}

\begin{proof}
First of all, suppose that we have proved that 
\begin{equation}\label{intermediategamma}
|\gamma(z)|_\Im\leq\frac{|\det(\gamma)||z|_\Im}{|cz+d|^2},\quad \forall\gamma=(\begin{smallmatrix} * & * \\ c & d\end{smallmatrix})\in\GL_2(K_\infty),\quad \forall z\in\Omega.
\end{equation}
In particular, for all $\widetilde{z}\in\Omega$, and with $\gamma$ replaced by $\gamma^{-1}=\delta^{-1}(\begin{smallmatrix} * & * \\ -c & a\end{smallmatrix})$ (where $\delta=\det(\gamma)$),
we get
$$|\gamma^{-1}(\widetilde{z})|_\Im\leq\frac{|\delta||\widetilde{z}|_\Im}{|-c\widetilde{z}+a|^2}.$$
We set $\widetilde{z}=\gamma(z)$. Then, $-c\widetilde{z}+a=\frac{\delta}{cz+d}$ and therefore,
$$|z|_\Im=|\gamma^{-1}(\widetilde{z})|_\Im\leq|\widetilde{z}|_\Im|\delta|\left|\frac{cz+d}{\delta}\right|^2=|\delta|^{-1}|cz+d|^2|\widetilde{z}|_\Im=|\delta|^{-1}|cz+d|^2|\gamma(z)|_\Im,$$
so that
$$\frac{|\delta||z|_\Im}{|cz+d|^2}\leq|\gamma(z)|_\Im,$$
and we get the identity we are looking for.
All we need is therefore to show that (\ref{intermediategamma}) holds.

Now, let $x\in\CC_\infty$ be such that $x$ is not a pole of $\gamma$. An easy calculation shows that
$$\gamma(z)-\gamma(x)=\frac{\det(\gamma)(z-x)}{(cz+d)(cx+d)}.$$
Hence, if $x\in K_\infty$ is not a pole of $\gamma$,
we have
\begin{equation}\label{equality}
|\gamma(z)-\gamma(x)|=\frac{|\det(\gamma)||z-x|}{|cz+d|^2}\frac{|cz+d|}{|cx+d|}.
\end{equation}
We can find $x\in K_\infty$ such that $|z-x|=|z|_\Im$ and with the property that $x$ is not a pole of $\gamma$ (we have noticed that there are infinitely many such elements).
We claim that 
$|cx+d|\leq |cz+d|$. If $c=0$ this is clear. Otherwise, if this were false we would have $|cx+d|>|cz+d|$ and
$$|c||z|_\Im=|c||z-x|=|cz+d-(cx+d)|=|cx+d|>|cz+d|_\Im=|c||z|_\Im$$ which would be impossible.
Hence, with the claim in mind, we deduce from (\ref{equality}):
$$|\gamma(z)|_\Im\leq|\gamma(z)-\gamma(x)|\leq\frac{|\det(\gamma)||z-x|}{|cz+d|^2}=\frac{|\det(\gamma)||z|_\Im}{|cz+d|^2}$$ by our choice of $x$ and we are done.
\end{proof}

\subsection{Rigid analytic spaces}

The notion of rigid analytic space originates in ideas of Tate in the years 1960'. We do not want to go in very precise details because there is already a plethora of important references, among which \cite{BGR,FP}. A more recent introduction to rigid analytic spaces is the chapter `Several approaches to non-archimedean geometry' by Conrad, see \cite[Chapter 2]{CON} (the whole volume is close, in many aspects, to the topics of the present text). Important is also Berkovich's viewpoint which is outlined in this volume, \cite{POI&TUR}. We discuss, in a rather informal way, the nature of these structures before making use of some very particular special cases. Let $L$ be a field with valuation $|\cdot|$, complete, algebraically closed.

We are going to describe a rigid analytic space over $L$ (or analytic space over $L$) as a triple
$$(X,G,\mathcal{O}_X)$$
where $X$ is a non-empty set, $G$ a Grothendieck topology on $X$, $\mathcal{O}_X$ a sheaf, satisfying several natural conditions. A {\em Grothendieck topology} $G$ on $X$ can be 
outlined as a set $\mathcal{S}$ of subsets $U$ of $X$ and, for all $U\in G$, a `covering' $\operatorname{Cov}(U)$ of $U$ again by elements of $G$. If $\mathcal{C}$ is the family of all such coverings (\footnote{Do not mix up with the curve $\mathcal{C}$ of the previous sections.}), then  $G$ is the datum $(\mathcal{S},\mathcal{C})$
and the quality of being a Grothendieck topology results in a collection of properties we shall not give here, refining the simpler notion of topology (see \cite{FP} for the precise collection of conditions). If a Grothendieck topology $G=(\mathcal{S}, \mathcal{C})$ on $X$ is given, then the elements of $\mathcal{S}$ are called the {\em admissible subsets} of $X$ and the elements of $\mathcal{C}$ are called the {\em admissible coverings}. This refines the notion of topology because if we forget the coverings,
the conditions we are left on $\mathcal{S}$ are precisely those of a topology on $X$ so that right at the beginning we could have said that $X$ is a topological space, and the admissible sets are just the open sets for this topology. We have of course a corresponding notion of 
morphism of Grothendieck's topological spaces which strengthens that of continuous maps of topological spaces: pre-images of admissible sets (resp. coverings) are again admissible.

What is a {\em sheaf} on a Grothendieck topological space? If we choose a ring $R$, a sheaf $\mathcal{F}$ of $R$-algebras ($R$-modules\ldots) is a contravariant functor from $\mathcal{S}$ (with inclusion) to the category of $R$-algebras (or $R$-modules\ldots this is called a {\em pre-sheaf}) which satisfy certain compatibility conditions. For instance, if $f,g\in\mathcal{F}(U)$, $U\in\mathcal{S}$
and $f|_V=g|_V$ for all $V\in\operatorname{Cov}(U)\in\mathcal{C}$, then $f=g$. Furthermore,
if we choose $\operatorname{Cov}(U)=(U_i)_{i\in I}\in\mathcal{C}$ and for all $i$, $f_i\in\mathcal{F}(U_i)$ are such that $f_i|_{U_i\cap U_j}=f_j|_{U_i\cap U_j}$, then there exists a 'continuation' $f\in\mathcal{F}(U)$ with $f|_{U_i}=f_i$ for all $i$ (this is an abstract formalisation of 'analytic continuation'). Every pre-sheaf can be embedded in a sheaf canonically, but checking that a given pre-sheaf is itself a sheaf might result in subtle problems. The datum of $(X,G,\mathcal{F})$ with $G$ a Grothendieck topology
and $\mathcal{F}$ a sheaf of $R$-algebras on $(X,G)$ is called a {\em Grothendieck ringed space of $R$-algebras} and there is a natural notion of morphism of such structures which mimics the more familiar notion of morphism of ringed spaces of algebraic geometry.
Say for commodity that $X,Y$ are two Grothendieck topological spaces with respective sheaves $\mathcal{F}$ and $\mathcal{G}$, then a morphism of Grothendieck ringed spaces of $R$-algebras
$$(X,\mathcal{F})\xrightarrow{(f,f^\sharp)}(Y,\mathcal{G})$$
is the datum of a morphism of Grothendieck topological spaces $f$ and for all $U\subset Y$ admissible, an $R$-algebra morphism $f^\sharp:\mathcal{G}(U)\rightarrow\mathcal{F}(f^{-1}(U))$. So far, we discussed Grothendieck topological spaces, sheaves etc. But now, what is a rigid analytic variety? A rigid analytic variety over $L$, our valued field, complete, algebraically closed (say, $L=\CC_\infty$, the most relevant in our notes), is a particular kind of Grothendieck ringed space; let us see how. We still need a few more tools.
We have the unit disk
$$D_L(0,1)=\{z\in L:|z|\leq 1\}$$
playing the role of a basic brick for constructing rigid analytic spaces, just as the affine line does for algebraic varieties. For this reason, we focus on 
affinoid algebras. An {\em affinoid algebra over $L$} is any quotient of a {\em Tate algebra}
$$\TT_n(L)=\widehat{L[\underline{t}]}_{\|\cdot\|}$$
by an ideal, where the Tate algebra $\TT_n(L)$ of dimension $n$ is the completion $\widehat{\cdot}$ of the polynomial ring $L[\underline{t}]$ in $n$ indeterminates $\underline{t}=(t_i)_{1\leq i\leq n}$
for the Gauss valuation $\|\cdot\|$ that we recall it is defined, for elements $a_{i_1,\ldots,i_n}\in L$, by $\|\sum_{i_1,\ldots,i_n}a_{i_1,\ldots,i_n}t_1^{i_1}\cdots t_n^{i_n}\|=\sup|a_{i_1,\ldots,i_n}|$. It is known that it is
noetherian, with unique factorization, of Krull dimension the number of variables $n$. The resulting quotient $\mathcal{A}$ of $\TT_n(L)$ (by an ideal)
is endowed with a structure of $L$-Banach algebra. In other words, the Gauss norm of 
$\widehat{L[\underline{t}]}$ induces a (sub-multiplicative) norm on $\mathcal{A}$, and it is complete. In fact, any $L$-Banach algebra $\mathcal{A}$ together with a continuous epimorphism $\TT_n(L)\rightarrow\mathcal{A}$ for some $n$, making $\mathcal{A}$ into a finitely generated $\TT_n(L)$-algebra, is an affinoid algebra. 
Affinoid algebras over $L$ are the basic bricks to construct a rigid analytic variety.

The maximal spectrum $\operatorname{Spm}(\mathcal{A})$ of an affinoid $L$-algebra $R$ can be made into a Grothendieck ringed space $(X,G,\mathcal{F})$ over $\mathcal{A}$; this is called an {\em affinoid variety} over $L$. If $X=\operatorname{Spm}(R)$ and $Y=\operatorname{Spm}(R')$, an $L$-algebra morphism $R\rightarrow R'$ defines a morphism of ringed spaces $Y\rightarrow X$ which is called a {\em morphism of affinoid algebras}. This serves to describe 
the other pieces of $(X,G,\mathcal{F})$. The admissible sets 
in $\mathcal{S}$ (recall that $G=(\mathcal{S},\mathcal{C})$) are exactly the images in $X$ of {\em open immersions} of affinoid varieties and similarly, we define the coverings in $\mathcal{C}$. This gives rise to a Grothendieck topology $G$ on $X=\operatorname{Spm}(\mathcal{A})$. Furthermore, we have the pre-sheaf $\mathcal{O}_X$ defined by associating to $U\subset X$ an admissible set the $L$-algebra $\mathcal{O}_X(U)=R'$ where $U=\operatorname{Spm}(R')$. Thanks to {\em Tate's acyclicity theorem}  one shows that this is in fact a sheaf (see \cite{TAT}, see also \cite[Theorem 4.2.2]{FP}). This result was generalised by Grauert and Gerritzen \cite[7.3.5 and 8.2]{BGR}). {\em Dulcis in fundo}, we have: 

\begin{Definition}
{\em A Grothendieck ringed space $X=(X,G,\mathcal{F})$ is a {\em rigid analytic variety over $L$} if $X$ has an admissible covering of admissible subsets $U$ which have the property that 
$(U,\mathcal{F}|_U)$ is an affinoid variety over $L$ for all $U$.}
\end{Definition}

\subsubsection{Analytification}\label{analytification}
An important process to construct rigid analytic spaces is the analytification of an algebraic variety.
Let $X/L$ be a scheme of finite type. The {\em analytification} $X^{an}$ of $X$ is a rigid analytic space over $L$ that can be defined by an affinoid covering starting from the geometric data as follows. We consider 
affine Zariski open subsets $U=\operatorname{Spec}(\mathcal{A})\hookrightarrow X$ and embeddings $U\hookrightarrow \mathbb{A}^N_L$ which correspond, on the algebraic side, to surjective $L$-algebra maps $L[\underline{t}]\rightarrow\mathcal{A}$ (where $\underline{t}$ denotes the set of independent variables $t_1,\ldots,t_N$) endowing $\mathcal{A}$ with a structure of $L[\underline{t}]$-algebra, for some $N$. Taking the completion 
for the Gauss valuation yields a surjective morphism:
$$\widehat{L[\underline{t}]}\rightarrow\mathcal{A}\otimes_{L[\underline{t}]}\widehat{L[\underline{t}]}$$
which gives rise to a map $V:=\operatorname{Spm}(\mathcal{A}\otimes_{L[\underline{t}]}\widehat{L[\underline{t}]})\hookrightarrow D_{L}(0,1)^N=\operatorname{Spm}(\widehat{L[\underline{t}]})$. We can proceed similarly for polydisks of different radii in $|L^\times|$ and this is used to construct
a rigid analytic space $U^{an}$
such that $V=U^{an}\cap D_{L}(0,1)^N$. Glueing, we construct the rigid analytic space $X^{an}$. For example, the {\em rigid affine line over $L$},
$\mathbb{A}_L^{1,an}$ is obtained by glueing together the rigid analytic spaces $D_L(0,r)$ along the inclusions with $r\in|L^\times|$. Similarly, the {\em rigid projective line over $L$}, $\PP^{1,an}_L$, can be constructed by glueing two copies of $D_L(0,1)$ along the set $\{z\in L:|z|=1\}$, or also glueing two copies of $\mathbb{A}^{1,an}_{L}$, see also Berkovich's construction in \cite[Definition II.1.5]{POI&TUR}. The {\em Berkovich's affine line} is described in detail in ibid. See \cite[Definition I.1.1]{POI&TUR}.

Rigid analytification defines a functor, called the 'GAGA functor' from the category of $L$-schemes  of finite type to the category of rigid analytic spaces over $L$. Note that we can also consider analytifications of morphisms, coherent sheaves etc. Finally, there is an alternative way to define the analytification functor over an affine variety $X$ over $L$, introduced by Berkovich, which makes the underlying topological space particularly easy to compute as it is defined over the set of multiplicative seminorms over the coordinate ring of $X$ satisfying certain compatibility conditions with the valuation of $L$. See \cite[Definition II.1.1]{POI&TUR} for the construction of the Berkovich spectrum of an algebra of finite type over $L$. See also Temkin's \cite[Chapter 1]{TEM} for a nice survey in the area.

\subsubsection{The rigid analytic variety $\Omega$} 
We now focus on $L=\CC_\infty$ with $A=H^0(\mathcal{C}\setminus\{\infty\},\mathcal{O}_{\mathcal{C}})$ in our usual notation. We discuss a structure of rigid analytic space over $\CC_\infty$ on $\Omega=\CC_\infty\setminus K_\infty$.
Note that 
$$\Omega=\bigcup_{M>1}U_M,$$
where $U_M=\{z\in\Omega:M^{-1}\leq |z|_\Im\leq |z|\leq M\}$,
the filtered union being over the elements $M\in|\CC_\infty|\setminus|K_\infty|$ with $M>1$.
Observe now:
\begin{Lemma}\label{constructions-affinoid} With $M\in|\CC_\infty|\setminus|K_\infty|$ we have 
$$U_M=D_{\CC_\infty}(0,M)\setminus\bigsqcup_{\begin{smallmatrix}\lambda\in\FF[\pi,\pi^{-1}]\\ \lambda=\lambda_{-\beta}\pi^{\beta}+\cdots+\lambda_{\beta}\pi^{-\beta}\\
1\leq|\pi|^{-\beta}\leq M\end{smallmatrix}}D_{\CC_\infty}^\circ(\lambda,M^{-1}).$$
\end{Lemma}
\begin{proof}
This easily follows from the fact that $K_\infty$ is locally compact in combination with the ultrametric inequality.
\end{proof}
Hence, $U_M$ is admissible and carries a structure of affinoid variety $U_M=\operatorname{Spm}(\mathcal{A}_M)$ where $\mathcal{A}_M$ is an integral affinoid algebra. We say that 
$U_M$ is a {\em connected affinoid} of $\PP^{1,an}_{\CC_\infty}$ (as in the language introduced in \cite{FP}, motivated by the integrality of $\mathcal{A}_M$).
In particular $\Omega$ can be covered (in fact filled) with connected affinoids and the analytic structure of $\Omega$ arises from viewing it as
the complementary in $\CC_\infty$ of smaller and smaller disks located over certain elements of $K_\infty$ which is close to the familiar view that we have also for the set $\CC\setminus\RR$. 
This gives the Grothendieck topology on $\Omega$, and 
the sheaf $\mathcal{O}_\Omega$ is that of {\em rigid analytic functions} over $\Omega$. Practically, a rigid  analytic function $f:\Omega\rightarrow\CC_\infty$ is a function such that the restriction
on every set $U_M$ is the uniform limit of a sequence of rational functions on $U_M$ without poles in $U_M$.

\subsection{Fundamental domains for $\Gamma\backslash\Omega$}

This subsection is motivated by an essential construction in the theory of Schottky groups, that of 
fundamental domains. Schottky groups have been first introduced by Schottky in 1877 in the complex setting; they are useful to analytically uniformise compact Riemann surfaces. In the years 1970, after the work of Tate on $p$-adic uniformisation of elliptic curves with split multiplicative reduction, Mumford discovered how to $p$-adically uniformise smooth projective curves of genus $g\geq 2$ with `split degenerate stable reduction' by using $p$-adic Schottky groups $\Gamma$ acting on non-archimedean variants $\Omega_\Gamma$ of the classical complex upper-half plane. The reader is encouraged to read the modern contribution of Poineau-Turchetti to this volume \cite{POI&TUR}. An older reference is \cite{GER&PUT}; it also contains determinant tools to explore this profound theory. Consequently, we will not give all the details, this would bring us too far away from our path.

Let us recall that, given a local field $L$ with valuation $|\cdot|$, the group $\operatorname{PGL}_2(L)$ (\footnote{Projective linear group over $L$, defined as the quotient of $\operatorname{GL}_2(L)$ by its center.}) acts
on the rigid analytic projective line $\PP^{1,\operatorname{an}}_F$ where $F$ is the completion of an algebraic closure of $L$ (see \cite{GER&PUT,FP}). 
A {\em Schottky group} over $L$ is a finitely generated subgroup $\Gamma$ of $\operatorname{PGL}_2(L)$ which is discrete and such that no element but the identity has finite order. Schottky groups are free (see \cite[Theorem (3.1)]{GER&PUT}) this being an important consequence of the fact that they act freely on certain rigid analytic spaces. Every Schottky group $\Gamma$ over $L$ has a compact {\em limit set} $\mathcal{L}_\Gamma\subset\PP^{1,\operatorname{an}}_{F}$ so that $\Gamma$ acts freely over
$\Omega_\Gamma:=\PP^{1,\operatorname{an}}_{F}\setminus\mathcal{L}_\Gamma$. 
The quotient space $\Gamma\backslash\Omega_\Gamma$ naturally carries a structure of rigid analytic space over $L$ which is associated with a smooth, geometrically connected, projective curve $X_\Gamma$ over $L$, of genus $g$ the rank of $\Gamma$. We learn from \cite[Theorem (4.3)]{GER&PUT} that 
every Schottky group $\Gamma$ in $\operatorname{PGL}_2(L)$ admits a {\em good fundamental domain} $\mathfrak{F}_\Gamma$. Without entering the details, for every element $z\in\Omega_\Gamma$ the set of $\gamma\in\Gamma$ such that $\gamma(z)\in\mathfrak{F}_\Gamma$ is non-empty and finite.
In fact, if $\gamma\in\Gamma$, then $\mathfrak{F}_\Gamma\cap\gamma(\mathfrak{F}_\Gamma)\neq\emptyset$ if and only if $\gamma\in\{1,\gamma_1^{\pm1},\ldots,\gamma_g^{\pm1}\}$, 
where $\gamma_1,\ldots,\gamma_g$ freely generate $\Gamma$.
Moreover, $\mathfrak{F}_\Gamma$ can be written as
$$\mathfrak{F}_\Gamma:=\PP^{1,\operatorname{an}}_{F}\setminus\bigsqcup_{i=1}^{2g}D_i$$
where the $D_i$'s are the rigid analytic spaces associated to disks $D_{F}^\circ(a_i,r_i)=\{z\in F:|z-a_i|<r_i\}$ with $r_i\in|L^\times|$ for all $i$, such that the disks $D_{F}(a_i,r_i)=\{z\in F:|z-a_i|\leq r_i\}$
are pairwise disjoint. One can therefore see easily that $\mathfrak{F}_\Gamma$ carries a structure of rigid analytic variety over $F$ (read also \cite[\S II.3.1]{POI&TUR} along with its more general settings and the theory of uniformisation of Mumford curves). 

The interesting point in this discussion is that if we set $L=K_\infty=\FF((\pi))$, $A=H^0(\mathcal{C}\setminus\{\infty\},\mathcal{O}_{\mathcal{C}})\subset K_\infty$, $F=\CC_\infty$ etc. the group
$\operatorname{PGL}_2(A)$ acts on $\Omega=\PP^{1,an}_F\setminus\PP^{1,an}_{K_\infty}$ but the action is in general not free; there usually are elliptic points (this happens, for instance, when $[\FF:\FF_q]$ is odd, see \cite{MAS&SCH}). Even more seriously, the group itself is not finitely generated (see Serre's book \cite{SER} for more details), so that $\operatorname{PGL}_2(A)$ {\em is not} a Schottky group. 

\subsubsection{Some structural properties of $\Gamma=\GL_2(A)$} For the purposes of the present paper, we will be content to study the case in which $\mathcal{C}$ has genus $0$, so that
in Lemma \ref{the-space-V} we have $V=\{0\}$ and therefore, $K_\infty\cong A\oplus \mathcal{M}_{K_\infty}$. It is easy to see that there exists a uniformiser $\pi$ of $K_\infty$ such that $A=\FF[\pi^{-1}]$. We can indeed choose $\pi=\theta^{-1}$ where $\theta$ is any element of $A$ with a simple pole at $\infty$. In particular, $A=\FF[\theta]$.

It is not difficult to show that the group $\GL_2(\FF[\theta])$ is generated by its subgroups
 $\GL_2(\FF)$ (finite) and the Borel subgroup $B(*)=\{(\begin{smallmatrix} * & * \\ 0 & *\end{smallmatrix})\}$. In fact, a Theorem of Nagao in \cite{NAG} asserts
 that, given a field $k$ and an indeterminate $t$,
 \begin{equation}\label{nagao}
 \GL_2(k[t])=\GL_2(k)*_{B(k)}B(k[t]),\end{equation}
 where $*_{B(k)}$ denotes the {\em amalgamated product} along $B(k)$, which is by definition the quotient of the free product $\GL_2(k)*B(k[t])$ by the normal subgroup generated by those elements arising from the natural identifications existing between the elements of $B(k)*1$ and $1*B(k)$ coming from the maps $$\GL_2(k)\rightarrow\GL_2(k)*B(k[t])\leftarrow B(k[t])$$ (a gluing along compatibility conditions). Note that $B(k[t])$ is not finitely generated, so that $\GL_2(k[t])$ is not finitely generated 
(this is trivial if $k$ is infinite) in contrast with a theorem of Livingston, asserting that $\GL_n(k[t])$
is finitely generated if $n\geq 3$, and also with the more familiar result that
$\operatorname{SL}_2(\ZZ)\cong\ZZ/2\ZZ*\ZZ/3\ZZ$ so that it is, in particular, finitely presented.

\begin{Corollary}
$\operatorname{PGL}_2(\FF[\theta])$ is not a Schottky group.
\end{Corollary}

\subsubsection{Bruhat-Tits trees and `good fundamental domains'}\label{Tits}

We recall that $K_\infty=\FF((\pi))$ for a uniformiser $\pi$, with $\FF$ a finite extension of $\FF_q$. Our first task is to 
describe a combinatorial structure which allows to `move inside' $\Omega$, the 
Bruhat-Tits tree; in practice, we can `move along annuli'. The second task, in the case $A=\FF[\theta]$, is to construct a subset of $\Omega$ that we can qualify as a `good fundamental domain' for the homography action of $\GL_2(A)$
over $\Omega$, being understood that $\GL_2(A)$ is not a Schottky group. 

We recall that if $x\in\CC_\infty$, $D_{\CC_\infty}^\circ(x,r)=\{z\in\CC_\infty:|z-x|<r\}$.
Let $S$ be a subset of $\CC_\infty^\times$ such that if $x,x'\in S$ are distinct,
$|x-x'|=\max\{|x|,|x'|\}$. Then, with $x\in S$, the sets $$D_x:=D_{\CC_\infty}^\circ(x,|x|)=x+D_{\CC_\infty}^\circ(0,|x|)$$ are pairwise distinct subsets of $\CC_\infty^\times$. Indeed, clearly, they do not contain $0$.
Moreover, if $x\neq x'$ we have $y\in D_{\CC_\infty}^\circ(x,|x|)\cap D_{\CC_\infty}^\circ(x',|x'|)$ if and only if we can find $z\in D_{\CC_\infty}^\circ(x,|x|), z'\in D_{\CC_\infty}^\circ(x',|x'|)$, such that 
$y=x+z=x'+z'$ with $|z|<|x|$ and $|z'|<|x'|$, so that $|z-z'|<\max\{|x|,|x'|\}$. This means that $
\max\{|x|,|x'|\}>|z-z'|=|x'-x|=\max\{|x|,|x'|\}$ which is impossible. 

We choose, for any element $r\in\ZZ_{>1}$, an element, denoted by $\pi^{\frac{1}{r}}\in\CC_\infty$,
with the property that $(\pi^{\frac{1}{r}})^r=\pi$, which exists because $\CC_\infty$ is algebraically closed. The set $\Sigma:=\{(\pi^{\frac{1}{r}})^s\}$ inherits the total order of $\RR$ by $\frac{s}{r}\in\QQ\subset\RR$ (\footnote{Thanks to Lemma \ref{mapping-c-infty} we can even additionally suppose that the elements $\pi^{\frac{1}{r}}$ are chosen in such a way that $\Sigma=\pi^\QQ$ is a subgroup of $\CC_\infty^\times$.}). We have observed (after Lemma \ref{lemma-sep}) that the valuation group of $|\cdot|$ is $|\pi|^\QQ$. Hence if $z\in\CC_\infty^\times$, we can find $r,s$ relatively prime, unique, such that $|z(\pi^{\frac{1}{r}})^s|=1$. Since the residue field of $\CC_\infty$ is $\FF^{ac}$, we obtain that there exists a unique
$\zeta\in(\FF^{ac})^\times$ such that $$|z-\zeta(\pi^{\frac{1}{r}})^s|<|z|.$$
If we set $S=\{\zeta(\pi^{\frac{1}{r}})^s:\zeta\in(\FF^{ac})^\times,r>1,s\in\ZZ\text{ such that }r,s\text{ are relatively prime}\}$, then for all $x,x'\in S$ distinct we have $|x-x'|=\max\{|x|,|x'|\}$ and we obtain a partition of $\CC_\infty^\times$:
\begin{equation}
\CC_\infty^\times=\bigsqcup_{x\in S}D_x.
\end{equation}
Let us now consider the subset
$$\widetilde{S}:=\{x\in S:|x|\not\in|\pi|^\ZZ\}\sqcup\{\zeta\pi^n:\zeta\in\FF^{ac}\setminus\FF,n\in\ZZ\}\subset S.$$
With it, we can still somehow reconstruct $\CC_\infty$. Indeed, the reader can easily see that
if $x\in\widetilde{S}$, $D_x\cap K_\infty=\emptyset$ and $$\CC_\infty=\left(K_\infty+\bigsqcup_{x\in\widetilde{S}}D_x\right)\sqcup K_\infty.$$
As a consequence we have
$$\Omega=K_\infty+\bigsqcup_{x\in\widetilde{S}}D_x$$
and
$$\bigsqcup_{x\in\widetilde{S}}D_x=\{z\in \Omega:|z|=|z|_\Im\}.$$

We observe that if $\lambda\in\QQ\setminus\ZZ$, then
$$\bigsqcup_{\begin{smallmatrix}x\in S\\ |x|=|\pi|^\lambda\end{smallmatrix}}D_x=
\bigsqcup_{\begin{smallmatrix}x\in \widetilde{S}\\ |x|=|\pi|^\lambda\end{smallmatrix}}D_x=\{z\in\CC_\infty:|z|=x\}=:C_\lambda.$$
We also set, for $\lambda\in\ZZ$, $$C_\lambda:=\bigsqcup_{\zeta\in\FF^{ac}\setminus\FF}D_{\CC_\infty}^\circ(\zeta\pi^\lambda,|\pi|^\lambda).$$ 
Note that $C_\lambda=\{z\in\Omega:|z|=|z|_\Im=|\pi|^\lambda\}$ for all $\lambda\in\QQ$.
For all $\lambda$, the set
$C_\lambda$ is invariant by translation of elements in $D_{K_\infty}(0,|\pi|^{\lceil\lambda\rceil})$, where $\lceil\cdot\rceil$ denotes the smallest between the integers which are larger than $(\cdot)$. If $\alpha\in K_\infty\setminus D_{K_\infty}(0,|\pi|^{\lceil\lambda\rceil})=\oplus_{i\leq\lfloor\lambda\rfloor}\FF\pi^i$ (with $\lfloor\cdot\rfloor$ the largest integer which is smaller than $(\cdot)$) then
$C_\lambda\cap(\alpha+C_\lambda)=\emptyset$.
We have obtained the next result.
\begin{Lemma}\label{partition-omega}
The following partition of $\Omega$ holds:
$$\Omega=\bigsqcup_{\begin{smallmatrix}\lambda\in\QQ\\
\alpha\in K_\infty\setminus D_{K_\infty}(0,|\pi|^{\lceil\lambda\rceil})\end{smallmatrix}}\alpha+C_\lambda.$$
\end{Lemma}
Note that this can be very easily used to construct admissible coverings of $\Omega$.
The above is the crucial statement which allows to construct the Bruhat-Tits tree associated to 
$\Omega$. It relies on the existence of a natural partial ordering on the set $\mathcal{T}:=\{\alpha+C_\lambda:\alpha\in K_\infty\setminus D_{K_\infty}(0,|\pi|^{\lceil\lambda\rceil}),\lambda\in\QQ\}$.
We declare that $\alpha+C_\lambda\succ \alpha'+C_{\lambda'}$ if $\lambda<\lambda'$
and $\alpha'+C_\lambda=C_\lambda$ so that, for example, $\alpha+C_\lambda\succ C_\lambda$ if and only if $|\alpha|\leq|\pi|^\lambda$. Then $\mathcal{T}$ can be enriched with the structure of a {\em tree}, the {\em Bruhat-Tits tree}. We recall that a tree $\mathcal{T}$ is a metric space such that, on one side, for any distinct points $P,P'$ of $\mathcal{T}$ there exists one and only one topological arc in $\mathcal{T}$ of extremities $P,P'$ and, on the other side, this arc is isometric to an interval of $\RR$ (this definition is due to Tits). A tree has edges and vertices. 
The {\em vertices} of our Bruhat-Tits tree $\mathcal{T}$ are represented by the subsets $\alpha+C_\lambda$ of $\CC_\infty$ with $\lambda\in\ZZ$ and the {\em edges} are represented by real intervals
$]n-1,n[$ with $n\in\ZZ$, with the extremities given by a couple of vertices $(\alpha+C_{n-1},\alpha'+C_{n})$ such that $\alpha'+C_{n-1}=C_{n-1}$. The intervals are oriented and our tree itself acquires an orientation. The upper direction is that of the negative $\lambda$'s or, alternatively, of the larger $|z|_\Im$'s. The edges are therefore organised so that at every lower (for the ordering)
extremity the vertex is a $q^{d_\infty}+1$ branching point with $q^{d_\infty}$ edges below and one above (with respect to the orientation). 

 The next picture represents a small piece of $\mathcal{T}$ for $q^{d_\infty}=2$.

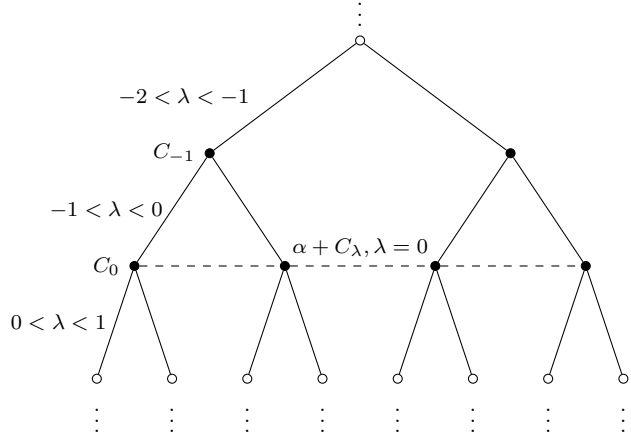
\tikzset{
  solid node/.style={circle,draw,inner sep=1.2,fill=black},
  hollow node/.style={circle,draw,inner sep=1.2},
}

\newcommand\payoff[1]{
  $\begin{pmatrix} #1 \end{pmatrix}$
}
\begin{center}
\begin{tikzpicture}[font=\footnotesize]
  \tikzset{
    level 1/.style={level distance=15mm,sibling distance=40mm},
    level 2/.style={level distance=15mm,sibling distance=20mm},
    level 3/.style={level distance=15mm,sibling distance=10mm},
    level 4/.style={level distance=15mm,sibling distance=5mm},
  }

  \node[hollow node,label=above:{\vdots}]{}
    child{node[solid node,label=left:{$C_{-1}$}]{}
      child{node(l1)[solid node,label=left:{$C_{0}$}]{}
        child{node[hollow node,label=below:{\vdots}]{}edge from parent node[left]{$0<\lambda<1$}}
        child{node[hollow node,label=below:{\vdots}]{}edge from parent node[right]{}}
        edge from parent node[left]{$-1<\lambda<0$}
      }
      child{node(l2)[solid node]{}
        child{node[hollow node,label=below:{\vdots}]{}edge from parent node[left]{}}
        child{node[hollow node,label=below:{\vdots}]{}edge from parent node[right]{}}
        edge from parent node[right]{}
      }
      edge from parent node[left,xshift=-10]{$-2<\lambda<-1$}
    }
    child{node[solid node,label=right:{}]{}
      child{node(r1)[solid node]{}
        child{node[hollow node,label=below:{\vdots}]{}edge from parent node[left]{}}
        child{node[hollow node,label=below:{\vdots}]{}edge from parent node[right]{}}
        edge from parent node[left]{}
      }
      child{node(r2)[solid node]{}
        child{node[hollow node,label=below:{\vdots}]{}edge from parent node[left]{}}
        child{node[hollow node,label=below:{\vdots}]{}edge from parent node[right]{}}
        edge from parent node[right]{}
      }
      edge from parent node[right,xshift=10]{}
    }
  ;

  \draw[dashed](l1)--(l2)--node[midway,above]{$\alpha+C_\lambda,\lambda=0$}(r1)--(r2);
\end{tikzpicture}
\end{center}

Also note that the euclidean closure of the image in $\mathcal{T}$ of any set $\alpha+\sqcup_{\lambda\in\QQ}C_\lambda$ for $\alpha\in K_\infty$ fixed is isometric to $\RR$ and
any two such sets, if distinct, have a upper half-line in common. Any element 
of $\Omega$ is $K_\infty$-translation equivalent to finitely many elements in $\sqcup_{\lambda\in\QQ}C_\lambda$ and finally, the homography action of $\operatorname{GL}_2(K_\infty)$ over $\Omega$ is compatible with a continuous action over $\mathcal{T}$ in a way that can be made completely explicit.

The structure of the spaces $\CC_\infty$ and $\Omega$ may look topologically very complicate but the Bruhat-Tits tree is some kind of `central nervous system' which allows to obtain a combinatorial 
picture of these spaces (or rather, their admissible coverings) and to move in their interior, by means of the {\em reduction map}, which is $\GL_2(K_\infty)$-equivariant 
$$\operatorname{red}:\Omega\rightarrow\mathcal{T},$$
defined by $z\mapsto \alpha+C_\lambda\in\mathcal{T}$ where $\alpha+C_\lambda$ is the unique element of the partition of Lemma \ref{partition-omega} such that $z\in\alpha+C_\lambda$. This presentation may look different, it is in fact essentially equivalent to that of Teitelbaum in \cite[Preliminaries]{TEI} (see also Teitelbaum's chapter in \cite{CON}). To help the reader to connect with the formalism of Teitelbaum, which also is that of \cite{GER&PUT}, note that the set $U(1)$ of \cite[p. 492]{TEI} plays the role of our disjoint union $\sqcup_{\lambda\in]-1,1[}C_\lambda$ and that the set $V$ introduced one page later is
equal to our $\sqcup_{\lambda\in]-1,0[}C_\lambda$. The sets $\gamma(U(1))$ for $\gamma\in\GL_2(K_\infty)$ define an admissible covering of $\Omega$ and $\mathcal{T}$ can be alternatively constructed defining edges and vertices by a criterion of overlapping for the various 
$\gamma(U(1))$'s and an identification between the set $\{\gamma(U(1)):\gamma\in\GL_2(K_\infty)\}$ and the quotient $\GL_2(\mathcal{O}_{K_\infty})\backslash
\GL_2(K_\infty)$, corresponding to the vertices.

We set $\mathfrak{F}:=\{z\in\Omega:|z|=|z|_\Im\geq 1\}$. By Lemma \ref{partition-omega},
we have $\mathfrak{F}=\sqcup_{\lambda\leq 0}C_\lambda$ and $\operatorname{red}(\mathfrak{F})$ is an upper half-line in $\mathcal{T}$. We deduce that
$$\mathfrak{F}=\CC_\infty\setminus\left(\bigsqcup_{\zeta\in\FF}D_{\CC_\infty}^\circ(\zeta,1)\sqcup\bigsqcup_{n\geq1}\bigsqcup_{\zeta\in\FF^\times}D_{\CC_\infty}^\circ(\zeta\pi^{-n},|\pi|^{-n})\right).$$
We now focus on the case $A=\FF[\theta]$. 
For  $z\in\Omega$ we denote by $\mathfrak{F}_z$ the set $\{z'\in\mathfrak{F}:\text{ there exists }\gamma\in\GL_2(A)\text{ such that }\gamma(z)=z'\}\subset\mathfrak{F}$.
We show:

\begin{Proposition}
For all $z$, the set $\mathfrak{F}_z$ is non-empty and finite.
\end{Proposition}

\begin{proof}
If $z\in\mathfrak{F}$ then there exists $x\in\widetilde{S}$ such that $|x|\geq 1$ and $z\in D_x$ and we see that the set of $a\in A$ such that $z-a\in\mathfrak{F}$ is finite. 
Note that $1/z\in D_{1/x}$ so that $1/z\not\in\mathfrak{F}$ (in fact, if $\gamma=(\begin{smallmatrix}0 & 1 \\ 1 & 0\end{smallmatrix})$, $\gamma(D_x)=D_{\gamma(x)}$). From Nagao's Theorem we deduce that the set $\{\gamma\in\GL_2(A):\gamma(z)\in\mathfrak{F}\}$ is finite so that, for all $z\in\Omega$,
$\mathfrak{F}_z$ is finite (but note that the cardinality is not uniformly bounded in terms of $z$). Let $z$ be in $\Omega$. If $|z|_\Im\geq 1$ there exists $a\in A$ such that $|z-a|=|z|_\Im$ 
and $\mathfrak{F}_z$ is non-empty. All we need to show is that if $z\in\Omega$ is such that
$|z|_\Im<1$, then there exists $\gamma\in\GL_2(A)$ such that $\gamma(z)\in\mathfrak{F}$.
To see this, there is no loss of generality in supposing that $|z|<1$. Indeed, we can replace $z$ with $z-a$ for $a\in A$. We can therefore write:
$$z=w+x+y$$
where $w\in\oplus_{i=1}^{\lfloor\lambda\rfloor}\FF\pi^i\in\mathcal{M}_{K_\infty}$, $x\in \widetilde{S}$ with $|x|=|\pi|^\lambda$ and $y\in D_x$. Applying $\gamma=(\begin{smallmatrix}0 & 1 \\ 1 & 0\end{smallmatrix})$ we see that $z$ is $\GL_2(A)$-equivalent to an element 
$z'\in\alpha'+C_{\lambda'}$ with $\lambda'$ such that $\lambda-\lambda'\in\ZZ_{>1}$ and 
$\alpha'\in\oplus_{i\leq \lfloor\lambda'\rfloor}\FF\pi^i$, so that, in particular, $|z'|_\Im>|z|_\Im$.
We can iterate this process with $z'$ playing the role of $z$. The fact that $\lambda-\lambda'\in\ZZ_{\geq1}$ implies that $z$ is $\GL_2(A)$-equivalent to an element of $\mathfrak{F}$
and $\mathfrak{F}_z$ is non-empty.
\end{proof}

This seems enough to allow us calling $\mathfrak{F}$ a `good fundamental domain' for $\Gamma\backslash\Omega$ with $A=\FF[\theta]$, even though it is undoubtedly not as well behaved as the good fundamental domains in the framework of Schottky groups. Note that
$\Gamma\backslash\mathcal{T}$ contains an `end': this metric space is not compact, but 
can be made compact with the addition of one point represented by one of the upper half-lines contained by $\mathcal{T}$ which, at the level of $\Gamma\backslash\Omega$, corresponds to a `cusp'.

Similar constructions are possible for $\Gamma=\GL_2(A)$ with a more general projective curve $\mathcal{C}$ but we do not describe them here. In this broader case it is possible to show that $\Gamma\backslash\mathcal{T}$ has the structure of a finite graph with 
finitely many ends attached to it. More general 'fundamental domains' can be constructed from the Bruhat-Tits tree of $\Omega$ and constructed
by Serre (see \cite[Theorem 10]{SER}) thanks to a more refined interpretation of the elements of 
$\Gamma\backslash\Omega$ as classes of rank two vector bundles over $\mathcal{C}$. We refer to ibid. for the details.

\subsection{An elementary result on translation-invariant functions over $\Omega$}\label{elementary-statement}
We recall that $\mathcal{H}$ denotes the complex upper-half plane. Let $f:\mathcal{H}\rightarrow\CC$ be a holomorphic function such that 
 for all $n\in\ZZ$ and for all $z\in\mathcal{H}$, $f(z+n)=f(z)$. Then, we can expand 
 $$f(z)=\sum_{n\in\ZZ}f_ne^{2\pi inz},\quad f_n\in\CC,$$
 a series which is convergent for $q(z)=e^{2\pi iz}$ in $$\dot{D}_{\CC}^\circ(0,1)=\{z\in\CC:0<|z|<1\}$$ the punctured open unit disk centered at $0$ of $\CC$ or equivalently, for $z$ in every horizontal strip of finite height in $\mathcal{H}$ (note that they are invariant by horizontal translation).

 \subsubsection{A digression}\label{digression}
The proof of the above statement for $f$ is simple and we can afford a short digression.
The function 
$z\mapsto q(z)$ does not allow a global holomorphic section $\mathcal{H}\leftarrow \dot{D}_\CC^\circ(0,1)$. But we can cover $\CC^\times$ with say,
three open half-planes $U_1,U_2,U_3$, and  there are sections $s_1,s_2,s_3$ defined and holomorphic over $U_1,U_2,U_3$ such that $s_i-s_j\in\ZZ$ over $U_i\cap U_j$ for all $i,j$.
Let $f$ be holomorphic on $\mathcal{H}$ such that $f(z+1)=f(z)$ for all $z\in\mathcal{H}$. Define $g_i(q)=f(s_i(q))$ for all $i=1,2,3$. Then, 
the compatibility conditions and the fact that the pre-sheaf of holomorphic functions over any open set is a sheaf (the well known principle of analytic continuation) ensure that this defines a holomorphic function $g(q)$ over 
 $\dot{D}^\circ_{\CC}(0,1).$ But the ring of holomorphic functions over $\dot{D}_\CC^\circ(0,1)$ is precisely that of the convergent double series $\sum_{n\in\ZZ}f_nq^n$, as one can easily see, and our claim follows. One also deduces that there is an isomorphism of Riemann's surfaces
 $$\mathcal{H}/\ZZ\cong\dot{D}^\circ_{\CC}(0,1)$$
 induced by $e^{2\pi iz}$, concluding the digression. 
 
 \medskip
 
 We now come back to our characteristic $p>0$ setting and we suppose, from now on, that 
 $$A=H^0(\PP_{\FF_q}^1\setminus\{\infty\},\mathcal{O}_{\PP_{\FF_q}^1}).$$
We note that $\Omega$ is invariant by translations of $a\in A$ and 
the function $$\exp_A(z)=z\prod_{a\in A\setminus\{0\}}\left(1-\frac{z}{a}\right)=\widetilde{\pi}^{-1}\exp_C(\widetilde{\pi}z)$$ is an entire function $\CC_\infty\rightarrow\CC_\infty$, $\FF_q$-linear, surjective, of kernel $A=\FF_q[\theta]$, hence also invariant by translations by elements of $A$. It is thus natural to ask for an analogue statement of the above, complex one.
Consider $R\in|\CC_\infty^\times|$. Now, note that $A$ acts on $\Omega_R$ by translations. Giving
$A\backslash\Omega_R$ the quotient topology we have:

\begin{Lemma}\label{homeo}
There is $S\in|\CC_\infty^\times|$ such that the function $\exp_A$ induces a homeomorphism of topological spaces
$$A\backslash\Omega_R\rightarrow\{z\in\CC_\infty:|z|\geq S\}.$$
\end{Lemma}

\begin{proof}
From the Weierstrass product expansion we see that, setting
$$S:=\max_{z\in D_{\CC_\infty}(0,R)}|\exp_A(z)|=:\|\exp_A\|_R=\|z\|_R\prod_{\begin{smallmatrix}a\in A\\ a\neq 0\end{smallmatrix}}\Big\|1-\frac{z}{a}\Big\|_R=R\prod_{\begin{smallmatrix}a\in A\\ a\neq 0\\ |a|<R\end{smallmatrix}}\frac{R}{|a|},$$
$\exp_A(D(0,R))=D(0,S)$ by Corollary \ref{imageofdisks}. Hence, $D^\circ(0,S)=D_{\CC_\infty}^\circ(0,S)=\exp_A(D^\circ(0,R))$ from which we deduce that $$\{z\in\CC_\infty:|\exp_A(z)|<S\}=A+D^\circ(0,R).$$
Recall that $K_\infty=A\oplus\mathcal{M}_{K_\infty}$. If $R\geq 1$, we have 
$D^\circ(0,R)\supset\mathcal{M}_{K_\infty}$. Now observe that 
$$\{z\in\CC_\infty:|z|_\Im<R\}=\cup_{a\in K_\infty}D^\circ(a,R)=\cup_{a\in A}D^\circ(a,R).$$ Therefore we have the chain of identities
$$A+D^\circ(0,R)=K_\infty+D^\circ(0,R)=\cup_{a\in K_\infty}D^\circ(a,R)=\{z\in\CC_\infty:|z|_\Im<R\}=\Omega\setminus\Omega_R,$$ and taking complementaries, we see that
$$\Omega_R=\{z\in\CC_\infty:|\exp_A(z)|\geq S\},\quad R\geq 1.$$
\end{proof}

\section{Some quotient spaces}\label{quotient-spaces}
Our topologies are totally disconnected and Lemma \ref{homeo} is weaker if compared with analogous statements in the complex setting.
Fortunately there is a structure of {\em quotient analytic space} over $\Omega_R/A$, and it is isomorphic to the analytic structure of the complementary of the disk $D^\circ(0,S)$. 

\subsection{$A$-periodic functions over $\Omega$}\label{periodic-functions}

We suppose that $A=\FF_q[\theta]$ all along this subsection.
The analogue for $\Omega=\CC_\infty\setminus K_\infty$ of the simple claim over $\CC$ of the beginning of \S \ref{elementary-statement} and the proof in \S \ref{digression} is not as easy to prove but it is true, and not too difficult.
In fact, the following result holds:

\begin{Proposition}\label{not-so-elementary}
Let $f:\Omega\rightarrow\CC_\infty$ be an analytic function such that $f(z+a)=f(z)$ for all $a\in A$. Then,
there exists $S\in|\CC_\infty^\times|$, $S<1$, such that
$$f(z)=\sum_{n\in \ZZ}f_n\exp_A(z)^n,\quad f_n\in\CC_\infty,$$
the series being uniformly convergent for $\exp_A(z)^{-1}$ in every annulus of $\dot{D}_{\CC_\infty}^\circ(0,S)=\{x\in\CC_\infty:0<|x|<S\}$, $S\in|\CC_\infty^\times|$, small enough.
\end{Proposition}

To prove this result and to motivate the proof we are giving, we need some preparation. 

\subsubsection{Analytification and quotients}
Let  $\mathcal{X}$ be a rigid analytic variety over a valued field $L$, complete and algebraically closed. Let us consider a group $\Gamma$ acting on $\mathcal{X}$ with 'admissible action'. 'Admissible action' means that 
$\mathcal{X}$ can be covered by $\Gamma$-stable admissible subsets and that $\Gamma$ acts through an embedding $\iota$ of $\Gamma$ in $\operatorname{Aut}(\mathcal{X})$, topological group, and the image is discrete. We are interested in such triples
$$(\mathcal{X},\Gamma,\iota).$$ For example, we can take $\Gamma=A$ acting on $\Omega$ or $\mathbb{A}^1_{\CC_\infty}$ by translations (the theme of Proposition \ref{not-so-elementary}) or $\Gamma=\GL_2(A)$ acting on $\Omega$ by homographies (the theme of the text). 

The quotient map $$\mathcal{X}\rightarrow\Gamma\backslash \mathcal{X}$$ can be used to 
define a structure of Grothendieck ringed space on the quotient $\Gamma\backslash \mathcal{X}$. 
A subset of $\Gamma\backslash \mathcal{X}$ is admissible if its pre-image is admissible, and the sections are $\Gamma$-invariant $\CC_\infty$-valued functions over pre-images of $\Gamma$-invariant subsets. One needs conditions under which the quotient acquires a structure of rigid analytic space. For example, a finite group $\Gamma$ acting on $\mathcal{X}=\operatorname{Spm}(\mathcal{A})$ affinoid variety which allows a covering by invariant admissible subsets gives rise to an isomorphism of affinoid varieties 
$\Gamma\backslash\operatorname{Spm}(\mathcal{A})\rightarrow\operatorname{Spm}(\mathcal{A}^\Gamma)$, where $\mathcal{A}^\Gamma$ is the sub-algebra of $\Gamma$-invariant 
elements of $\mathcal{A}$; see \cite[\S 6.3.3]{BGR}. See also Hansen's more general \cite[Theorem 1.3]{HAN}. 

We invoke the analytification functor in \S \ref{analytification} by choosing $\mathcal{X}=X^{an}$.
If $X$ is a scheme of finite type over $L$ with an 'admissible action' of a finite group $\Gamma$ 'admissible', now in the algebraic sense that there is a covering with $\Gamma$-invariant affine sub-schemes, it can be proved that there exists a unique scheme structure (of finite type over $L$) on the ringed quotient space $$p:X\rightarrow\Gamma\backslash X.$$
The following proposition is due to Amaury Thuillier: we warmly thank him for having brought our attention to it. 
\begin{Proposition}\label{thuillier}
The canonical map $\Gamma\backslash X^{an}\rightarrow(\Gamma\backslash X)^{an}$
is an isomorphism of rigid analytic varieties.
\end{Proposition}

\begin{proof}
We can suppose, without loss of generality,
$X=\operatorname{Spec}(\mathcal{A})$ affine, so that
$\Gamma\backslash X=\operatorname{Spec}(\mathcal{A}^\Gamma)$. 
In terms of algebras, we have (horizontal arrows are surjective and vertical arrows injective, and $\widehat{L[\underline{t}]}$ is  the standard Tate $L$-algebra in the variables $\underline{t}=(t_1,\ldots,t_N)$ for some $N$):
$$
\begin{tikzpicture}
         \matrix (m) [matrix of math nodes,row sep=3em,column sep=3em,minimum width=2em]{
              L[\underline{t}]& & \mathcal{A} \\
    \widehat{L[\underline{t}]} & & \displaystyle{\frac{\widehat{L[\underline{t}]}}{\operatorname{ker}(\pi)}}.\\   };
         \path[-stealth]
            ([yshift=-3pt]m-1-1.east) edge node [above,yshift=1.0ex] {$\pi$} ([yshift=-3pt]m-1-3.west);
            \path[-stealth]
            (m-1-1) edge node [left] {} (m-2-1);
            \path[-stealth]
            (m-1-3) edge node [right] {} (m-2-3);
            \path[-stealth]
            (m-2-1) edge node [below] {} (m-2-3);
        \end{tikzpicture}
$$

Then we have: $$\mathcal{A}_V:=\frac{\widehat{L[\underline{t}]}}{\operatorname{ker}(\pi)}=\mathcal{A}\otimes_{L[\underline{t}]}\widehat{L[\underline{t}]}=H^0(V,\mathcal{O}_{X^{an}})$$ 
where $V:=\operatorname{Spm}(\mathcal{A}\otimes_{L[\underline{t}]}\widehat{L[\underline{t}]})\subset(\Gamma\backslash X)^{an}$.

The $L$-algebra $\mathcal{B}=\mathcal{A}\otimes_{\mathcal{A}^\Gamma}\mathcal{A}_V$ is finite over $\mathcal{A}_V$, hence it inherits a structure of affinoid $L$-algebra. We deduce, with $p^{an}:X^{an}\rightarrow (\Gamma\backslash X)^{an}$ the analytification of $p$, that $W=(p^{an})^{-1}(V)$ is a $\Gamma$-invariant affinoid domain of $X^{an}$ and $\mathcal{A}_W=H^0(W,\mathcal{O}_{X^{an}})=\mathcal{B}$. 
The quotient space $\Gamma\backslash W$
is also affinoid, of algebra $\mathcal{B}^\Gamma$ (see \cite[\S 6.3.3]{BGR}). Therefore, all we need to show is that the canonical morphism 
$$\mathcal{A}_V\rightarrow \mathcal{B}=\mathcal{A}\otimes_{\mathcal{A}^\Gamma}\mathcal{A}_V$$
induces an isomorphism $\mathcal{A}_V\rightarrow \mathcal{B}^\Gamma=(\mathcal{A}\otimes_{\mathcal{A}^\Gamma}\mathcal{A}_V)^\Gamma$.

The morphism
$\mathcal{A}\rightarrow \mathcal{A}_V$ is flat \cite[Theorem 3.4.1, (ii)]{BER}. Therefore the exact sequence
$$0\rightarrow \mathcal{A}^\Gamma\rightarrow \mathcal{A}\xrightarrow{\oplus(g-\operatorname{Id}_\mathcal{A})}\bigoplus_{g\in \Gamma}\mathcal{A}$$
yields an exact sequence
$$0\rightarrow \mathcal{A}_V=\mathcal{A}^\Gamma\otimes_{\mathcal{A}^\Gamma}\mathcal{A}_V\rightarrow \mathcal{A}\otimes_{\mathcal{A}^\Gamma}\mathcal{A}_V\xrightarrow{\oplus(g-\operatorname{Id}_\mathcal{A})}\bigoplus_{g\in \Gamma}\mathcal{A}\otimes_{\mathcal{A}^\Gamma}\mathcal{A}_V.$$ We have thus that $\mathcal{A}_V$ is equal to the kernel of the last arrow, which is just $\mathcal{B}^\Gamma$.
\end{proof}

We consider $L=\CC_\infty$ and we denote by $A(n)$ the $\FF_q$-vector space $\{a\in A:|a|<|\theta|^n\}$ (dimension $n$ and cardinality $q^n$).
If $X=\mathbb{A}^{1}_{\CC_\infty}$ and we look at $\Gamma=A(n)$ acting on $X$ by translations, we have the quotient scheme
$\Gamma\backslash X=\operatorname{Spec}(\CC_\infty[x]^\Gamma)$. Note that $\CC_\infty[x]^\Gamma=\CC_\infty[E_n(x)]$ with $E_n$ characterised by Proposition \ref{propthreepropofek}, by Euclidean division. Proposition \ref{thuillier} applies. 

We introduce the sets for $n\geq 1$
$$\mathcal{B}_{n}=D_{\CC_\infty}^\circ(0,|\theta|^{n})\setminus\bigcup_{a\in A(n)}D_{\CC_\infty}^\circ(a,1).$$
We define, in parallel, with $l_n=(\theta-\theta^q)\cdots(\theta-\theta^{q^n})$: $$\mathcal{C}_{n}=D_{\CC_\infty}^\circ(0,|l_n|)\setminus D_{\CC_\infty}^\circ(0,1).$$
Each of these sets has an admissible covering by affinoid subsets so that it is a rigid analytic sub-variety of $\mathbb{A}_{\CC_\infty}^{1,an}$. A function $f:\mathcal{B}_{n}\rightarrow\CC_\infty$ is analytic if its restriction to every affinoid subset is analytic. Note that $\mathcal{B}_n\subset\mathcal{B}_{n+1}$ and $\mathcal{C}_n\subset\mathcal{C}_{n+1}$ for all $n\geq 1$.
We set $$\psi_m:=1-\frac{\tau}{l_{m}^{q-1}},\quad m\geq 0$$ (recall that $\tau(x)=x^q$ for $x\in\CC_\infty$). It is easy to see that $\psi_n$ induces an isometric biholomorphic isomorphism of $\mathcal{C}_m$ for all $n\geq m$. In particular
the non-commutative infinite product
$$\mathcal{F}_n:=\cdots\left(1-\frac{\tau}{l_{n+1}^{q-1}}\right)\left(1-\frac{\tau}{l_{n}^{q-1}}\right)\in K[[\tau]]$$
induces an isometric biholomorphic isomorphism of $\mathcal{C}_n$ (for every $n$).

In a similar vein, Proposition \ref{thuillier} implies:

\begin{Corollary}
The function $\mathcal{E}_n=l_nE_n$ is a degree $q^n$ \'etale covering $\mathcal{B}_n\rightarrow\mathcal{C}_n$
which induces an isomorphism of rigid analytic spaces
$$A(n)\backslash\mathcal{B}_n\rightarrow\mathcal{C}_n,$$ where 
the analytic structure on the pre-image is induced by the analytification of $\operatorname{Spec}(\CC_\infty[x]^{A(n)})$.
\end{Corollary}

\subsubsection{Proof of Proposition \ref{not-so-elementary}}
A global section $g_n$ of $\mathcal{O}_{\mathcal{C}_n}$ can be identified, in a unique way, with a convergent series
$$\sum_{k\in\ZZ}g^{(n)}_kx^k,\quad g^{(n)}_k\in\CC_\infty.$$
Let $f:\Omega\rightarrow\CC_\infty$ be a rigid analytic function with the property that for all $a\in A$, $f(z+a)=f(z)$. We fix $m>0$, let $n$ be such that $n\geq m.$ Then, $f:\mathcal{B}_n\rightarrow\CC_\infty$ is holomorphic
such that $f(z+a)=f(z)$ for all $a\in A(n)$ and therefore there exists a unique $g_n\in\mathcal{O}_{\mathcal{C}_n}$ such that $f(z)=g_n(\mathcal{E}_n(z))$
over $\mathcal{C}_n$ and 
we can write:
$$f(z)=\sum_{k\in\ZZ}g^{(n)}_k(\mathcal{E}_n(z))^k.$$
We observe that $\mathcal{B}_m\subset\mathcal{B}_n$. 
Thus, we have the following commutative diagram for $n>m$, where the left vertical arrows are the identity, and the bottom right vertical arrow is $\psi_m$, while the top one is $\psi_{m+1,n}$,
where $\psi_{m,n}$ is the composition $\psi_{m,n}:=\psi_{n-1}\circ\cdots\circ\psi_m$:
$$\begin{array}{ccc}\mathcal{B}_m  & \xrightarrow{\mathcal{E}_n} & \mathcal{C}_m\\
\uparrow & & \uparrow \\
\mathcal{B}_m  & \xrightarrow{\mathcal{E}_{m+1}} & \mathcal{C}_m\\
\uparrow & & \uparrow \\
\mathcal{B}_m  & \xrightarrow{\mathcal{E}_{m}} & \mathcal{C}_m,\\
\end{array}$$
and there also exists a unique 
$g_m\in\mathcal{O}_{\mathcal{C}_m}$ such that $f(z)=g_m(\mathcal{E}_m(z))$, this time over $\mathcal{C}_m\subset\mathcal{C}_n$ so that, noticing that $\psi_{m,n}$ induces an isometric biholomorphic isomorphism of $\mathcal{C}_m$, we must have:
$$g^{(n)}(\psi_{m,n}(x))=g^{(m)}(x),\quad x\in\mathcal{C}_m.$$
In particular, we have the equality
$$g^{(n+1)}(\psi_{n}(x))=g^{(n)}(x),\quad x\in\mathcal{C}_m.$$ Since
$\psi_n(x)=x(1-(\frac{x}{l_n})^{q-1})$ and 
$\psi_n(x)^k=x^k(1+\sigma_{n,k}(x))$ with $|\sigma_{n,k}(x)|\leq|\frac{x}{l_n}|^{q-1}<1$ for all $n\geq m$, $k\in\ZZ$, we deduce that the function
$g^{(n+1)}-g^{(n)}$ tends to zero uniformly on every admissible subset of $\mathcal{C}_m$, for $n\geq m$. This means that the sequence of functions $(g^{(n)})_{n\geq m}$ converges to an element $g=\sum_kg_kx^k\in\mathcal{O}_{\mathcal{C}_m}$ uniformly on every admissible subset of $\mathcal{C}_m$.

With this new function $g$ the existence of which is given by Cauchy convergence criterion, we can write:
$$g^{(m)}(x)=g(\mathcal{F}_m(x)),\quad x\in\mathcal{C}_m.$$
We use the results of \S \ref{factorization-property} and more precisely Proposition \ref{prod-facto}, or with a more manageable notation, 
(\ref{noncommutativeidentity}). We thus recall the identity of entire functions:
$$\exp_A=\mathcal{F}_n\underbrace{\left(1-\frac{\tau}{l_{n-1}^{q-1}}\right)\cdots\left(1-\frac{\tau}{l_1^{q-1}}\right)\left(1-\tau\right)}_{\mathcal{E}_n}.$$
In particular, by uniqueness:
$$f(z)=g(\exp_A(z)),\quad z\in\mathcal{B}_m,\quad \forall m.$$
Since the sets $\mathcal{B}_n$ cover the set $\Omega_1:=\{z\in\CC_\infty:|z|_\Im\geq1\}$, the result follows.

Restated in more geometric, but essentially equivalent langage, the arguments of the proof of Proposition \ref{not-so-elementary} lead to:
\begin{Proposition}\label{rigid-isom}
For all $M\in[1,\infty[\cap|\CC_\infty^\times|$, the function $z\mapsto \frac{1}{\exp_A}$ yields an isomorphism of rigid analytic spaces
$A\backslash\Omega_M\cong \dot{D}_{\CC_\infty}(0,S)=D_{\CC_\infty}(0,S)\setminus\{0\}$ for some $S\geq1$ depending on $M$.
\end{Proposition}

\begin{Problem}\label{problem3}
{\em The above proof, although simple, is longer than the one we gave in the digression \ref{digression}, in the complex case. This leads to the following question: is it possible to construct explicitly an admissible covering $(U_i)_i$ of an annulus $D_{\CC_\infty}(0,R)\setminus D_{\CC_\infty}^\circ(0,r)$ and local inverses $g_i\in\mathcal{O}_{U_i}$ of the function $\exp_A$ or even better, the function $\frac{1}{\exp_A}$, delivering a simpler proof of Proposition \ref{not-so-elementary} and making no use of the process of analytification?}
\end{Problem}

Also, note that the fact that the Grothendieck ringed space $A\backslash\mathbb{A}^{1,an}$ carries a structure of rigid analytic variety and much more general results in this vein can be also  deduced from Simon H\"aberli's thesis \cite[Proposition 2.34]{HAB}.

\subsubsection{The Bruhat-Tits tree and $\exp_A$} As a complement for the previous discussions, in this subsection we describe how the Bruhat-Tits tree of \S \ref{Tits} can be used to study the function $\exp_A$. We are going to see that somewhat, $\exp_A$ defines a covering
$\mathbb{A}^{1,an}_{\CC_\infty}\rightarrow\mathbb{A}^{1,an}_{\CC_\infty}$ `ramified of 
degree $q^\infty$'; the reader is invited to compare with the results of \S \ref{local-class}. To give more strength to this, we use again Proposition \ref{prod-facto}. We are therefore led to analyse the image of $\mathcal{E}_n=l_nE_n$ on $D_{\CC_\infty}^\circ(0,|\theta|^n)$ and then, take the limit for $n\rightarrow\infty$. We note that 
$$D_{\CC_\infty}^\circ(0,|\theta|^n)\setminus K_\infty=\bigsqcup_{\begin{smallmatrix}\lambda\in\QQ\cap]-n,\infty[\\ \alpha\in\oplus_{-\lambda\leq i<n}\FF_q\theta^i\end{smallmatrix}}\alpha+C_\lambda.$$
Since
$$\mathcal{E}_n(z)=\frac{l_n}{d_n}\prod_{a\in A(n)}(z-a)$$
is $\FF_q$-linear of kernel $A(n)$, it suffices to study how $\mathcal{E}_n$ behaves over 
$$\mathcal{T}_n=\bigsqcup_{\begin{smallmatrix}\lambda\in\QQ\cap]-n,\infty[\\ \alpha\in\oplus_{-\lambda\leq i<0}\FF_q\theta^i\end{smallmatrix}}\alpha+C_\lambda.$$
Note that if $\lambda\leq 0$, the direct sum over $i$ is empty. This means that in the Bruhat-Tits tree $\mathcal{T}$, $\mathcal{T}_n$ entails a very simple subtree which can be obtained by 
glueing in $0$ a segment $]-n,0]$ (the subtree $\mathcal{T}_n^-$) with the union of $q$ disjoint copies 
of a complete $q$-ary tree equating $\mathcal{T}_0^+=\operatorname{red}(D_{\CC_\infty}^\circ(0,1)\setminus K_\infty)$ (independent of $n$), so that $\mathcal{T}_n=\mathcal{T}^-_n\sqcup\mathcal{T}^+_0$ and $\mathcal{T}_0=\mathcal{T}_0^+$. Since $\mathcal{E}_n$ induces an isometric isomorphism of $D^\circ_{\CC_\infty}(0,1)$ such that for all $z\in D^\circ_{\CC_\infty}(0,1)$, $\mathcal{E}_n(z)=z+z'$ with $|z'|<|z|$, it induces the identity on $\mathcal{T}^+_0$, and this, for all $n\geq 0$. The action of the maps $\mathcal{E}_n$ are all equal to the action of
$\mathcal{E}_0(z)=z$ on $\mathcal{T}^+_0$. We now choose $n>0$ and we look at 
the behaviour of $\mathcal{E}_n$ on $\mathcal{T}^-_n$, which is the most interesting part of the story.

Consider $x$ such that $\operatorname{red}(x)\in\mathcal{T}^-_n$. Then, there exists $i>0$ maximal with the property that $x\in\mathcal{T}^-_i\setminus\mathcal{T}^-_{i-1}$ ($\mathcal{T}^-_0$ is empty by definition) and there exists a unique $\lambda\in\QQ$ with $-\lambda\in[i-1,i[$ such that $x\in C_\lambda$. 
We recall that $\prod_{0\neq a\in A(n)}a=\frac{d_n}{l_n}$, see (\ref{a-product}).
We have:
\begin{eqnarray*}
\mathcal{E}_n(x)&=&\frac{l_n}{d_n}\prod_{a\in A(i)}(x-a)\prod_{a\in A(n)\setminus A(i)}(x-a)\\
&=&\mathcal{E}_i(x)\frac{l_n}{d_n}\frac{d_i}{l_i}\prod_{a\in A(n)\setminus A(i)}(-a)\prod_{a\in A(n)}\left(1-\frac{x}{a}\right)\\
&=&(1+\xi)\mathcal{E}_i(x),
\end{eqnarray*}
where $\xi\in D^\circ_{\CC_\infty}(0,1)$ (because $\frac{|x|}{|a|}<1$ for all $a\in A(n)\setminus A(i)$). If $y\in D^\circ_{\CC_\infty}(0,|x|)$ we get
$$\mathcal{E}_n(x+y)=\mathcal{E}_i(x)+\underbrace{\xi\mathcal{E}_i(x)+(1+\xi)\mathcal{E}_i(y)}_{\text{element of }C^\circ_{\CC_\infty}(0,|\mathcal{E}_i(x)|)}.$$
We deduce that the map
$$D^\circ_{\CC_\infty}(x,|x|)\xrightarrow{\mathcal{E}_n}D^\circ_{\CC_\infty}(\mathcal{E}_i(x),|\mathcal{E}_i(x)|)$$
is an \'etale covering of degree $q^i$. Hence, the image by $\mathcal{E}_n$ of $\operatorname{res}^{-1}(\mathcal{T}^-_i\setminus\mathcal{T}^-_{i-1})$ (annulus) is an \'etale covering of degree $q^i$
of the annulus
$$\frac{l_i}{d_i}\Big[D^\circ_{\CC_\infty}(0,|\theta|^{iq^i})\setminus D^\circ_{\CC_\infty}(0,|\theta|^{(i-1)q^{i}})\Big]=D^\circ_{\CC_\infty}(0,|l_i|)\setminus D^\circ_{\CC_\infty}(0,|l_{i-1}|).$$
From this it is not difficult to deduce that $\mathcal{E}_n$ defines a covering 
$D^\circ_{\CC_\infty}(0,|\theta|^n)\rightarrow D^\circ_{\CC_\infty}(0,|l_n|)$ ramified of degree $q^n$ at the points of $A(n)$ and \'etale on the complementary of these points but we get even more.
Namely, that for any $z\in \Omega$, $\operatorname{res}(\exp_A(z))$ can be very easily computed. If $|z|_\Im<1$ then $|\exp_A(z)|<1$ and if $z\not\in K_\infty$, $\operatorname{res}(\exp_A(z))$ is equal to $\operatorname{res}(z-a)$ where $a\in A$ is the unique element such that $\operatorname{res}(z-a)\in\mathcal{T}^+_0$. If $|z|_\Im\geq 1$ then 
$\operatorname{res}(\exp_A(z))=\operatorname{res}(\mathcal{E}_n(z))$ for all but finitely many $n$ (depending on how large is $|z|_\Im$).

We consider $\mathcal{T}^-_\infty=\cup_{n\geq 1}\mathcal{T}^-_n$ (homeomorphic to $\RR_{\leq 0}$) and $\mathcal{T}_\infty=\mathcal{T}^-_\infty\sqcup\mathcal{T}^+_0$.  Note that $\operatorname{res}(\mathfrak{F})=\mathcal{T}^-_\infty$ and $\operatorname{res}(\mathfrak{F}\sqcup\{z\in\Omega:|z|,|z|_\Im<1\})=\mathcal{T}_\infty$. In the terminology of \S \ref{Tits}, $\mathfrak{F}\sqcup\{z\in\Omega:|z|,|z|_\Im<1\}$ can be viewed as a `good fundamental domain' for the action of $A$ over $\Omega$ by translations. We ultimately get, with a few more details to develop which are left to the reader:

\begin{Proposition}
The map $\exp_A$ induces a surjective, $A$-periodic map $\mathfrak{F}\rightarrow\mathbb{A}^{1,an}_{\CC_\infty}\setminus D^\circ_{\CC_\infty}(0,1)$ and rigid analytic isomorphisms
$A\backslash\mathfrak{F}\rightarrow \mathbb{A}^{1,an}_{\CC_\infty}\setminus D^\circ_{\CC_\infty}(0,1)$ and $A\backslash\mathbb{A}^{1,an}_{\CC_\infty}\rightarrow \mathbb{A}^{1,an}_{\CC_\infty}$.\end{Proposition}

Note that $\mathfrak{F}$ is not, properly speaking, invariant by $A$-translations, but $A$-translations define an equivalence relation on $\mathfrak{F}$. The above statement needs to be interpret in the light of the richer combinatorial structure described earlier. In the classical setting we have, of course, the classical well known properties that the Eulerian exponential $z\mapsto e^z$ induces analytic isomorphisms $\ZZ\backslash\mathcal{H}\rightarrow D_{\CC}^\circ(0,1)$
and $\ZZ\backslash\CC\rightarrow \CC^\times$. Interestingly too, we note that, just as
$\CC=\mathcal{H}\sqcup\RR\sqcup\mathcal{H}^-$ (the latter is the lower complex half-plane), here we have an analogous decomposition $$\CC_\infty=\Omega\sqcup K_\infty=\Omega_1\sqcup \Omega^-\sqcup K_\infty$$ with $\Omega_1=\{z\in\Omega:|z|_\Im\geq 1\}$, $\Omega^-=\{z\in\Omega:|z|_\Im< 1\}$.

We hope that, with this description, we have convinced the reader that the functions $\exp_A$ and the Carlitz's exponential carry an extraordinary structural richness. We now complete our discussion with the quick exposition of some properties of the 
quotient $\GL_2(A)\backslash\Omega$ and then we move our attention to Drinfeld modular forms.

\subsection{The quotient $\GL_2(A)\backslash\Omega$}\label{quotient-A}

 In the previous subsection we gave, in the most explicit way, but also in compatibility with the purposes of this text, a description of the analytic structure of the quotient space ($A=\FF_q[\theta]$ acting by translations) $A\backslash\Omega_1$. 
Following \cite[Chapter 10]{GER&PUT}), we now describe the action of $\GL_2(\FF_q)$ on certain admissible subsets of $\Omega$. We consider $M\in|\CC_\infty^\times|$ and we set
$$\Omega_M:=\{z\in\Omega:|z|_\Im\geq M\}.$$
Note that this set, which is called {\em horocycle neighbourhood of $\infty$}, is non-empty and is invariant by translations by elements of $K_\infty$. The multiplication by elements of $\FF_q^\times$ induce bijections of $\Omega_M$. Here is a lemma that will be useful later.

\begin{Lemma}\label{with-this-help}
If $M>1$ and if $\gamma\in\GL_2(A)$ is such that
$\gamma(\Omega_M)\cap\Omega_M\neq\emptyset$, then
$\gamma$ belongs to the Borel subgroup $(\begin{smallmatrix}* & * \\ 0 & * \end{smallmatrix})$ of $\GL_2(A)$.
\end{Lemma}

\begin{proof}
Let $\gamma=(\begin{smallmatrix}a & b \\ c & d \end{smallmatrix})\in\GL_2(A)$.
By Lemma \ref{invariance}, $|\gamma(z)|_\Im=\frac{|z|_\Im}{|cz+d|^2}$.
Let us suppose that 
$z,\gamma(z)\in\Omega_M$, and that $c\neq 0$. 
Then, since $|c|\geq 1$ if $c\in A\setminus\{0\}$,
$$|cz+d|\geq|cz+d|_\Im=|c||z|_\Im\geq |z|_\Im.$$ Then,
$\gamma(z)\in\Omega_M$ implies that $|z|_\Im\geq M|cz+d|^2\geq M|z|_\Im^2$ so that $M^{-1}\geq |z|_\Im$. Now, if $M>1$, from $|z|_\Im\geq M$ we get a contradiction.
\end{proof}

We set, with $M\in|\CC_\infty^\times|\cap]1,\infty[$:
 $$\mathcal{D}_M:=D_{\CC_\infty}(0,M)\setminus(\FF_q+D_{\CC_\infty}^\circ(0,M^{-1}))\subset\Omega.$$ This is the complementary in $\PP_{\FF_q}^{1,an}(\CC_\infty)$ of the union of $q+1$ disjoint disks and is an affinoid subset of $\Omega$. In the following, we can choose $M=|\theta|^{\frac{1}{2}}$. It is easy to see that the group $\GL_2(\FF_q)$ acts by homographies
 on $\mathcal{D}_M$ (note that more generally, the subsets $\{z\in\CC_\infty:|z|\leq q^n,|z|_\Im\geq q^{-n}\}$, which also are affinoid subsets, are invariant under the action by homographies of the subgroups of $\GL_2(A)$ finitely generated by $\GL_2(\FF_q)$ and $\{(\begin{smallmatrix} \lambda & \theta^i \\ 0 & \mu\end{smallmatrix}):\lambda,\mu\in\FF_q^\times,i\leq n\}$, the union of which is $\GL_2(A)$). Further, if $\gamma\in\GL_2(A)$, one easily sees that if $\gamma(\mathcal{D}_M)\cap\mathcal{D}_M\neq\emptyset$, then $\gamma\in\GL_2(\FF_q)$.
It is also easily seen that $$\Omega=\bigcup_{\gamma\in\GL_2(A)}\gamma(\mathcal{D}_M).$$
We can apply Proposition \ref{thuillier} to the isomorphism of affine varieties
 $$\GL_2(\FF_q)\backslash\mathbb{A}_{\CC_\infty}^{1}\xrightarrow{j_0}\mathbb{A}_{\CC_\infty}^{1},$$
 where $$j_0(z)=-\frac{(1+z^{q-1})^{q+1}}{z^{q-1}}$$
 (this is the {\em finite $j$-invariant} of Gekeler in \cite{GEK2}) to obtain an isomorphism of analytic spaces 
 $$\GL_2(\FF_q)\backslash\mathcal{D}_M\cong D_{\CC_\infty}(0,1).$$
 
 In parallel, we have the Borel subgroup $B=B(A)=\{(\begin{smallmatrix} * & * \\ 0 & *\end{smallmatrix})\}$ which acts on $\Omega_M$ and the isomorphism of analytic spaces $B\backslash\Omega_M\cong\dot{D}_{\CC_\infty}(0,S)$ induced by the map $\exp_A(z)^{-1}$ (Proposition \ref{rigid-isom}). We recall from Lemma \ref{with-this-help} that $\gamma\in\GL_2(\FF_q)$ is such that $\gamma(\Omega_M)\cap\Omega_M\neq\emptyset$ if and only if $\gamma$ is in $B$. 
 
 There is a procedure of gluing 
 two quotient rigid analytic spaces with such compatibility boundary conditions, into a new rigid analytic space,
 along with (\ref{nagao}) for $k=\FF_q$ and $t=\theta$. Note that $\mathcal{D}_M\cap\Omega_M=\{z\in\CC_\infty:|z|_\Im=|z|=M\}$
 and the two actions of $B$ over $\Omega_M$ and of $\GL_2(\FF_q)$ on $\mathcal{D}_M$
 agree with the action of $B\cap\GL_2(\FF_q)$ on $\mathcal{D}_M\cap\Omega_M$ and the gluing of these two quotient spaces is a well defined analytic space whose underlying 
 topological space is homeomorphic to the quotient topological space $\GL_2(A)\backslash\Omega$ which also carries a natural structure of analytic space. Additionally, this quotient space is isomorphic to
 the gluing of $D_{\CC_\infty}(0,1)$ and $\CC_\infty\setminus D_{\CC_\infty}^\circ(0,1)$ along $\{z\in\CC_\infty:|z|=1\}$,
 which is in turn isomorphic to $\CC_\infty$. This construction finally yields:
 \begin{Theorem}\label{ger-put}
 There is an isomorphism between 
 the quotient rigid analytic space $\Gamma\backslash\Omega$ and the rigid analytic affine line $\mathbb{A}^{1,an}_{\CC_\infty}$.
 \end{Theorem}
  
\section{Drinfeld modular forms}\label{Drinfeld-modular-forms}

We give a short synthesis on Drinfeld modular forms for the group $\Gamma=\GL_2(A)$ in the simplest case where $A=\FF_q[\theta]$, so that we can prepare the next part of this paper, where we construct modular forms for $\Gamma$ with (vector) values in certain $\CC_\infty$-Banach algebras. 

The map 
$$\GL_2(K_\infty)\times\Omega\rightarrow\CC_\infty^\times$$
defined by $(\gamma,z)\mapsto J_\gamma(z)=cz+d$ if $\gamma=(\begin{smallmatrix} * & * \\ c & d \end{smallmatrix})$ behaves like the classical factor of automorphy for $\GL_2(\RR)$. Indeed we have the cocycle condition:
$$J_{\gamma\delta}(z)=J_\gamma(\delta(z))J_\delta(z),\quad \gamma,\delta\in\GL_2(K_\infty).$$ Note that the image is indeed in $\CC_\infty^\times$, as $z,1$ are $K_\infty$-linearly independent if $z\in\Omega$.

\begin{Definition}{\em 
Let $f:\Omega\rightarrow\CC_\infty$ be an analytic function. We say that $f$ is {\em modular-like} of weight $w\in\ZZ$ if for all $z\in\Omega$,
$$f(\gamma(z))=J_\gamma(z)^wf(z),\quad \forall\gamma\in\GL_2(A).$$ 
It is a simple exercise to verify that $w$ is uniquely determined.

We say that a modular-like function of weight $w$ is:
\begin{enumerate}
\item {\em weakly modular} (of weight $w$) if there exists $N\in\ZZ$ such that the map $z\mapsto |\exp_A(z)^Nf(z)|$ is bounded over $\Omega_M$ for some $M>1$,
\item a {\em modular form} if the map $z\mapsto |f(z)|$ is bounded over $\Omega_M$ for some $M>1$.
\item a {\em cusp form} if it is a modular form and $\max_{z\in\Omega_M}|f(z)|\rightarrow0$ as 
$M\rightarrow\infty$.
\end{enumerate} 
}\end{Definition}
Let $f$ be modular like (of weight $w\in\ZZ$). Taking $\gamma=(\begin{smallmatrix} 1 & * \\ 0 & 1\end{smallmatrix})$ we see that $f(z+a)=f(z)$ for all $a\in A$. Therefore, by Proposition \ref{not-so-elementary}, there is a convergent series expansion of the type 
$$f(z)=\sum_{i\in\ZZ}f_i\exp_A(z)^i,\quad f_i\in\CC_\infty.$$ 
There is a rigid analytic analogue of Riemann's principle of removable singularities due to Bartenwerfer (see \cite{BAR})
in virtue of which we see that the $\CC_\infty$-vector space $M_w^!$ of 
weak modular forms of weight $w$ embeds in the field of Laurent series
$\CC_\infty((u))$ with the discrete valuation given by the order in $u$, where $u=u(z)$ is the {\em uniformiser at infinity}
$$u(z)=\frac{1}{\widetilde{\pi}\exp_A(z)}=\frac{1}{\widetilde{\pi}}\sum_{a\in A}\frac{1}{z-a},$$
which is an analytic function $\Omega\rightarrow\CC_\infty$.
Since $M^!_w\cap M^!_{w'}=\{0\}$ if $w\neq w'$ we have a
$\CC_\infty$-algebra $M^!=\oplus_wM^!_m$ which also embeds in the field of Laurent series $\CC_\infty((u))$. Denoting by $M_w$ the $\CC_\infty$-vector space of modular forms of weight $w$ and by $M=\oplus_wM_w$ the $\CC_\infty$-algebra of modular forms, we also have an 
embedding $M\rightarrow\CC_\infty[[u]]$ and cusp forms generate an ideal whose image in $\CC_\infty[[u]]$ is contained in the ideal generated by $u$. 

It is easy to deduce, from the modularity property, that $M_w^!\neq\{0\}$ implies $q-1\mid w$. Furthermore,
for all $w$ such that $M_w\neq\{0\}$, $M_w$ can be embedded via $u$-expansions in $\CC_\infty[[u^{q-1}]]$ and therefore the $\CC_\infty$-vector space of cusp forms $S_w$ can be embedded in $u^{q-1}\CC_\infty[[u^{q-1}]]$.

\subsection{$u$-expansions}

We have seen that we can associate in a unique way to any Drinfeld modular form $f$ a formal series $\sum_{i\geq 0}f_iu^i\in\CC_\infty[[u]]$
which is analytic in some disk $D(0,R)$, $R\in|\CC_\infty^\times|\cap]0,1[$. This is the analogue of the 
'Fourier series' of a complex-valued modular form for $\operatorname{SL}_2(\ZZ)$; for such a function $f:\mathcal{H}\rightarrow\CC$
we deduce, from $f(z+1)=f(z)$, a Fourier series expansion
$$f=\sum_{i\geq 0}f_iq^i,\quad f_i\in\CC,$$ converging for $q=q(z)=e^{2\pi i z}\in D_\CC^\circ(0,1)$. We want to 
introduce some useful tools for the study of $u$-expansions of Drinfeld modular forms.

For $n\geq 0$ we introduce the $\CC_\infty$-linear map $\CC_\infty[z]\xrightarrow{\mathcal{D}_n}\CC_\infty[z]$ uniquely determined by 
$$\mathcal{D}_n(z^m)=\binom{m}{n}z^{m-n}.$$ 
Note that we have Leibniz's formula $\mathcal{D}_n(fg)=\sum_{i+j=n}\mathcal{D}_i(f)\mathcal{D}_j(g)$.
The linear operators $\mathcal{D}_n$ extend in a unique way to $\CC_\infty(z)$ and further, on the $\CC_\infty$-algebra of analytic functions over any rational subset of $\Omega$ therefore 
inducing linear endomorphisms of the $\CC_\infty$-algebra of analytic functions $\Omega\rightarrow\CC_\infty$. Additionally, if $f:\Omega\rightarrow\CC_\infty$ is analytic and 
$A$-periodic, $\mathcal{D}_n(f)$ has this same property, and for all $n$, $\mathcal{D}_n$
induces $\CC_\infty$-linear endomorphisms of $\CC_\infty[[u]]$ (this last property follows from the fact that $\mathcal{D}_n(u)$ is bounded on $\Omega_M$ as one cas easily see distributing $\mathcal{D}_n$ on $u=\frac{1}{\widetilde{\pi}}\sum_{a\in A}\frac{1}{z-a}$, which gives $(-1)^n
\frac{1}{\widetilde{\pi}}\sum_{a\in A}\frac{1}{(z-a)^{n+1}}$). We normalise $\mathcal{D}_n$ by setting:
$$D_n=(-\widetilde{\pi})^{-n}\mathcal{D}_n.$$

\begin{Lemma}
For all $n\geq 0$, $D_n(K[u])\subset u^2K[u]$.
\end{Lemma}

\begin{proof}
It suffices to show that for all $n\geq 0$, $D_n(u)\in u^2K[u]$. We proceed by induction on $n\geq 0$; there is nothing to prove for $n=0$. Recall that $u(z)=\frac{1}{\exp_C(\widetilde{\pi}z)}$. Then, by Leibniz's formula:
\begin{eqnarray*}
\lefteqn{0=D_n(1)=D_n(u\exp_C(\widetilde{\pi}z))}\\
&=&D_n(u)\exp_C(\widetilde{\pi}z)+\sum_{\begin{smallmatrix}
i+q^k=n\\ k\geq 0
\end{smallmatrix}}D_i(u)D_{q^k}(\exp_C(\widetilde{\pi}z)),
\end{eqnarray*}
because $\exp_C$ is $\FF_q$-linear. In fact, $D_{q^k}(\exp_C(\widetilde{\pi}z))$ is constant
and equals the coefficient of $z^{q^k}$ in the $z$-expansion of $\exp_C$, which is $\frac{1}{d_k}$. We can therefore use induction to conclude that
$$D_n(u)=-u\left(-\sum_{\begin{smallmatrix}
i+q^k=n\\ k\geq 0
\end{smallmatrix}}D_i(u)d_k^{-1}\right)\in u^2K[u].$$ 
\end{proof}
The polynomials $G_{n+1}(u):=D_n(u)\in K[u]$ ($n\geq 1$) are called the {\em Goss polynomials} (see \cite[\S 3]{GEK}). It is easy to deduce from the above proof that
$D_j(u)=u^{j+1}$ as $j=1,\ldots,q-1$. There is no general formula currently available to compute $D_j(u)$ for higher values of $j$.

\subsubsection{Constructing Drinfeld modular forms}
The first non-trivial examples of Drinfeld modular forms have been described by Goss in his Ph. D. Thesis. To begin this subsection, we follow Goss \cite{GOS2} and we show how to construct non-zero Eisenstein series by using that $Az+A$ is strongly discrete in $\CC_\infty$ if $z\in\Omega$. We set:
$$E_w(z)=\sideset{}{'}\sum_{a,b\in A}\frac{1}{(az+b)^w}.$$ There are many sources where the reader can find a proof of the following lemma (see for instance \cite[(6.3)]{GEK}), but we prefer to give full details.
\begin{Lemma}\label{modularity-eisenstein}
The series $E_w$ defines a non-zero element of $M_w$ if and only if $w>0$ and $q-1\mid w$.
\end{Lemma}
\begin{proof}
The above series converges uniformly on every set $\Omega_M$
and this already gives that $E_w$ is analytic over $\Omega$.
The first property, that $E_w$ is modular-like of weight $w$, follows from a simple rearrangement of the sum defining $E_w(\gamma(z))$ for $\gamma\in\Gamma$ and its (unconditional) convergence, which leaves it invariant by permutation of its terms. Additionally, it is very easy to see that all terms involved in the sum are bounded on $\Omega_M$ for every $M$ which, by the ultrametric inequality, implies that $E_w$ itself is bounded on $\Omega_M$ for every $M$. It remains to describe when the series are zero identically, or non-zero.

For the non-vanishing property, we give an explicit evidence why $E_w$ has a $u$-expansion in $\CC_\infty[[u]]$, and we derive from partial knowledge of its shape the required property (but we are not able to compute in limpid way the coefficients of the $u$-expansion!).
First note that 
$$D_n(u)=\frac{1}{\widetilde{\pi}^{n+1}}\sum_{b\in A}\frac{1}{(z-b)^{n+1}},$$
so that we can use the Goss' polynomials $G_{n+1}(u)=D_n(u)$ as a 'model' to construct the $u$-expansion of $E_w$. Now, observe, for $w>0$:
$$E_w(z)=\sum_{b\in A}\frac{1}{b^w}+\sideset{}{'}\sum_{a\in A}\sum_{b\in A}\frac{1}{(az+b)^w}.$$ If $(q-1)\mid w$, we note that $$\sum_{b\in A}\frac{1}{b^w}=-\prod_P\left(1-P^{-w}\right)^{-1}=:-\zeta_A(w),$$
where the product runs over the monic irreducible polynomials $P\in A$ and therefore is non-zero. Then, if $(q-1)\mid w$ and if $A^+$ denotes the subset of monic polynomials in $A$:
\begin{eqnarray*}
E_w(z)&=&-\zeta_A(w)-\sum_{a\in A^+}\sum_{b\in A}\frac{1}{(az+b)^w}\\
&=&-\zeta_A(w)-\widetilde{\pi}^w\sum_{a\in A^+}G_w(u(az)),
\end{eqnarray*}
a series which converges uniformly on every affinoid subset of $\Omega$. Note that for $a\in A\setminus\{0\}$, the function
$u(az)$ can be expanded as a formal series $u_a$ of $u^{|a|}K[[u]]$ (normalise $|\cdot|$ by 
$|\theta|=q$) locally converging at $u=0$ (in a disk of positive radius $r$ independent of $a$).
This yields the explicit series expansion (convergent for the $u$-valuation, or for the 
sup-norm over the disk $D(0,r)$ in the variable $u$):
\begin{equation}\label{expansion-Ew}
E_w(z)=-\zeta_A(w)-\widetilde{\pi}^w\sum_{a\in A^+}G_w(u_a).
\end{equation}
This also shows that $E_w$ is, in this case, not identically zero. Indeed $\zeta_A(w)$ is non-zero, while the part depending on $u$ in the above expression tends to zero as $|z|_\Im$ tends to $\infty$. On the other hand, if $(q-1)\nmid w$, the factor 
of automorphy $J_\gamma^w$ does not induce a factor of automorphy for the group 
$\operatorname{PGL}_2(A)$ defined as the quotient of $\GL_2(A)$ by scalar matrices and this implies that any modular form of such weight $w$ vanishes identically, and so it happens that $E_w$ vanishes in this case.
\end{proof}

\begin{Remark}{\em It is instructive at this point to compare our observations with the settings of the original, complex-valued Eisenstein series. Indeed, it is well known, classically, that if $w>2$, $2\mid w$ and $q=e^{2\pi iz}$:
$$E_w(z)=\sideset{}{'}\sum_{a,b\in\ZZ}\frac{1}{(az+b)^w}=2\zeta(w)+2\frac{(2\pi i)^{\frac{w}{2}}}{(\frac{w}{2}-1)!}\sum_{n\geq 1}\frac{n^{\frac{w}{2}-1}q^n}{1-q^n},\quad \Im(z)>0.$$
The analogy is therefore between the series $$\sum_{a\in A^+}G_w(u_a)$$ and $$\sum_{n\geq 1}\frac{n^{\frac{w}{2}-1}q^n}{1-q^n}.$$ However, it is well known that the latter series can be further expanded as follows, with $\sigma_k(n)=\sum_{d|n}d^k$:
$$\sum_{n\geq 1}\sigma_{\frac{w}{2}-1}(n)q^n.$$ For the series $\sum_{a\in A^+}G_w(u_a)$, 
this aspect is missing, and there is no available intelligible recipe to compute the coefficients of the $u$-expansion of $E_w$ directly, at the time being.}\end{Remark}

\subsection{Construction of non-trivial cusp forms}

We have constructed non-trivial modular forms, but they are not cusp forms. 
We construct non-zero cusp forms in this section.
Let $z$ be an element of $\Omega$. Then, $\Lambda=\Lambda_z=Az+A$ is an $A$-lattice of rank $2$ of $\CC_\infty$.
By Theorem \ref{drinfeld-theorem}, we have the Drinfeld $A$-module $\phi:=\phi_\Lambda$ which is of rank $2$. 
Hence, we can write
$$\phi_\theta(Z)=\theta Z+\widetilde{g}(z)Z^q+\widetilde{\Delta}(z)Z^{q^2},\quad \forall(z,Z)\in\Omega\times\CC_\infty$$ for functions $\widetilde{g},\widetilde{\Delta}:\Omega\rightarrow\CC_\infty$.

We consider the function $\Omega\times\CC_\infty\xrightarrow{(z,Z)\mapsto\mathbb{E}(z,Z)}\CC_\infty$
which associates to $(z,Z)$ the value 
\begin{equation}\label{definition-alpha-i}
\mathbb{E}(z,Z):=\exp_\Lambda(Z)=\sum_{i\geq 0}\alpha_i(z)Z^{q^i}=Z\sideset{}{'}\prod_{\lambda\in \Lambda}\left(1-\frac{Z}{\lambda}\right)
\end{equation}
at $Z$ of the exponential series $\exp_\Lambda$
associated to the $A$-lattice $\Lambda=\Lambda_z$ of $\CC_\infty$. It is an analytic function and we have $\phi_a(\exp_\Lambda(Z))=\exp_\Lambda(aZ)$ for all $a\in A$.

The following result collects the various functional properties of $\EE(z,Z)$; proofs rely on simple computations that we leave to the reader.
\begin{Lemma}\label{twotypesofrelation}
For all $z\in\Omega$, $Z\in \CC_\infty$, $\gamma\in\Gamma$ and $a\in A$:
\begin{enumerate}
\item $\phi_\Lambda(a)(\mathbb{E}(z,Z))=\mathbb{E}(z,aZ)$,
\item $\mathbb{E}(\gamma(z),Z)=J_\gamma(z)^{-1}\mathbb{E}(z,J_\gamma(z)Z)$.
\item $\mathbb{E}(z,Z+az+b)=\mathbb{E}(z,Z)$, for all $a,b\in A$.
\end{enumerate}
\end{Lemma}
\begin{Remark}{\em 
Loosely, we can say that $\mathbb{E}$ is a 'non-commutative modular form of weight $(-1,1)$'.
The second formula can be also rewritten as:
$$\mathbb{E}\left(\gamma(z),\frac{Z}{J_\gamma(z)}\right)=J_\gamma(z)^{-1}\mathbb{E}(z,Z),\quad \gamma\in\GL_2(A),$$
so that $\mathbb{E}$ functionally plays the role of a Jacobi form of level $1$, weight $-1$ and index $0$ (this is in close analogy with the Weierstrass $\wp$-functions).}\end{Remark}

By taking the formal logarithmic derivative in the variable $Z$ of the Weierstrass product expansion of $\exp_\Lambda(Z)$ (for $z$ fixed) we note that 
$$\frac{Z}{\EE(z,Z)}=1-\sum_{\begin{smallmatrix} k\geq 0\\ (q-1)\mid k\end{smallmatrix}}E_k(z)Z^k$$ so that the coefficients in this expansion in powers of $Z$ are analytic functions on $\Omega$, from which we deduce, by inversion, that the coefficient functions $\alpha_i:\Omega\rightarrow\CC_\infty$ of $\mathbb{E}$ are analytic. By Lemma \ref{modularity-eisenstein} and the homogeneity of the algebraic expressions expressing the functions $\alpha_i$ in terms of the Eisenstein series $E_k$ we see that $\alpha_i\in M_{q^i-1}$ for all $i\geq 0$. As $|z|_\Im\rightarrow\infty$ we have
$E_k(z)\rightarrow-\zeta_A(k)$, after a simple computation we see that 
$$\EE(z,Z)\rightarrow \exp_A(Z)$$
uniformly for $Z\in D$ for every disk $D\subset\CC_\infty$. This means that the functions $\alpha_i$ are not cusp forms (the coefficients of $\exp_A\in K_\infty[[\tau]]$ are all non-zero). 
To construct cusp forms, we now look at the coefficients $\widetilde{g},\widetilde{\Delta}$ of $\phi_\theta$ which are functions of the variable $z\in\Omega$.
By (1) and (2) of Lemma \ref{twotypesofrelation}, for $\gamma\in\Gamma$, writing now $\phi_{\Lambda_z}(\theta)$ in place of $\phi_\theta$:
\begin{eqnarray*}
\phi_{\Lambda_{\gamma(z)}}(\theta)(J_\gamma(z)^{-1}\EE(z,J_\gamma(z)Z))&=&\phi_{\Lambda_{\gamma(z)}}(\theta)(\EE(\gamma(z),Z))\\
&=&\EE(\gamma(z),\theta Z)\\
&=&J_\gamma(z)^{-1}\EE(z,\theta J_\gamma(z)Z).
\end{eqnarray*}
Hence, $\phi_{\Lambda_{\gamma(z)}}(\theta)(J_\gamma(z)^{-1}\EE(z,W))=J_\gamma(z)^{-1}\EE(z,\theta W)=J_\gamma(z)^{-1}\phi_{\Lambda_z}(\EE(z,W))$ for $W\in \CC_\infty$. Since it is obvious that the coefficient functions
$\widetilde{g},\widetilde{\Delta}$ are analytic on $\Omega$, they are in this way respectively modular-like functions of respective weights $q-1$ and $q^2-1$. Furthermore:

\begin{Lemma}
$\widetilde{g}\in M_{q-1}\setminus S_{q-1}$ and $\widetilde{\Delta}\in S_{q^2-1}\setminus\{0\}$. Additionally,
$\widetilde{\Delta}(z)\neq0$ for all $z\in \Omega$.
\end{Lemma} 

\begin{proof}
The modularity of $\widetilde{g}$ and $\widetilde{\Delta}$ follows from the previously noticed fact that $\exp_{\Lambda_z}(Z)\rightarrow\exp_A(Z)$ uniformly with $Z$ in disks as $|z|_\Im\rightarrow\infty$. Indeed, this implies that $\phi_\theta(Z)\rightarrow\theta Z+\widetilde{\pi}^{q-1}Z^q$ (uniformly on every disk) so that $\widetilde{g}\rightarrow\widetilde{\pi}^{q-1}$ and $\widetilde{\Delta}\rightarrow0$ as $|z|_\Im\rightarrow\infty$ and we see that 
$\widetilde{g}$ is a modular form of weight $q-1$ which is not a cusp form, and $\widetilde{\Delta}$ is a cusp form. 

We still need to prove that $\widetilde{\Delta}$ is not identically zero; to do this, we prove now the last property of the lemma, which is even stronger. Assume by contradiction that there exists $z\in\Omega$ such that $\widetilde{\Delta}(z)=0$.
Then $$\phi_{\Lambda_z}(\theta)=\theta+\widetilde{g}(z)\tau$$ which implies that the exponential $\exp_{\Lambda_z}$ induces an isomorphism of $A$-modules
$\exp_{\Lambda_z}:\CC_\infty/\Lambda_z\rightarrow C(\CC_\infty)$ (the Carlitz module). But this disagrees with Theorem \ref{drinfeld-theorem} which would deliver an isomorphism $\Lambda_z\cong A$ between lattices 
of different ranks. This proves that $\widetilde{\Delta}$ does not vanish on $\Omega$. 
\end{proof}
Following Gekeler in \cite{GEK}, we define the modular forms $g,\Delta$ of respective weights $q-1$ and $q^2-1$ by $\widetilde{g}=\widetilde{\pi}^{q-1}g$ and $\widetilde{\Delta}=\widetilde{\pi}^{q^2-1}\Delta$. The reason for choosing these normalisations is that 
it can be proved that the $u$-expansions of $g,\Delta$ have coefficients in $A$.
We are not far from a complete proof of the following (see \cite[(5.12)]{GEK} for full details):
\begin{Theorem}\label{structure-modular-forms}
$M=\oplus_{w\in\ZZ}M_w=\CC_\infty[g,\Delta]$
\end{Theorem}
The proof rests on three crucial properties (1) existence of Eisenstein series (2) existence of the cusp form $\Delta$ which additionally is nowhere vanishing on $\Omega$, and (3) modular forms of weight $0$ for $\Gamma$ are constant, which follows from the fact that a modular form of weight $0$ can be identified with a holomorphic function over $\PP_{\FF_q}^1(\CC_\infty)$ by Theorem \ref{ger-put}, which is constant. We omit the details.

\subsubsection{Drinfeld modular forms and the Bruhat-Tits tree}

We briefly sketch the interaction between Drinfeld modular forms and the Bruhat-Tits tree, mainly inviting the reader, yet in quite an informal way, to read the important work of Teitelbaum in \cite{TEI}. A simple computation indicates that if $f$ is a rigid analytic function over the annulus
$V=\sqcup_{-1<\lambda<0}C_\lambda$ (or on a more general annulus in $\Omega$) so that 
$f$ is defined by a convergent series $\sum_{i\in\ZZ}f_iz^i$ with the coefficients $f_i$ in $\CC_\infty$, then the residue $$\operatorname{Res}_V(f(z)dz):=f_{-1}$$ does not depend on the local coordinate chosen to express the differential form $\omega=f(z)dz$. Namely, if $t$ is another local coordinate and $z=z(t)=\sum_{i>0}z_it^i$ with $z_i\in\CC_\infty$ and $z_1\in\CC_\infty^\times$ (with suitable convergence conditions), then the coefficient of
$t^{-1}dt$ in $\omega(z(t))=f(z(t))dz(t)=f(z(t))\frac{dz}{dt}dt$ is also equal to $f_{-1}$, and 
in particular, $\operatorname{Res}_V(fdz)$ does not depend on the choice of the `center' of the annulus.

We consider $\mathcal{T}^e$ the set of the oriented edges of the Bruhat-Tits tree. The elements are in one-to-one correspondence with the disjoint subsets of $\Omega$:
$$V_{n,\alpha}:=\alpha+\bigsqcup_{\lambda\in]n-1,n[}C_\lambda,\quad n\in\ZZ,\quad \alpha\in\oplus_{i\leq n-1}\FF\pi^i.$$ Note that $V_{n,\alpha}=\{z\in\CC_\infty:|\pi|^n<|z-\alpha|<|\pi|^{n-1}\}$, which is an annulus centered at elements of $K_\infty$ with inner radius $|\pi|^{n}$ and outer radius $|\pi|^{n-1}$, $n$ varying in $\ZZ$. Moreover, $V=V_{0,0}$. If $f:\Omega\rightarrow\CC_\infty$ is a rigid analytic function, then $f$ is rigid analytic on every $V_{n,\alpha}$ and we have a well defined residue map $$\mathcal{T}^e\xrightarrow{\operatorname{res}(f)}\CC_\infty$$ which is a `harmonic function' in virtue of the {\em ultrametric residue theorem} (see \cite[\S 3]{GER&PUT}; we do not give full details and definitions of `harmonic functions' etc., this would bring us too far away from the objectives of this paper). Of course, we do not expect 
the map $\operatorname{res}(f)$ to reproduce faithfully the behaviour of $f$. For example, if $f$
is entire over $\CC_\infty$ then all the residues of the differential form $fdz$ are clearly zero and
$\operatorname{res}(f)$ vanishes identically, which might not be the case for $f$. 

Where the map $\operatorname{res}(f)$ becomes really useful is with rigid analytic functions $f$ which are determined by more elaborate patching of local data than just entire functions. Typically, functions defined by globally non uniform convergent series over $\Omega$. 
If $f$ is a Drinfeld modular form, Teitelbaum proved, in a much more general setting ($\Gamma$ arithmetic subgroup of $\GL_2(A)$), that a suitable variant of the residue map provides us 
with an isomorphism of $\CC_\infty$-vector spaces
$$S_w(\Gamma)\rightarrow C_{\operatorname{har}}(\Gamma,w),$$
where $S_w(\Gamma)$ is the space of Drinfeld cusp forms of weight $w$ for $\Gamma$ as defined in ibid. and generalising our space $S_w$ for $\Gamma=\GL_2(A)$, and where 
$C_{\operatorname{har}}(\Gamma,w)$ is the space of `weight $w$ harmonic cocycles' for $\Gamma$. This map can be defined also over $M_w(\Gamma)$, the space of Drinfeld modular forms of weight $w$ for $\Gamma$. Then, the kernel is spanned by the Eisenstein series of weight $w$. For this and other deep properties such as a homological interpretation of the residue map and an interesting and yet mysterious analysis of the Fourier series of cusp forms, see the paper \cite{TEI}.

\section{Eisenstein series with values in Banach algebras}\label{Modular-forms-Tate-algebras}

The final purpose of this and the next more advanced sections of the present paper is to show certain identities for a variant-generalisation of Eisenstein series (see Theorem \ref{identity-eisenstein-1}).
We recall that $A=\FF_q[\theta]$.
Let $B$ be a $\CC_\infty$-Banach algebra with sub-multiplicative norm $\|\cdot\|$ (\footnote{That is, $\|ab\|\leq \|a\|\|b\|$ for all $a,b\in B$. We adopt the simpler notations $\|\cdot\|$ and $|\cdot|$ at the place of $|\cdot|_\infty$ etc. that we have used in the first few sections of our text.}) norm $\|\cdot\|$ (extending the norm $|\cdot|$ of $\CC_\infty$) with the property that $\|B\|=|\CC_\infty|$. Let $X$ be a rigid analytic variety. We set 
$$\mathcal{O}_{X/B}=\mathcal{O}_X\widehat{\otimes}_{\CC_\infty}B,$$ with $\mathcal{O}_X$ the structural sheaf of $X$, of $\CC_\infty$-algebras.
In other words, if $U\subset X$ is an affinoid subset of $X$, then $\mathcal{O}_X(U)$ carries the supremum norm $\|\cdot\|_{U}$ and we define $\mathcal{O}_{X/B}(U)$ to be the completion of $\mathcal{O}_X\otimes_{\CC_\infty}B$
for the norm induced by $\|f\otimes b\|=|f|_{U}$, for $f\in \mathcal{O}_X(U)$ and $b\in B$. If $B$ has a countable orthonormal 
basis $\mathcal{B}=(b_i)_{i\in\mathcal{I}}$, an element $f\in\mathcal{O}_{X/B}(U)$ has a convergent series expansion
$$f=\sum_{i\in\mathcal{I}}f_ib_i,$$
where $f_i\in\mathcal{O}_X(U)$, with $|f_i|_U\rightarrow0$ for the Fr\'echet filter on $\mathcal{I}$.

One sees that that
Tate's acyclicity Theorem extends to this setting, namely, if $X$ is an affinoid variety, $\mathcal{O}_{X/B}$ is a sheaf of $B$-algebras. The global sections are the {\em analytic functions} $X\rightarrow B$.

We will mainly use the cases $X=\Omega$ and $X=\mathbb{A}^{s,an}_{\CC_\infty}$. If $X=\mathbb{A}^{s,an}_{\CC_\infty}$, an element 
of $\mathcal{O}_{X/B}$ is a {\em $B$-valued entire function of $s$ variables}. We can identify it with a map $\CC_\infty^s\rightarrow B$ allowing a series expansion in $B[[\underline{t}]]$ with $\underline{t}=(t_1,\ldots,t_s)$ converging on $D(0,R)^s$
for all $R>0$. A bounded entire function $\CC_\infty\rightarrow B$ is constant (this is a  generalisation of Liouville's theorem which uses the hypothesis that $\|B\|=|\CC_\infty|$ is not discrete, see \cite{PEL&PER}).

We work with $B$-valued 
analytic functions where $B=\KK$ is the completion of $\CC_\infty(\underline{t})$ for the Gauss norm $\|\cdot\|=\|\cdot\|_\infty$, where
$\underline{t}=(t_1,\ldots,t_s)$. We have $\|\KK\|=|\CC_\infty|$ and the residue field is $\FF_q^{\operatorname{ac}}(\underline{t})$. In all the following, we consider matrix-valued analytic functions and we extend 
norms to matrices in the usual way by taking the supremum of norms of the entries of a matrix.

We extend the $\FF_q$-automorphism $\tau:\CC_\infty\rightarrow\CC_\infty$, $x\mapsto x^q$,
$\FF_q(\underline{t})$-linearly and continuously on $\KK$. The subfield of the fixed elements $\KK^{\tau=1}=\{x\in\KK:\tau(x)=x\}$ is easily seen to be equal to $\FF_q(\underline{t})$ by a simple variant of Mittag-Leffler theorem. 
Let $\lambda_1,\ldots,\lambda_r\in\CC_\infty$
be $K_\infty$-linearly independent. This is equivalent to saying that the $A$-module 
$$\Lambda=A\lambda_1+\cdots+A\lambda_r\subset\CC_\infty$$ is an $A$-lattice.
In this way, the exponential function
$\exp_\Lambda$ induces a continuous open $\FF_q(\underline{t})$-linear
endomorphism of $\KK$, the kernel of which contains $\Lambda\otimes_{\FF_q}\FF_q(\underline{t})$ (it can be proved that $\exp_\Lambda$ is surjective over $\KK$ and the kernel is exactly 
$\Lambda\otimes_{\FF_q}\FF_q(\underline{t})$ but we do not need this in the present paper).
The Drinfeld $A$-module $\phi=\phi_\Lambda$ gives rise to a structure of $\FF_q(\underline{t})^{nr\times n}[\theta]$-module $$\phi(\KK^{nr\times n})$$ by simply using the $\FF_q(\underline{t})$-vector space structure of $\KK$ and defining the multiplication $\phi_\theta$ by $\theta$ 
with the above extension of $\tau$. 

We consider an injective $\FF_q$-algebra morphism
$$A\xrightarrow{\chi}\FF_q(\underline{t})^{n\times n}$$ and we set, with $(\lambda_1,\ldots,\lambda_r)$ an $A$-basis of $\Lambda$ (the exponential now applied coefficientwise):
$$\omega_\Lambda=\exp_\Lambda\left((\theta I_n-\chi(\theta))^{-1}\left(\begin{matrix}\lambda_1I_n \\ \vdots \\ \lambda_rI_n\end{matrix}\right)\right)\in\KK^{rn\times n}.$$

\begin{Lemma}\label{omega-lambda}
For all $a\in \FF_q(\underline{t})[\theta]$ we have the identity  $\phi_a(\omega_\Lambda)=\chi(a)\omega_\Lambda$ in $\KK^{rn\times n}$.
\end{Lemma}
\begin{proof} Since the variables $t_i$ are central for $\tau$ and $\FF_q(\underline{t})[\theta]$ is euclidean, it suffices to show that 
$\phi_\theta(\omega_\Lambda)=\chi(t)\omega_\Lambda$. Now observe, for $a\in A$:
\begin{eqnarray*}
\phi_\Lambda(a)(\omega_\Lambda) & =& \exp_\Lambda((\theta I_n-\chi(\theta))^{-1}
\left(\begin{matrix} (aI_n-\chi(a)+\chi(a))\lambda_1\\
\vdots\\
(aI_n-\chi(a)+\chi(a))\lambda_r\end{matrix}\right)\\
&=&\chi(a)\omega_\Lambda,
\end{eqnarray*}
because $(\theta I_n-\chi(\theta))^{-1}(aI_n-\chi(a))\in\FF_q(\underline{t})[\theta]^{n\times n}$
so that $(\theta I_n-\chi(\theta))^{-1}(aI_n-\chi(a))\lambda_i$ lies in the kernel of $\exp_\Lambda$ (applied coefficientwise).
\end{proof}
Hence, $\omega_\Lambda$ is a particular instance of {\em special function} as defined and studied
in \cite{ANG&TAV,GAZ&MAU}.
Note also that the map
$$\Phi_\Lambda:Z\mapsto\exp_\Lambda((\theta I_n-\chi(\theta))^{-1}Z)$$
defines an entire function $\CC_\infty\rightarrow\KK^{n\times n}$. An easy variant of 
the proof of Lemma \ref{omega-lambda} delivers:
\begin{Lemma}\label{omega-lambda-2}
We have the functional equation 
$\tau(\Phi_\Lambda(Z))=(\chi(\theta)-\theta I_n)\Phi_\Lambda(Z)+\exp_\Lambda(Z)I_n$ in $\KK^{n\times n}$.
\end{Lemma}

We now introduce a 'twist' of the logarithmic derivative of $\exp_\Lambda$.
We recall that $A\xrightarrow{\chi}\FF_q(\underline{t})^{n\times n}$ is an injective $\FF_q$-algebra morphism. We introduce the {\em Perkins' series} (introduced in a slightly narrower setting by Perkins in his Ph. D. thesis \cite{PER0}):
$$\psi_\Lambda(Z):=\sum_{a_1,\ldots,a_r\in A}\frac{1}{Z-a_1\lambda_1-\cdots-a_r\lambda_r}
(\chi(a_1),\ldots,\chi(a_r)),\quad Z\in\CC_\infty$$ (depending on the choice of the basis of $\Lambda$ as well as on the choice of the algebra morphism $\chi$). The series converges for $Z\in\CC_\infty\setminus\Lambda$ to a 
function $\CC_\infty\setminus\Lambda\rightarrow\KK^{n\times rn}.$
We have (after elementary rearrangement of the terms):
\begin{equation}\label{simple-translation}
\psi_\Lambda(Z-b_1\lambda_1-\cdots-b_r\lambda_r)=\psi_\Lambda(Z)-(\chi(b_1),\ldots,\chi(b_r))\exp_\Lambda(Z)^{-1},\quad b_1,\ldots,b_r\in A.
\end{equation}
The next proposition explains why we are interested in 
the Perkins' series: they can be viewed as generating series of certain $\KK$-vector-valued Eisenstein series that we introduce below. Determining identities for the Perkins' series results in determining identities for such Eisenstein series.
\begin{Proposition}\label{series-expansion}
There exists $r\in|\CC_\infty^\times|$ such that the following series expansion, convergent for $Z$ in $D(0,r)$, holds:
$$\psi_\Lambda(Z)=-\sum_{\begin{smallmatrix}j\geq 1\\ j\equiv1
{(q-1)}\end{smallmatrix}}Z^{j-1}\mathcal{E}_\Lambda(j;\chi),$$
where for $j\geq 1$, $$\mathcal{E}_\Lambda(j;\chi):=\sideset{}{'}\sum_{a_1,\ldots,a_r\in A}\frac{1}{(a_1\lambda_1+\cdots+a_r\lambda_r)^j}(\chi(a_1),\ldots,\chi(a_r))\in\KK^{n\times rn}.$$
\end{Proposition}
The series $\mathcal{E}_\Lambda(j;\chi)$ is the {\em Eisenstein series of weight $j$} associated to $\Lambda$ and $\chi$. Note that this is in deep correspondence with the {\em canonical deformations} of the Carlitz module in Tavares Ribeiro's contribution to this volume, \cite[\S 4.2]{TAV}. The reader can make these connections deeper with an accurate analysis on which we skip here.

\begin{Problem}\label{problem4}
{\em Develop the appropriate generalisation of the theory of harmonic cocycles of Teitelbaum \cite{TEI} and construct the residue map along the notion of $\KK$-vector-valued modular form which naturally includes the above Eisenstein series as in \cite{PEL&PER3}.}
\end{Problem}

\begin{proof}[Proof of Proposition \ref{series-expansion}]
Since $\Lambda$ is strongly discrete, $D(0,r)\cap(\Lambda\setminus\{0\})=\emptyset$
for some $r\neq0$. Then, we can expand, for the coefficients $a_i$ not all zero,
$$\frac{1}{Z-a_1\lambda_1-\cdots-a_r\lambda_r}=\frac{-1}{a_1\lambda_1+\cdots+a_r\lambda_r}\sum_{i\geq 0}\left(\frac{Z}{a_1\lambda_1+\cdots+a_r\lambda_r}\right)^i.$$ The result follows from the fact that $\mathcal{E}_\Lambda(j;\chi)$, which is always convergent for $j>0$, vanishes identically for $j\not\equiv1\pmod{q-1}$ which is easy to check observing that 
$\Lambda=\lambda\Lambda$ for all $\lambda\in\FF_q^\times$, and reindexing the sum defining $\mathcal{E}_\Lambda(j;\chi)$.
\end{proof}

\begin{Lemma}\label{lemma-sharp}
The function $F^\sharp(Z):=\exp_\Lambda(Z)\psi_\Lambda(Z)$ defines an entire function
$\CC_\infty\rightarrow\KK^{n\times rn}$ such that, for all 
$\lambda=a_1\lambda_1+\cdots+a_r\lambda_r\in\Lambda$,
$F^\sharp(\lambda)=(\chi(a_1),\ldots,\chi(a_r))\in\FF_q(\underline{t})^{n\times nr}$.
\end{Lemma}
\begin{proof}
This easily follows from the fact that $\psi_\Lambda$ converges at $Z=0$, and (\ref{simple-translation}).  
\end{proof}
The function $\psi_\Lambda$ is intimately related to the exponential $\exp_\Lambda$ by means of the following result, where $\exp_\Lambda$ on the right is the unique continuous map $\KK^{n\times n}\rightarrow\KK^{n\times n}$ which induces a $\FF_q(\underline{t})^{n\times n}[\theta]$-module morphism
$\KK^{n\times n}\rightarrow\phi_\Lambda(\KK^{n\times n})$.
\begin{Lemma}\label{identity-rudy-gen}
We have the identity of entire functions $\CC_\infty\rightarrow\KK^{n\times n}$ of the variable $Z$:
$$\exp_\Lambda(Z)\psi_\Lambda(Z)\omega_\Lambda=\exp_\Lambda((\theta I_n-\chi(\theta))^{-1}Z).$$
\end{Lemma}
\begin{proof} By Lemma \ref{lemma-sharp}, the function $$F(Z):=F^\sharp(Z)\cdot\omega_\Lambda:\CC_\infty\rightarrow\KK^{n\times n}$$
is an entire function such that $$F(\lambda)=(\chi(a_1),\ldots,\chi(a_r))\omega_\Lambda\in\KK^{n\times n},\quad \forall\lambda=a_1\lambda_1+\cdots+a_r\lambda_r\in\Lambda.$$
We set $$G(Z)=\exp_\Lambda((\theta I_n-\chi(\theta))^{-1}Z).$$ 
Let $\lambda=a_1\lambda_1+\cdots+a_r\lambda_r\in \Lambda$. We have, by Lemma \ref{omega-lambda},
\begin{eqnarray*}
G(\lambda) & =& \exp_\Lambda((\theta I_n-\chi(\theta))^{-1}((a_1I_n-\chi(a_1)+\chi(a_1))\lambda_1+\cdots+(a_rI_n-\chi(a_r)+\chi(a_r))\lambda_r)\\
&=&(\chi(a_1),\ldots,\chi(a_r))\omega_\Lambda.
\end{eqnarray*}
Hence, the entire functions $F,G$ agree on $\Lambda$. The function $F-G$
is an entire function $\CC_\infty\rightarrow\KK^{n\times n}$ which vanishes over $\Lambda$.
Hence, $$H(Z)=\frac{F(Z)-G(Z)}{\exp_\Lambda(Z)}$$ defines an entire function over $\CC_\infty$.
Now, it is easy to see that $$\lim_{|Z|\rightarrow\infty}\|H(Z)\|=0.$$ Since 
the valuation group of $\KK$ is dense in $\RR^\times$, the appropriate generalisation of Liouville's theorem
\cite[Proposition 8]{PEL&PER} for entire functions holds in our settings and $H=0$ identically.
\end{proof}
\begin{Remark}{\em 
More generally, we can study $A$-module maps
$$\Lambda\xrightarrow{\chi}\KK^{n\times n}$$
with bounded image (the $A$-module structure on $\KK^{n\times n}$ being  
induced by an injective algebra homomorphism $A\hookrightarrow\FF_q(\underline{t})\hookrightarrow\KK^{n\times n}$) and 
Perkins' series
$$\psi_\Lambda(n;\chi):=\sum_{\lambda\in\Lambda}\frac{\chi(\lambda)}{(Z-\lambda)^n}.$$}
\end{Remark}

Lemma \ref{identity-rudy-gen} delivers an identity for $\psi_\Lambda$ in terms of certain
analytic functions of the variable $Z$ which are explicitly computable in terms of $\exp_\Lambda$. To see this, observe that the $\KK$-algebra of analytic functions $D(0,r)\rightarrow\KK$ is stable by the $\KK$-linear divided higher derivatives $\mathcal{D}_{Z,n}$ defined by $\mathcal{D}_{Z,n}(Z^m)=\binom{m}{n}Z^{m-n}$. In particular, 
$\mathcal{D}_{Z,n}(\psi_\Lambda)$ is well defined for any $n>0$.
We write $f^{(k)}$
for $\tau^k(f)$, $f\in\KK$ or for $f$ more generally a $\KK^{r\times s}$-valued map for arbitrary integers $r,s$. If $f=\sum_{i\geq 0}f_iZ^i$ is an analytic function over a disk $D(0,r)$ in the variable $Z$, then $f^{(k)}=\sum_{i\geq 0}\tau(f_i)Z^{q^ki}$
is again analytic if $k\geq 0$.
Observe that in particular,
$$\psi_\Lambda(Z)^{(k)}=\mathcal{D}_{q^k-1}(\psi_\Lambda(Z)),\quad k\geq 0.$$
Lemma \ref{identity-rudy-gen} implies 
$$\psi_\Lambda(Z)\omega_\Lambda=\mathcal{H}(Z):=\exp_\Lambda(Z)^{-1}\exp_\Lambda((\theta I_n-\chi(\theta))^{-1}Z),$$ and we note that on the right we have
an analytic function
$D(0,r)\rightarrow\KK^{n\times n}$ for some $r\in|\CC_\infty^\times|$. Applying 
$\mathcal{D}_{q^k-1}$ on both sides of this identity and observing that $\omega_\Lambda$ does not depend on $Z$, we deduce:
$$\psi_\Lambda(Z)^{(k)}\omega_\Lambda=\mathcal{D}_{q^k-1}(\mathcal{H})(Z),\quad k\geq 0.$$ Now, since the function $\psi_\Lambda(Z)^{(k)}$ is in fact an analytic function of the variable $Z^{q^k}$, this is also true for the function $\mathcal{D}_{q^k-1}(\mathcal{H})(Z)$ so that 
$$\mathcal{H}_k(Z)=(\mathcal{D}_{q^k-1}(\mathcal{H})(Z))^{(-k)},\quad k\geq 0$$
are all analytic functions $D(0,r)\rightarrow\KK^{n\times n}$ (note that $\mathcal{H}_0=\mathcal{H}$). We introduce the matrices
$$\boldsymbol{\Omega}_\Lambda=(\omega_\Lambda,\omega_\Lambda^{(-1)},\ldots,\omega_\Lambda^{(1-r)})\in\KK^{rn\times rn},\quad \boldsymbol{\mathcal{H}}_\Lambda(Z)=(\mathcal{H}_0,\ldots,\mathcal{H}_{r-1}),$$ where the latter is an $n\times rn$-matrix of analytic functions $D(0,r)\rightarrow\KK$. Then, 
$$\psi_\Lambda(Z)\boldsymbol{\Omega}_\Lambda=\boldsymbol{\mathcal{H}}_\Lambda(Z).$$
But a simple variant of the Wronskian lemma (see \cite[\S 4.2.3]{PEL&BOU}) implies that 
$\boldsymbol{\Omega}_\Lambda$ is invertible. We have reached:
\begin{Theorem}\label{identity-eisenstein-series}
The identity $\psi_\Lambda(Z)=\boldsymbol{\mathcal{H}}_\Lambda(Z)\boldsymbol{\Omega}_\Lambda^{-1}$ holds, for functions locally analytic at $Z=0$. 
\end{Theorem}
The identity of the previous theorem connects the 'twisted logarithmic derivative' $\psi_\Lambda(Z)$ to the inverse Frobenius twists of the divided higher derivatives of the mysterious function $\boldsymbol{\mathcal{H}}$, which are certainly not always easy to compute, unless $r=1$, where there is no higher derivative to compute at all. If we set, additionally, $\chi=\chi_t$
where $\chi_t(a)=a(t)$ so that $n=1$, then we reach a known identity, which was first discovered by R. Perkins
in \cite{PER} (that we copy below adapting it to our notations):
$$\exp_A(Z)\omega(t)\sum_{a\in A}\frac{a(t)}{Z-a}=\exp_A\left(\frac{Z}{\theta-t}\right),$$
with $\omega$ Anderson-Thakur's function and $\exp_A(Z)=Z\sideset{}{'}\prod_{a\in A}(1-\frac{Z}{a})$. This formula is expressed in \cite[Theorem 1]{PEL&PER} in a slightly different
manner by using Papanikolas' deformation of the Carlitz logarithm. Note that these references also contain other types of generalisation. The above formula can be viewed as an analogue 
of \cite[Lemma 1.3.21]{KAT} (the analogy can be pursued further). We owe this remark to 
Lance Gurney that we thankfully acknowledge.

\begin{Problem}\label{problem5}{\em 
This should be considered as a starting point for an extension of Kato's arguments related to the connection between the zeta-values phenomenology and Iwasawa's theory appearing in \cite{KAT}. One may ask how far a parallel with Kato's viewpoint can go.}\end{Problem}

\section{Modular forms with values in Banach algebras}\label{Modular-forms-revisited}

In this section, more technical than the previous ones, we suppose that $B$ is a Banach $\CC_\infty$-algebra with norm $\|\cdot\|$ such that
$\|B\|=|\CC_\infty|$ and we suppose that it is endowed with a countable orthonormal basis $\mathcal{B}=(b_i)_{i\in\mathcal{I}}$. The example on which we are focusing here is that of $B=\KK$, the completion of 
the field $\widehat{\CC_\infty(\underline{t})}$ for the Gauss valuation $\|\cdot\|$. Any basis 
of $\FF_q^{\operatorname{ac}}(\underline{t})$ as a vector space over $\FF_q^{\operatorname{ac}}$ is easily seen to be an orthonormal basis of $\KK$. We recall that we have considered, in \S \ref{Modular-forms-Tate-algebras}, a notion of $B$-valued analytic function.
The main purpose of this section is to show, through some examples, that if $N>1$, there is a generalisation 
$$\Omega\rightarrow \KK^{N\times 1}$$ of Drinfeld modular form which cannot by studied by using just 'scalar' Drinfeld modular forms. 

We consider a representation
$$\rho:\Gamma\rightarrow\GL_N(\FF_q(\underline{t}))\subset\GL_N(\KK).$$
\begin{Definition}\label{new-modular-forms}{\em 
Let $f:\Omega\rightarrow \KK^{N\times 1}$ be an analytic function. We say that $f$ is {\em modular-like} (for $\rho$) of weight $w\in\ZZ$ if for all $\gamma\in\GL_2(A)$,
$$f(\gamma(z))=J_\gamma(z)^w\rho(\gamma)f(z),\quad \gamma\in\GL_2(A).$$ We say that a modular-like function of weight $w$ is:
\begin{enumerate}
\item {\em weakly modular} (of weight $w$) if there exists $L\in\ZZ$ such that the map $z\mapsto \|\exp_A(z)^Lf(z)\|$ is bounded over $\Omega_M$ for some $M>1$,
\item a {\em modular form} if the map $z\mapsto \|f(z)\|$ is bounded over $\Omega_M$ for some $M>1$.
\item a {\em cusp form} if it is a modular form and $\max_{z\in\Omega_M}\|f(z)\|\rightarrow0$ as 
$M\rightarrow\infty$.
\end{enumerate} 
}\end{Definition}
We denote by $M^!_w(\rho),M_w(\rho),S_w(\rho)$ the $\KK$-vector spaces of weak modular, modular, and cusp forms of weight $w$ for $\rho$. Note that these notations are loose, in the sense that these vector spaces strongly depend of the choice of $\KK$ (in particular, of the variables $\underline{t}=(t_i)$). 

We now describe a very classical example with $N=1$ and $B=\CC_\infty$ (no variables $\underline{t}$ at all). If $\rho:\Gamma\rightarrow\CC_\infty^\times$ is a representation, there exists $m\in\ZZ/(q-1)\ZZ$ unique, such that $\rho(\gamma)=\det(\gamma)^{-m}$ for all $\gamma$. We write $$\rho=\det{}^{-m}$$ (note that this is well defined). Gekeler constructed  a
cusp form $h\in S_{q+1}(\det^{-1})\setminus\{0\}$; see
\cite[(5.9)]{GEK}. The first few terms of its $u$-expansion in $\CC_\infty$ can be computed explicitly by various methods (including the explicit formulas (\ref{formula-h-1}) and (\ref{formula-h-2}) below):
\begin{equation}\label{explicith}h(z)=-u(1+u^{(q-1)^2}+\cdots).\end{equation}
We deduce that $h^{q-1}\Delta^{-1}$ is a Drinfeld modular form of weight zero which is constant by Theorem \ref{structure-modular-forms}. The factor of proportionality is easily seen to be $-1$: $\Delta=-h^{q-1}$.

The computation in (\ref{explicith}) can be pushed to  coefficients of higher powers of the uniformiser $u$ by using two formulas that we describe here. The first formula is due to L\'opez \cite{LOP}. 
We have the convergent series expansion (in both $K[[u]]$ for the $u$-adic metric and 
in $D(0,r)$ for some $r\in|\CC_\infty|\cap]0,1[$ for the norm of the uniform convergence)
\begin{equation}\label{formula-h-1}
h=-\sum_{\begin{smallmatrix} a\in A\\ \text{monic}\end{smallmatrix}}a^qu_a\in A[[u]].
\end{equation}
The second formula is due to Gekeler \cite{GEK3} and is an analogue of Jacobi's product formula $$\Delta=q\prod_{n\geq 0}(1-q^n)^{24}\in q\ZZ[[q]]$$ for the classical complex-valued normalised discriminant cusp form $\Delta$ (we have an unfortunate and unavoidable conflict of notation here!). 
Gekeler's formula is the following $u$-convergent product expansion:
\begin{equation}\label{formula-h-2}
h=-u\prod_{\begin{smallmatrix} a\in A\\ \text{monic}\end{smallmatrix}}\left(u^{|a|}C_a\left(\frac{1}{u}\right)\right)^{q^2-1}\in A[[u]],\end{equation} with $C_a$ the multiplication by $a$ for the Carlitz module structure.  
Note that $(u^{|a|}C_a(\frac{1}{u}))^{q^2-1}\in 1+K[[u]]$ and the $u$-valuation of $$\left(u^{|a|}C_a(u^{-1})\right)^{q^2-1}-1$$ goes to infinity as $a$ runs in $A\setminus\{0\}$. 
One deduces, from Gekeler's result \cite[Theorem (5.13)]{GEK}, that $M_w(\det^{-m})=h^mM_{w-m_0(q+1)}$ if $m_0=m\cap\{0,\ldots,q-2\}$ ($m$ is a class modulo $q-1$).

 \subsection{Weak modular forms of weight $-1$} We analyse another class of representations, this time in higher dimension and we construct a new kind of modular form associated to it.
Let $$A\xrightarrow{\chi}\FF_q(\underline{t})^{n\times n}$$ be an injective $\FF_q$-algebra morphism. Then, the map $$\rho_\chi:\Gamma\rightarrow\GL_{2n}(\FF_q(\underline{t}))\subset\GL_{2n}(\KK)$$
defined by 
$$\rho_{\chi}\begin{pmatrix} a & b \\ c & d \end{pmatrix}=\begin{pmatrix} \chi(a) & \chi(b) \\ \chi(c) & \chi(d) \end{pmatrix}$$ is a representation of $\Gamma$. We denote by 
$\rho_\chi^*$ the contragredient representation $$\rho_\chi^*={}^t\rho_\chi^{-1}.$$
We shall study the case $\rho=\rho_\chi$ or $\rho_\chi^*$. We also set $N=2n$. 

We construct weak modular forms of weight $-1$ associated to the representations $\rho_\chi$; the main result is Theorem \ref{being-of-weight-1} where we show that a certain matrix function defined in (\ref{definition-FF}) has its columns which are weak modular forms of weight $-1$. We think that this construction is interesting because there seems to be no analogue of it in the settings of complex-vector-valued modular forms for $\operatorname{SL}_2(\ZZ)$.

Before going on, we need the next lemma, where we give a uniform bound for the valuations of the coefficients of the $u$-expansions $\sum_{m\geq 0}c_{i,m}u^m$ of the modular forms $\alpha_i$ appearing in (\ref{definition-alpha-i}).
\begin{Lemma}\label{lemmaB}
There exists a constant $C>0$ such that for all $i,m\geq 0$,
$$|c_{i,m}|\leq q^{-iq^i}|\widetilde{\pi}|^{q^i-1}C^m.$$
\end{Lemma}
\begin{proof}
This is \cite[Lemma 2.1]{PEL00}. Although the statement presented in this reference is  correct, there is a typographical problem in (2.17) so that, to avoid confusion, we give full details here. We set without loss of generality $|\theta|=q$. We recall (\cite[(2.14)]{PEL00}) that
$$\alpha_i=\frac{1}{\theta^{q^i}-\theta}(\widetilde{g}\alpha_{i-1}^q+\widetilde{\Delta}\alpha_{i-2}^{q^2}),\quad i>0,$$ with the initial values $\alpha_0=1$ and $\alpha_{-1}=0$. Now, writing additionally the $u$-expansions:
$$\widetilde{g}=\sum_{i\geq 0}\widetilde{\gamma}_iu^i,\quad \widetilde{\Delta}=\sum_{i\geq 0}\widetilde{\delta}_iu^i,$$
we find (as in ibid.)
$$c_{i,m}=\frac{1}{\theta^{q^i}-\theta}\left(\sum_{j+qk=m}\widetilde{\gamma}_jc_{i-1,k}^q+
\sum_{j'+q^2k'=m}\widetilde{\delta}_{j'}c_{i-2,k'}^{q^2}\right),\quad i>0,\quad m\geq 0$$ with the initial values $c_{i,0}=\frac{\widetilde{\pi}^{q^i-1}}{d_i}$ and $c_{-1,m}=0$.
Clearly, we can choose $C>0$ such that $|\widetilde{\delta}_j|\leq C^j$ and 
$|\widetilde{\gamma}_j|\leq C^j|\widetilde{\pi}^{q-1}|$
for all $j\geq 0$, and additionally, we can suppose that the inequality of the Lemma is true for $|c_{i,m}|$ with $i=0,1$. We now prove the inequality by induction over $i$. Indeed, note that
if $j+qk=m$, then, by induction hypothesis,
$|\widetilde{\gamma}_jc_{i-1,k}^q|\leq C^jq^{-(i-1)q^{i-1}q}C^{kq}|\widetilde{\pi}|^{q^i-q}|\widetilde{\pi}|^{q-1}\leq C^mq^{-(i-1)q^{i}}|\widetilde{\pi}|^{q^i-1}$ and similarly,
if $j+q^2k=m$, then we have $|\widetilde{\delta}_jc_{i-2,k}^{q^2}|\leq C^mq^{-(i-2)q^i}|\widetilde{\pi}|^{q^{i-2}-1},$ and the inequality follows.
\end{proof}
We write $\vartheta=\chi(\theta)$. If we set $$W=(\theta I_n-\vartheta)^{-1}\in\GL_n(\KK),$$ we have that for all $a\in A$:
\begin{equation}\label{chia}
(\chi(a)-a I_n)W\in\FF_q(\underline{t}_\Sigma)[\theta]^{n\times n}.
\end{equation}
Now, we consider, for $\chi$ and $W$ as in (\ref{chia}), the matrix function
$Q(z)=\binom{zW}{W}$, which is a holomorphic function $\Omega\rightarrow\KK^{N\times n}$. 
We observe that if $\gamma=(\begin{smallmatrix}a & b \\ c & d\end{smallmatrix})\in\Gamma$, 
then
$$Q(\gamma(z))=J_\gamma(z)^{-1}\binom{(az+b)W}{(cz+d)W}\equiv J_\gamma(z)^{-1}\rho_\chi(\gamma)Q(z)\pmod{\Lambda_z^{N\times n}}.$$
Hence, if we set
\begin{equation}\label{definition-FF}
\mathbb{F}(z):=\mathbb{E}(z,Q(z)),
\end{equation}
then, by the fact that $\Lambda_z\otimes\FF_q(\underline{t})$ is contained in the kernel of $\exp_{\Lambda_z}$,
$$\mathbb{F}(\gamma(z))=J_\gamma(z)^{-1}\mathbb{E}(z,J_\gamma(z)J_\gamma(z)^{-1}\rho_\chi(\gamma)Q(z))=J_\gamma(z)^{-1}\rho_\chi(\gamma)\mathbb{F}(z),\quad \forall\gamma\in\Gamma.$$
This means that the function $\mathbb{F}:\Omega\rightarrow\KK^{N\times n}$ is modular-like of weight $-1$ for $\rho_\chi$.
We are going to describe this function $\mathbb{F}$ in more detail.
\begin{Theorem}\label{being-of-weight-1}
We have $\mathbb{F}\in M^!_{-1}(\rho_\chi)^{1\times n}$.
\end{Theorem}
\begin{proof}
We set
$e_C(z)=\exp_C(\widetilde{\pi}z)$
so that $u(z)=\frac{1}{e_C(z)}$. Lemma \ref{omega-lambda-2} 
implies:
$$\tau(e_C(W))=(\vartheta-\theta I_n)e_C(W),\quad \tau(e_C(zW))=(\vartheta-\theta I_n)e_C(zW)+e_C(z).$$

The subset $\mathcal{W}\subset\RR_{>0}$
of the $r\in|\CC_\infty|$ such that the elements $|d_i^{-1}r^{q^i}|$ are all 
distinct for $i\geq 0$ is dense in $\RR_{>0}$. Let $z\in\CC_\infty$ be such that
$r=|\widetilde{\pi}z|\in\mathcal{W}$. Then:
$$|e_C(z)|=\max_{i}\{q^{-iq^i}|\widetilde{\pi}|^{q^i}|z|^{q^i}\}.$$
We write $\mathbb{F}=\binom{\mathcal{F}_1}{\mathcal{F}_2}$ with $\mathcal{F}_i:\Omega\rightarrow\KK^{n\times n}$.
We first look at the matrix function
$$\mathcal{F}_1=\exp_\Lambda(zW)=\sum_{i\geq 0}\alpha_i(z)z^{q^i}\tau^i(W).$$ We suppose that $|u(z)|<\frac{1}{B}$
with $B$ as in Lemma \ref{lemmaB}. Then $$\mathcal{F}_1=\sum_{i\geq 0}z^{q^i}\tau^i(W)\sum_{j\geq 0}c_{i,j}u^j$$ so that if $\|zW\widetilde{\pi}\|=r\in\mathcal{W}$ with $|u|<\frac{1}{B}$,
then 
\begin{eqnarray*}
\|\mathcal{F}_1\|&=&\max_{i,j}\{|z|^{q^i}q^{-iq^i}|\widetilde{\pi}|^{q^i-1}(\underbrace{C|u|}_{<1})^j\}\\
&=&\|\exp_C(\widetilde{\pi}zW)\|\\
&=&\|e_C(z/\theta)\|, 
\end{eqnarray*}
and $\frac{\mathcal{F}_1}{e_C(z/\theta)}-\widetilde{\pi}^{-1}I_n$ is bounded as $|z|_\Im$ is bounded from below.
 
 We now look at the matrix function $\mathcal{F}_2=e_\Lambda(W)$. Since $\mathcal{F}_2=\sum_{i\geq 0}\alpha_i(z)\tau^i(W)$, for  $|u|<\frac{1}{B}$ we get in a similar way that 
$\mathcal{F}_2-\widetilde{\pi}^{-1}e_C(W)$ goes to zero as $|z|_\Im\rightarrow\infty$. 
Hence, the $n$ columns of the matrix function $\mathbb{F}$, which are modular-like of weight $-1$ are weak modular forms of $M^!_{-1}(\rho_\chi)$.
\end{proof}

We set
$$\mathfrak{F}=(\mathbb{F},\tau(\mathbb{F}))=\begin{pmatrix}\mathcal{F}_1 & \tau(\mathcal{F}_1) \\
\mathcal{F}_2 & \tau(\mathcal{F}_2)\end{pmatrix}.$$ 
Then, $\mathfrak{F}$ is an analytic function $\Omega\rightarrow\KK^{N\times N}$ and the first $n$ columns are weak modular forms of weight $-1$, while the last $n$ columns are weak modular forms of weight $-q$ (for the representation $\rho_\chi$).
\begin{Lemma}\label{difference-equation}
We have the difference equation $\tau(\mathfrak{F})=\mathfrak{F}\Phi$
where 
$$\Phi=\begin{pmatrix} 0 & \widetilde{\Delta}^{-1}(\chi(\theta)-\theta I_n) \\ 1 & -\widetilde{\Delta}^{-1}\widetilde{g}I_n\end{pmatrix}.$$
\end{Lemma}

\begin{proof}
For any choice of $n,m>0$, we extend the function $\mathbb{E}(z,Z)$ of Lemma \ref{twotypesofrelation} to 
$$\Omega\times\KK^{n\times m}\xrightarrow{\mathbb{E}}\KK^{n\times m}$$
by setting $\mathbb{E}(z,Z)=\sum_{i\geq 0}\alpha_i(z)\tau^i(Z)$ (so $\tau$ acts diagonally). 
Lemma \ref{twotypesofrelation} holds in this generalised setting, where the Drinfeld modules $\phi_\Lambda$ now acts on $\KK^{n\times m}$ (case of $\Lambda=\Lambda_z$).
The present statement follows from (1) of Lemma \ref{twotypesofrelation} with $a=\theta$ in a manner which is 
sensibly similar to that of \cite[Theorem 1.3]{PEL00}. Indeed, note that, with $\phi_\Lambda(\theta)=\theta+\widetilde{g}\tau+\widetilde{\Delta}\tau^2$, we have 
$\phi_\Lambda(\theta)(\FF)-\chi(\theta)\FF=0$.
\end{proof}

\begin{Lemma}\label{maincongruenceF}
We have that
$\sup_{z\in\Omega_M}\|\mathfrak{F}-\mathcal{X}\mathcal{Y}\mathcal{Z}\|\rightarrow0$ as $M\rightarrow\infty,$ where
$$\mathcal{X}=\begin{pmatrix}I_n & 0 \\ 0 & e_C(W)\end{pmatrix},\quad \mathcal{Y}=\begin{pmatrix}e_C(zW) & \tau(e_C(zW)) \\ I_n & \vartheta-\theta I_n\end{pmatrix},\quad 
\mathcal{Z}=\begin{pmatrix}\widetilde{\pi}^{-1}I_n & 0 \\ 0 & \widetilde{\pi}^{-q}I_n\end{pmatrix}.$$
\end{Lemma}

\begin{proof}
We observe (recall that $\vartheta=\chi(\theta)$):
$$\mathcal{X}\mathcal{Y}\mathcal{Z}=\begin{pmatrix}\widetilde{\pi}^{-1}e_C(zW) & \widetilde{\pi}^{-q}((\vartheta-\theta I_n)e_C(zW)+e_0 I_n)\\ \widetilde{\pi}^{-1}e_C(W) & \widetilde{\pi}^{-q}(\vartheta-\theta I_n)e_C(W)\end{pmatrix}.$$
Since the second block column of $\mathfrak{F}$ is the image by $\tau$ of the first block column,
all we need to show is that $\sup_{z\in\Omega_M}\|\mathbb{F}-\binom{\widetilde{\pi}^{-1}e_C(zW)}{\widetilde{\pi}^{-1}e_C(W)}\|\rightarrow0$ as $M\rightarrow\infty$. We note that
$$\mathcal{F}_1=e_\Lambda(zW)=\widetilde{\pi}^{-1}e_C(zW)+\underbrace{\sum_{i\geq 0}z^{q^i}\tau^i(W)\sum_{j>0}c_{i,j}u^j}_{=:\Upsilon}.$$
We show that $\|\Upsilon\|$ tends to zero when $|z|_\Im\rightarrow\infty$. We suppose that 
$|z|_\Im$ is large so that $|u| C<1$. then, the double series defining $\Upsilon$ is convergent
and we can write 
$$\Upsilon=\sum_{j>0}\sum_{i\geq 0}u^jc_{i,j}z^{q^i}\tau^i(W).$$ The general term of this series,
$\Upsilon_{i,j}:=u^jc_{i,j}z^{q^i}\tau^i(W)$, has absolute value which satisfies:
\begin{eqnarray*}
\|\Upsilon_{i,j}\|&\leq & q^{-iq^i}|\widetilde{\pi}|^{q^i-1}(|u| C)^j|z|^{q^i}\|W\|^{q^i}\\
&\leq& |u| C\max_i\{|z|^{q^i}\|W\|^{q^i}|\widetilde{\pi}|^{q^i-1}\}\\
&\leq&|\widetilde{\pi}|^{-1}C\left|\frac{e_C(z/\theta)}{e_C(z)}\right|
\end{eqnarray*}
and tends to zero as $|z|_\Im\rightarrow\infty$. In a similar way, one proves that $\|\mathcal{F}_2-\widetilde{\pi}^{-1}e_C(W)\|$ tends to zero in the same way, we leave the details to the reader.
\end{proof}
\begin{Lemma} We have
$\|\det(\mathfrak{F})-(-1)^ne_C(z)^n\widetilde{\pi}^{-n(q+1)}\det(e_C(W))\|\rightarrow0$ as $|z|_{\Im}\rightarrow\infty$,
and $\det(e_C(W))$ is non-zero.
\end{Lemma}
\begin{proof}
The formula follows directly from the expression for $\mathcal{X}\mathcal{Y}\mathcal{Z}$.
The non-vanishing of $\det(e_C(W))$ is easy to show.
\end{proof}
 This result implies that the columns of $\mathfrak{F}$ are linearly independent. Moreover, it is plain that $\sup_{z\in\Omega_M}\|\det(\mathfrak{F}^{-1})-(-1)^n
u^n\widetilde{\pi}^{(q+1)n}\det(e_C(W))^{-1}\|\rightarrow0$ as $M\rightarrow\infty$. Since at once
the scalar function $F=\det(\mathfrak{F}^{-1})$ satisfies $F(\gamma(z))=J_\gamma(z)^{n(q+1)}\det(\gamma)^{-n}F(z)$ for all $z\in\Omega$ and $\gamma\in\Gamma$, we get 
$F\in M_{n(q+1)}(\det^{-n})\otimes_{\CC_\infty}\KK$. Now, $Fh^{-n}$ is a modular form of weight $0$, therefore equal to an element of $\KK^\times$. We obtain:
\begin{Corollary}\label{coro617} We have
$\det(\mathfrak{F}^{-1})=(-1)^n\widetilde{\pi}^{-(q+1)n}h^n\det(e_C(W))^{-1}$ and, writing $\mathfrak{H}:={}^t\mathfrak{F}^{-1}=(\mathfrak{H}_1,\mathfrak{H}_2)$ with $\mathfrak{H}_i:\Omega\rightarrow\KK^{n\times n}$, we have that
the $n$ columns of $\mathfrak{H}_1$ are linearly independent modular forms of weight $1$ and the $n$ columns of $\mathfrak{H}_2$ are linearly independent modular forms of weight $q$ for the representation $\rho^*_\chi$.
\end{Corollary}
What can be further proved is, by setting $$M(\rho_\chi^*)=\bigoplus_wM_w(\rho_\chi^*)$$
the weight-graded $(M\otimes_{\CC_\infty}\KK)$-module of modular forms for $\rho_\chi^*$, where $M=\bigoplus_wM_w(\boldsymbol{1})$ is the $\CC_\infty$-algebra of scalar modular forms ($\boldsymbol{1}$ is the trivial representation):
\begin{Theorem}
$M(\rho_\chi^*)=(M\otimes_{\CC_\infty}\KK)^{1\times N}\mathfrak{H}.$
\end{Theorem}
We will not give the details of the deduction of the proof of this theorem from Corollary \ref{coro617}, since it rests on an easy generalisation and modification of \cite[Theorem 3.9]{PEL&PER3}. Instead of this, we insist on the result of Gekeler \cite[Theorem (5.13)]{GEK}, which implies that $$M_w(\det{}^{-m})=M_{w-m(q+1)}h^m,\quad m\leq q-1$$
with $h$ the Poincar\'e series of weight $q+1$ and 'type $1$' defined in ibid. (5.11) (with $u$-expansion (\ref{explicith})) so that,
with $M(\det^{-m})=\oplus_wM_w(\det{}^{-m})$, $$M(\det{}^{-m})=Mh^m.$$ In view of this, we can think about $\mathfrak{H}$ (up to normalisation) as to a matrix-valued generalisation of the Poincar\'e series $h$.

\subsection{Jacobi-like forms}
 We consider the series
$$\Psi(z,Z):=\psi_{\Lambda_z}(Z)=\sum_{a,b\in A}\frac{1}{Z-az-b}(\chi(a),\chi(b)),$$ converging for 
$Z\in\CC_\infty\setminus\Lambda$ where $\Lambda=\Lambda_z=Az+A$, $z\in\Omega$.
We have the following functional identities
$$\Psi\left(\gamma(z),\frac{Z}{J_\gamma(z)}\right)=J_\gamma(z)\Psi(z,Z)\rho(\gamma)^{-1},\quad \gamma\in\Gamma,$$ together with the identities arising from (\ref{simple-translation}). 
Proposition \ref{series-expansion} implies that, for $Z\in D(0,r)$ for some $r\in|\CC_\infty|\cap]0,1[$,
$${}^t\Psi(z,Z)=-\sum_{\begin{smallmatrix}j> 0\\ j\equiv1(q-1)\end{smallmatrix}}Z^{j-1}\mathcal{E}(j;\chi)$$
where $\mathcal{E}(j;\chi)$ is the Eisenstein series (non-vanishing if $j\equiv1\pmod{q-1}$) $$\mathcal{E}(j;\chi):=\sideset{}{'}\sum_{a,b\in A}\frac{1}{(az+b)^j}\binom{\chi(a)}{\chi(b)},$$ which satisfies
$$\mathcal{E}(j;\chi)(\gamma(z))=J_\gamma(z)^j\rho_\chi^*(\gamma)\mathcal{E}(j;\chi),\quad \gamma\in\Gamma,\quad z\in\Omega.$$ Since 
it is also apparent that $\|\mathcal{E}(j;\chi)(z)\|$ is bounded on $\Omega_M$ for $M>1$ and $j>0$,
we deduce that the $n$ columns of $\mathcal{E}(j;\chi)$ are modular forms of weight $j$ for $\rho_\chi^*$ in the sense of Definition \ref{new-modular-forms} (see \cite[\S 3.2.1]{PEL&PER3} for a special case). By Theorem \ref{identity-eisenstein-series} we obtain 
\begin{equation}\label{psi-lambda}
\Psi(z,Z)=[\mathcal{H}(Z),\mathcal{D}_{q-1}(\mathcal{H})(Z)^{(-1)}]\boldsymbol{\Omega}_\Lambda(z)^{-1}
\end{equation} which allows to explicitly compute the Eisenstein series $\mathcal{E}(j;\chi)$ in terms of the function $\mathcal{H}(Z)$. To make this interesting relation a little bit more transparent, we 
give below an explicit expression of the matrix $\boldsymbol{\Omega}_\Lambda(z)^{-1}$. We have:
\begin{equation}\label{explicit-identity}
\boldsymbol{\Omega}_\Lambda(z)^{-1}=\begin{pmatrix}0 & 1\\ 1 & 0\end{pmatrix}
\tau^{-1}(\Phi)\mathfrak{F}^{-1}=\begin{pmatrix}1 & -\left(\frac{\widetilde{g}}{\widetilde{\Delta}}\right)^{\frac{1}{q}}\\ 0 & (\chi(\theta)-\theta^{\frac{1}{q}})\widetilde{\Delta}^{-\frac{1}{q}}\end{pmatrix}\mathfrak{F}^{-1},
\end{equation}
with $\Phi$ the matrix defined in Lemma \ref{difference-equation}. 
To see this, observe that in the notation of Theorem \ref{identity-eisenstein-series}, 
$$\boldsymbol{\Omega}_\Lambda(z)=(\FF,\tau^{-1}(\FF))=\tau^{-1}(\mathfrak{F})\begin{pmatrix}0 & 1\\ 1 & 0\end{pmatrix},$$ with $\Lambda=\Lambda_z$ as above.
By Lemma \ref{difference-equation}, $\tau(\mathfrak{F})=\mathfrak{F}\Phi$ , so that $\tau^{-1}(\mathfrak{F})=\mathfrak{F}(\tau^{-1}(\Phi))^{-1}$ which yields
$$\boldsymbol{\Omega}_\Lambda=\tau^{-1}(\mathfrak{F})\begin{pmatrix}0 & 1\\ 1 & 0\end{pmatrix}=
\mathfrak{F}\begin{pmatrix}0 & 1\\ 1 & 0\end{pmatrix}\begin{pmatrix}0 & 1\\ 1 & 0\end{pmatrix}(\tau^{-1}(\Phi))^{-1}\begin{pmatrix}0 & 1\\ 1 & 0\end{pmatrix},$$ which implies (\ref{explicit-identity}) by the (licit) inversion of the two sides.

Substituting in (\ref{psi-lambda}) and transposing, we get:
$$-\sum_{\begin{smallmatrix}j\geq 1\\ j\equiv1(q-1)\end{smallmatrix}}\mathcal{E}(j;\chi)Z^{j-1}=\mathfrak{H}\begin{pmatrix}1 & 0\\ -\left(\frac{\widetilde{g}}{\widetilde{\Delta}}\right)^{\frac{1}{q}} & \widetilde{\Delta}^{-\frac{1}{q}}({}^t\chi(\theta)-\theta^{\frac{1}{q}})\end{pmatrix}{}
\begin{pmatrix}
{}^t\mathcal{H}(Z)\\ \mathcal{D}_{q-1}({}^t\mathcal{H})(Z)^{(-1)}\end{pmatrix}.$$
For example, the Eisenstein series of weight one $\mathcal{E}(1;\chi)$ arises as the coefficient of $Z^0$ in the left-hand side and the above yields an explicit formula for it. Note that
the constant term of the $Z$-expansion of ${}^t[\mathcal{H}(Z),\mathcal{D}_{q-1}(\mathcal{H})(Z)^{(-1)}]$ is
$${}^t[(\theta I_n-\chi(\theta))^{-1},\alpha_1(z)^{\frac{1}{q}}((\theta I_n-\chi(\theta))^{-1}-(\theta^{\frac{1}{q}} I_n-\chi(\theta))^{-1})].$$
The formula that we get is this one:
$$
-\mathcal{E}_1(1;\chi)=\mathfrak{H}\begin{pmatrix}1 & 0\\ -\left(\frac{\widetilde{g}}{\widetilde{\Delta}}\right)^{\frac{1}{q}} & \widetilde{\Delta}^{-\frac{1}{q}}({}^t\chi(\theta)-\theta^{\frac{1}{q}})\end{pmatrix}\begin{pmatrix}{}^t(\theta I_n-\chi(\theta))^{-1}\\\alpha_1(z)^{\frac{1}{q}}{}^t((\theta I_n-\chi(\theta))^{-1}-(\theta^{\frac{1}{q}} I_n-\chi(\theta))^{-1})\end{pmatrix},$$ and what looks as a miracle at first sight is that 
it greatly simplifies, by using the explicit computation of $\alpha_1$ which arises from \cite[(2.14)]{PEL00}, and which is $\alpha_1=\frac{\widetilde{g}}{\theta^q-\theta}$, we reach the following:
\begin{Theorem}\label{identity-eisenstein-1}
The following identity holds $$\mathcal{E}_1(1;\chi)=-\mathfrak{H}\binom{{}^t(\theta I_n-\chi(\theta))^{-1}}{0_n},$$
involving $N\times n$ matrices whose columns are modular forms of weight $1$.
\end{Theorem}
In fact, this is not a miracle; it is just due to the fact that the left-hand side must be bounded at the infinity; this is only possible if the second matrix entry of the column above is identically zero, because it is anyway a multiple by a constant matrix of the weak
modular form $\widetilde{g}/\widetilde{\Delta}$ (this somewhat forces $\alpha_1$ to be equal to the above multiple of $\widetilde{g}$, giving this artificial impression of miraculous simplification). 
It is easy from here to deduce \cite[Theorem 8]{PEL1} in the special case of $N=2$, $n=1$ and $\chi=\chi_t$.


\begin{thebibliography}{99}

\bibitem{AND} G. Anderson. {\em $t$-motives.} Duke Math. J. 53 (1986), 457--502.

\bibitem{ABP} G. W. Anderson, W. D. Brownawell \& M. A. Papanikolas. {\em Determination of the algebraic relations among special $\Gamma$-values in positive characteristic.} Ann. of Math. 160 (2004) 237--313.



\bibitem{ANG&PEL2} B. Angles, F. Pellarin. {\em Universal Gauss-Thakur sums and $L$-series.} Invent. Math. 200 (2), (2015) 653--669.

  
\bibitem{ANG&TAV} B. Angl\`es, F. Tavares Ribeiro. {\em Arithmetic of function fields units.} Math. Annalen, 367, (2017) 501--579.

\bibitem{ART&WAP} E. Artin \& G. Whaples. {\em Axiomatic characterization of fields by the product formula for valuations.} Bull. Amer. Math. Soc. 51, (1945), 469--492.

\bibitem{AX} J. Ax. {\em Zeros of Polynomials over Local Fields.} Journal of Algebra, 15, (1970), 417--428.



\bibitem{CON} M. Baker, B. Conrad, S. Dasgupta, K. S. Kedlaya \& J. Teitelbaum. {\em $p$-adic Geometry: Lectures from the 2007 Arizona Winter School
p-adic Geometry Lectures.} AMS University Lecture Series, 45, (2008).

\bibitem{BAR} W. Bartenwerfer. {\em Der erste Riemannsche Hebbarkeitssatz im nichtarchimedischen Fall.} J. Reine Angew. Math. 286/87 (1976), 144--163.

\bibitem{BBP1} D. Basson, F. Breuer \& R. Pink. {\em Drinfeld modular forms of arbitrary rank, Part I: Analytic Theory.} Preprint. {\tt arXiv:1805.12335}

\bibitem{BBP2} D. Basson, F. Breuer \& R. Pink. {\em Drinfeld modular forms of arbitrary rank, Part II: Comparison with Algebraic Theory.} Preprint. {\tt arXiv:1805.12337}

\bibitem{BBP3} D. Basson, F. Breuer \& R. Pink. {\em Drinfeld modular forms of arbitrary rank, Part III: Examples.} Preprint. {\tt arXiv:1805.12339}

\bibitem{BER} V. Berkovich. {\em Spectral theory and analytic geometry over non-archimedean fields.} Mathematical Surveys and Monographs No. 33 (1990).


\bibitem{BGR} S. Bosch, U. G\"{u}ntzer \& R. Remmert.  
{Non-Archimedean analysis.}
Grundlehren der Mathematischen Wissenschaften, 261,
Springer-Verlag, Berlin, (1984).

  
  \bibitem{BOE1} G. B\"ockle. {\em An Eichler-Shimura isomorphism over function fields between Drinfeld modular forms and cohomology classes of crystals.} Manuscript, currently available on the webpage of the author.
  
 \bibitem{BOE2} G. B\"ockle. {\em Hecke characters associated to Drinfeld modular forms.} With an appendix by the author and T. Centeleghe.
Compositio Math. 151 (2015), 2006--2058.  

  \bibitem{BOE&HAR} G. B\"ockle \& U. Hartl. {\em Uniformizable Families of $t$-motives.} Trans. Amer. Math. Soc. 359 (2007), 3933--3972.





\bibitem{CAR} L. Carlitz. {\em On certain functions connected with polynomials in a Galois field.}
Duke Math. J. 1, (1935), 137--168.

\bibitem{CAS} J. W. S. Cassels. {\em Local fields.} Cambridge University Press, (1986).

\bibitem{VIZ} L. Di Vizio. {\em Difference Galois theory for the `applied' mathematician.} This Volume.

\bibitem{DRI} V. Drinfeld. {\em Elliptic modules.} Math. Sb. 94 (1974), 594--627 (in Russian). English translation in Math. USSR Sb. (1974).


\bibitem{FP} J. Fresnel, \& M. van der Put. {\em Rigid Analytic Geometry and its Applications.} Birkh\"auser, Boston (2004).

\bibitem{GAZ&MAU} Q. Gazda \& A. Maurischat. {\em Special Functions and Gauss-Thakur Sums in Higher Rank and Dimension.} Preprint 2019. {\tt arXiv:1903.07302}.

\bibitem{GEK}  E.-U. Gekeler. {\em On the coefficients of Drinfeld modular forms.}
Invent. Math. 93, (1988), 667--700.

\bibitem{GEK3} E.-U. Gekeler. {\em A product expansion for the discriminant function of Drinfeld modules of rank
two.} J. Number Theory 21 (1985), 135--140.


\bibitem{GEK2} E.-U. Gekeler. {\em Finite modular forms} Finite Fields and Their Applications 7, (2001), 553--572.

\bibitem{GEK-DMF1} E. U. Gekeler. {\em On Drinfeld modular forms of higher rank.}
Journal de Th\'eorie des Nombres de Bordeaux Vol. 29 (2017), pp. 875--902.

\bibitem{GEK-DMF2} E. U. Gekeler. {\em On Drinfeld modular forms of higher rank II.} J. Number Theory, (2019).

\bibitem{GEK-DMF3} E. U. Gekeler. {\em On Drinfeld modular forms of higher rank III: The analogue of the $k/12$-formula.}
J. Number Theory Vol. 192, (2018), 293--306.

\bibitem{GEK-DMF4} E. U. Gekeler. {\em On Drinfeld modular forms of higher rank IV: Modular forms with level.} J. Number Theory, (2019).

\bibitem{GER&PUT} L. Gerritzen \& M. van der Put. {\em Schottky Groups and Mumford Curves.} Springer Lectures Notes in Mathematics 817. Springer Berlin Heidelberg, (1980).

\bibitem{GOS1} D. Goss. {\em The algebraist's upper half-plane.} Bulletin of the Amer. Math. J. 2, (1980), 391--415.

\bibitem{GOS2} D. Goss. {\em $\pi$-adic Eisenstein series for Function Fields.} Compositio Math. 41, (1980), 3--38.

\bibitem{GOS3} D. Goss. {\em Modular forms for $\FF_r[T]$,} J. Reine Angew. Math. 317, (1980), 16--39.

\bibitem{GOS} D. Goss. {\em Basic Structures of Function Field Arithmetic.} Springer Verlag, Berlin, (1996).

\bibitem{HAB} S. H\"aberli. {\em Satake compactification of analytic Drinfeld modular varieties.}
ETH thesis dissertation 25544. Z\"urich, 2018.

\bibitem{HAN} D. Hansen. {\em Quotients of adic spaces by finite groups.} Math. Res. Letters. To appear.

\bibitem{HAR&JUS} U. Hartl \& A.-K. Juschka. {\em Pink's theory of Hodge structures and the Hodge conjecture over function fields}. In {\em Hodge structures, transcendence and other motivic aspects.} Editors G. B\"ockle, D. Goss, U. Hartl, M. Papanikolas, Series of Congress Reports, European Mathematical Society 2020.

\bibitem{HAY} D. R. Hayes. {\em Explicit class field theory for rational function fields.} Trans. AMS. 189, (1974), 77--91. 

\bibitem{HAY2} D. R. Hayes. {\em Explicit class field theory in global function fields.} in: Studies in Algebra and Number Theory, in: Adv. in Math. Suppl. Stud., 6, Academic Press, New York, (1979), pp. 173--217.

\bibitem{KAT} K. Kato. {\em Lectures on the approach to Iwasawa theory for
Hasse-Weil $L$-functions via $B_{dR}$.} In {\em Arithmetic Algebraic Geometry.} Springer Lectures Notes in Mathematics 1553. 
CIME, Trento, 1991.

\bibitem{KED} K. Kedlaya. {\em The algebraic closure of the power series field in positive characteristic.} Proc.
Amer. Math. Soc. 129 (2001), 3461--3470.


\bibitem{LOP} B. L\'opez. {\em A non-standard Fourier expansion for the Drinfeld discriminant function}. Arch.
Math. 95 (2010) 143--150.

\bibitem{LUB&TAT} J. Lubin \& J. Tate. {\em Formal complex multiplication in local fields.} Ann. Math. 81 (1965), 380--387.


\bibitem{MAS&SCH} A. W. Mason \& A. Schweizer. {\em Elliptic points of the Drinfeld modular groups.} Math. Z. 279, (2015), 1007--1028.

\bibitem{NAG} H. Nagao. {\em On $\GL(2, K[x])$}. J. Inst. Polytechn., Osaka City Univ., Ser. A. 10, (1959), 117--121.

\bibitem{PAP0} M. A. Papanikolas. {\em Tannakian duality for Anderson-Drinfeld
motives and algebraic independence of Carlitz logarithms},
Invent. Math. 171, 123-174 (2008).


\bibitem{PEL&BOU} F. Pellarin. {\em Aspects de l'ind\'ependance alg\'ebrique en caract\'eristique non nulle.} Bourbaki seminar. Volume 2006/2007. Expos\'es 967--981. Paris: Soci\'et\'e Math\'ematique de France. Ast\'erisque 317, 205-242 (2008).

\bibitem{PEL1} F. Pellarin. {\em Values of certain $L$-series in positive characteristic.}
Ann. of Math. 176, (2012), 2055--2093


\bibitem{PEL00} F. Pellarin. {\em Estimating the order of vanishing at infinity of Drinfeld quasi-modular forms}, J. Reine Angew. Math. 687, (2014), 1--42.




\bibitem{PEL5} F. Pellarin. {\em A note on multiple zeta values in Tate algebras.} Riv. Mat. Univ. Parma, 7, (2016), 71--100.

\bibitem{PEL&PER} F. Pellarin, \& R. B. Perkins. {\em On certain generating functions in positive characteristic.} Monat. Math. 180, (2016), 123--144. 


\bibitem{PEL&PER3} F. Pellarin, \& R. B. Perkins. {\em On vectorial Drinfeld modular forms over Tate algebras.} Int. J. of Number Theory, 14, (2018), 1729--1783. 

\bibitem{PER0} R. B. Perkins. {\em On Special Values of Pellarin's $L$-Series.} Ph.D. Dissertation. The Ohio State University (2013).

\bibitem{PER} R. B. Perkins. {\em Explicit formulae for $L$-values in positive characteristic.}
Math. Z. 248, (2014), 279--299.

\bibitem{POI&TUR} J. Poineau \& D. Turchetti. {\em Berkovich curves and Schottky uniformisation.} This volume.


\bibitem{SER} J.-P. Serre. {\em Local fields.} Berlin, New York: Springer-Verlag, (1980).

\bibitem{STI} H. Stichtenoth. {\em Algebraic function fields and codes.} Graduate Texts in Mathematics 254, Springer Verlag (2008).



 
\bibitem{TAT} J. Tate. {\em Rigid analytic spaces.} Invent. Math. 12, (1971) 257--289.

\bibitem{TAV} F. Tavares Ribeiro. {\em On the Stark units of Drinfeld modules.} This volume.

\bibitem{TEI} J. Teitelbaum. {\em The Poisson kernel for Drinfeld modular curves.} J. of the AMS, 4, (1991), 491--511.

\bibitem{TEM} M. Temkin. {\em Introduction to Berkovich Analytic Spaces.} In {\em Berkovich Spaces and Applications.}
A. Ducros, Ch. Favre \& J. Nicaise Editors, Springer Lecture Notes in Mathematics 2119, (2015).
 


\bibitem{WAD} L. Wade. {\em Certain quantities transcendental over $GF(p^n,x)$.} Duke Math. J. 8 (1941), 707--720.

\bibitem{ZYW} D. Zywina. {\em Explicit class field theory for global function fields.} J. Number Theory, 133, (2013), 1062--1078.

\end{thebibliography}
\end{document}